\documentclass[8pt]{article}

\usepackage{amsmath,amsfonts,amssymb}
\usepackage{bbm}

\usepackage{lipsum}

\usepackage{xcolor}
\usepackage{ulem}

\newtheorem{defn}{Definition}[section]
\newtheorem{theorem}[defn]{Theorem}
\newtheorem{lemma}[defn]{Lemma}
\newtheorem{proposition}[defn]{Proposition}
\newtheorem{cor}[defn]{Corollary}
\newtheorem{remark}[defn]{Remark}
\newtheorem{example}[defn]{Example}

\newenvironment{proof}{{\bf Proof }}{{\vskip 0.1cm \hfill$\Box$}}

\def\N {{\mathbb N}}

\def\R {{\mathbb R}}
\def\E{{\mathbb E}}
\def\P{{\mathbb P}}
\def\M{{\mathbb M}}
\newcommand{\F}{\mathcal{F}}

\begin{document}


\noindent
{\Large \bf Uniqueness in law for singular degenerate SDEs with respect to a (sub-)invariant measure}
\\ \\
\bigskip
\noindent
{\bf Haesung Lee, Gerald Trutnau{\footnote{The research of Gerald Trutnau was supported by the National Research Foundation of Korea (NRF) grant funded by the Korea government (MSIP) : NRF-2016K2A9A2A13003815.}}}  \\
\noindent
{\bf Abstract.} We show weak existence and uniqueness in law for a  general class of stochastic differential equations in $\mathbb{R}^d$, $d\ge 1$, with prescribed sub-invariant measure $\widehat{\mu}$. The dispersion and drift coefficients of the stochastic differential  equation are allowed to be degenerate and discontinuous, and locally unbounded, respectively. 
Uniqueness in law is obtained  via $L^1(\R^d,\widehat{\mu})$-uniqueness in a subclass of continuous Markov processes, namely right processes that have $\widehat{\mu}$ as sub-invariant measure and have continuous paths for $\widehat{\mu}$-almost every starting point. Weak existence is obtained for a broader class via the martingale problem, by first constructing a sub-Markovian $C_0$-semigroup of contractions with respect to $\widehat{\mu}$ and then applying generalized Dirichlet form theory.


\section{Introduction} \label{intro}
The main goal of this work is to provide a detailed analysis for a class of highly singular, possibly degenerate stochastic differential equations (hereafter SDE) on $\R^d$ with respect to a prescribed sub-invariant measure $\widehat{\mu}$. The analysis comprises weak existence via the martingale problem, $L^1(\R^d,\widehat{\mu})$-uniqueness for the corresponding generator, 
non-explosion results, in particular conditions for  $\widehat{\mu}$ to be an invariant measure, conditional uniqueness in law, and properties of the time reversal with respect to $\widehat{\mu}$. Our approach is functional analytic together with probabilistic techniques. In a subsequent work, building on the results developed here, we use additionally elliptic regularity results for PDEs and consider the same equations with some further regularity assumptions on the coefficients to provide a pointwise analysis for every starting point $x\in \R^d$.\\
Below, we first describe in detail our work. Then, we give some motivations and relate our work to existing literature. Finally, we explain the organization of this paper.\\
{\bf Weak existence, methods and implications:} 
Suppose that the set of assumptions {\bf (A)} of Section \ref{framework} on $\rho, \psi, A, \mathbf{B}$ holds unless otherwise stated and consider the SDE
\begin{equation}\label{weaksolutionintro}
X_t = X_0+  \int_0^t \widehat{\sigma}(X_s) \, dW_s+   \int^{t}_{0}   \mathbf{G}(X_s) \, ds, \quad 0\le  t <\infty,
\end{equation}
where $W$ is a $d$-dimensional standard Brownian motion, $\widehat{\sigma}= \sqrt{\frac{1}{\psi}}\cdot \sigma$, $\sigma=(\sigma_{ij})_{1 \le i,j \le d}$ is any matrix (possibly non-symmetric) consisting of locally bounded and measurable functions such that $\sigma \sigma^T(x) =A(x)$ for a.e. $x\in \R^d$, $\mathbf{G}=(g_1, \ldots, g_d)=\beta^{\rho, A,\psi} + \mathbf{B}$, and $\beta^{\rho, A,\psi}$ is the so-called logarithmic derivative of $\rho$, associated to $A$ and $\psi$ (cf. \eqref{logder 1}). Here the degeneracy of \eqref{weaksolutionintro} stems from the factor $\sqrt{\frac{1}{\psi}}$, which may have zeros and we note that the class of coefficients $\widehat{\sigma}$ and $\mathbf{G}$ for which we obtain weak existence to \eqref{weaksolutionintro}, at least for a.e. starting point $X_0=x\in \R^d$, is fairly large (see Proposition \ref{weakexfairlygeneral}). Now, consider the infinitesimal generator $L$ corresponding  to \eqref{weaksolutionintro} given by \eqref{definition of L}.   A main feature that we use for our analysis is the infinitesimal invariance of $\widehat{\mu}=\psi\rho dx$ for $(L,C_0^{\infty}(\R^d))$, i.e.
\begin{equation}\label{infinit invariance}
\int_{\R^d} Lf\, d\widehat{\mu}=0\qquad \text{for all }f\in C_0^{\infty}(\R^d).
\end{equation}
Following the main ideas of \cite{St99}, the infinitesimal invariance allows us to obtain some closed extension of $L$ on $L^1(\R^d,\widehat{\mu})$ that generates a $C_0$-semigroup of contractions $(\overline{T}_t)_{t>0}$ (also denoted by $(T_t)_{t>0}$ according to the space it acts on) by functional analytic means (cf. Theorem \ref{mainijcie}, which extends \cite[Theorem 1.5]{St99} for $\psi\not\equiv 1$). Note here, that we cannot apply \cite[Theorem 1.5]{St99}  directly, since $\psi$ may not be differentiable. Hereby the generalized Dirichlet form \eqref{fwpeokce} is constructed, which is shown to have some regularity properties (cf. proof of Proposition \ref{Huntex}) in order to obtain a weak solution to \eqref{weaksolutionintro} for a.e. starting point $X_0=x\in \R^d$ in Theorem \ref{prop:3.1.6}. Here of course, in order to ensure that the solution exists for all times, we need to impose a non-explosion condition (i.e. any condition that implies $T_t1=1$, $t>0$, for instance the explicit condition \eqref{growthconser 2}). The infinitesimal invariance \eqref{infinit invariance} implies further that $\widehat{\mu}$ is a sub-invariant measure for $(\overline{T}_t)_{t>0}$, hence also a sub-invariant measure for the weak solution $X$ (see \eqref{subinvprocdef}, and \cite[Remark 3.44]{LST22}). It is then easy to present explicit conditions for $\widehat{\mu}$ to become an invariant measure (for instance \eqref{growthconser 3}), namely invariance is equivalent to the conservativeness of the dual semigroup (cf. Theorem \ref{theoconserva}(i)). On the other hand, invariance is equivalent to $L^1$-uniqueness (see Theorem \ref{analycharin}). Another interesting feature provided in our approach is the Lyons-Zheng decomposition \cite{Tr9}, which holds if both semigroups, the original and the dual one, are conservative (see again \eqref{growthconser 2}, \eqref{growthconser 3} for explicit conditions that imply conservativeness of both semigroups). The Lyons-Zheng decomposition has several applications (see for instance \cite{Ta89}, \cite{TaTr}). In  particular, on any interval $[0,T]$, the time-reversed process and the co-process are identical in law under $\widehat{\mu}$ in this framework (cf. \cite[Lemma 2.1]{Tr9}).\\
{\bf Conditional uniqueness in law: }In order to explain our uniqueness in law result, let us suppose that additionally to the set of assumptions {\bf (A)}, all entries of the diffusion matrix $A=(a_{ij})_{1 \leq i,j \leq d}$ consist of locally H\"{o}lder continuous functions on $\mathbb{R}^d$ and that $\widehat{\mu}$ is a (possibly infinite) invariant measure for $(\overline{T}_t)_{t>0}$. Then $(L, C_0^{\infty}(\R^d))$ is $L^1(\mathbb{R}^d, \widehat{\mu})$-unique (cf. Theorem  \ref{analycharin} and Corollary \ref{cor4.4}, which are our main theorems about $L^1(\mathbb{R}^d, \widehat{\mu})$-uniqueness). Let $\widetilde{\M}$ be any right process (which is basically a right continuous strong Markov process with some lifetime) with sub-invariant measure $\widehat{\mu}$ which solves \eqref{weaksolutionintro} with respect to $\widehat{\mu}$ (see Definition \ref{defnofmartin}(iii)). Then $T_t1=1$, $t>0$, and for any $v \in L^1(\R^d,\widehat{\mu})_b$, $v\ge 0$, such that $\int_{\mathbb{R}^d} v d\widehat{\mu} = 1$, the law of $\widetilde{\M}$ with initial distribution $v\widehat{\mu}$ coincides with the law of the weak solution with initial distribution $v\widehat{\mu}$ constructed by generalized Dirichlet form theory (see Theorem \ref{uniqueness in  law}, which is our main general result about conditional uniqueness in law). Thus if $T_t1=1$, $t>0$, and the additional assumptions to {\bf (A)} hold as above, then we obtain a right process $\mathbb{M}$ that solves (1) with respect to $\widehat{\mu}$ and uniqueness in law for \eqref{weaksolutionintro} with respect to $\widehat{\mu}$, which is actually equivalent to uniqueness in law for a.e. starting point (see Definition \ref{defnofmartin}(iv)), holds among all right processes which admit $\widehat{\mu}$ as sub-invariant measure. Our main explicit result is then Theorem \ref{weakexfairlygeneral 2}, where explicit conditions for conditional uniqueness in law are given. We note that the conditional uniqueness in law as stated above is not trivial since $\sqrt{\frac{1}{\psi}}$, hence $\sqrt{\frac{1}{\psi}}\cdot \sigma_{ij}$, $1\le i,j\le d$,  can be discontinuous and even have zeros. In fact, already in dimension one, there exist elementary SDEs with continuous degenerate dispersion coefficient (satisfying our assumptions) for which uniqueness in law fails for every starting point of the state space (cf. e.g. \cite{PavShe20} and also references therein) and therefore uniqueness in law, even for only a.e. starting point $x\in \R^d$ and $d=1$, may fail to hold for \eqref{weaksolutionintro}. One is then naturally led to put conditions on the admissible solutions. For instance one can search for solutions in the class of strong Markov processes, and/or processes that spent zero time at the points of degeneracy, etc. (see again e.g. \cite{PavShe20}). Here, we are looking for uniqueness in the class of right processes with prescribed sub-invariant measure. Such processes automatically spend for a.e. starting point zero time at sets of Lebesgue measure zero (cf. \eqref{eq:subinv 1}). By this we provide tractability to highly singular SDEs, where classical techniques to obtain weak existence and uniqueness, for instance the Girsanov transformation, may not be at hand as the drift may merely be in $L^2_{loc}(\R^d,\R^d,dx)$.\\
{\bf Some motivations: }The motivation to consider a partial differential operator $L$ with suitable coefficients by the aid of a given (sub-)invariant measure stems from applications of SDEs, for instance to the sampling of probability distributions or more generally to ergodic control problems (see e.g. \cite{Borkar}, \cite{HHS05}, \cite{IKLT22}). SDEs with prescribed probability invariant measure have been used as a tool for stochastic gradient Markov chain Monte Carlo (MCMC) samplers in \cite{MCF15}. In particular, a main observation of  \cite{HHS05}, and also \cite{MCF15}, is faster convergence to equilibrium $\widehat{\mu}$ when $\mathbf{B}\not \equiv 0$ in \eqref{weaksolutionintro} for certain subclasses of $\mathbf{B}$ and coefficients of $L$. Note here that the equilibrium remains the same if $\mathbf{B}$ varies, since $\mathbf{B}$ is supposed to be divergence free with respect to $\widehat{\mu}$. Right processes with given (sub)-invariant (also called supermedian)  measure are extensively studied in \cite{BB04}. Conditions for convergence toward equilibrium, holding particularly also in the non-reversible case (i.e. when $\mathbf{B}\not \equiv 0$), were studied in \cite{BCR,LST22}.\\
{\bf Related work: }Finally, we discuss work related to the (functional) analytic uniqueness, i.e. the uniqueness of semigroups or bilinear forms, and probabilistic uniqueness, i.e. uniqueness of of laws, considered here. In 
\cite{AR93}, the martingale problem is considered in  the frame of $m$-special standard processes related to quasi-regular semi-Dirichlet forms and different kinds of Markov uniqueness.  In  \cite[Chaper 1 a)]{Eberle} the martingale problem is considered in a class of diffusion processes with sub-invariant measure $m$.
In \cite[Definition 2.5, Proposition 2.6]{St99}
the martingale problem is considered in the frame of right processes with given sub-invariant measure and related to $L^1$-uniqueness. This is the framework that we adopt in \cite{LST22} and which is also the basis for this work. Furthermore, $L^1$-uniqueness results related to our work can be found in \cite{RS11, OR12} and as already mentioned in \cite{St99}. In \cite{L97, LL06}, partial differential operators $\mathcal{A}$ with unbounded coefficients are analyzed in $L^2$-spaces with prescribed invariant measure in order to obtain optimal regularity results for $\mathcal{A}$ and the corresponding analytic semigroups and resolvents. \\
{\bf Organization of the paper:} In Section \ref{framework} we introduce the basic notations, our main assumption {\bf (A)}, and types of convergences (see e.g. \eqref{closablet}) that will be used throughout this text. In Section \ref{l1closed}, we construct in Theorem \ref{mainijcie} a generator that a solution to \eqref{weaksolutionintro} should have. Subsequently, we obtain the corresponding generalized Dirichlet form and then show that it has an associated process in Proposition \ref{Huntex}, which in fact solves \eqref{weaksolutionintro} (cf. Theorem \ref{prop:3.1.6}). In Section \ref{l1closeduni} we investigate the $L^1(\R^d,\widehat{\mu})$-uniqueness of the generator constructed in Section \ref{l1closed}. Section \ref{conditionaluniquness} is devoted to our uniqueness in law results and we present some examples (cf. Example \ref{example section 5}). In Section \ref{aux} we present auxiliary results that do not really fit in the main body of the text and are better presented at the end.
\section{Notations, conventions, framework}\label{framework}
{\bf In this section, we will introduce notations, conventions and the framework that will be (if not otherwise mentioned) in force throughout  this work.} Notations which are not mentioned in this text, will be those as defined in \cite{LST22}.
\begin{defn} \label{basidefn} 
(i) For a (possibly non-symmetric) matrix of locally integrable functions $B = (b_{ij})_{1 \leq i,j \leq d}$ on $\mathbb{R}^d$ and a vector field $\bold{E}=(e_1,\ldots, e_d) \in L^1_{loc}(\mathbb{R}^d, \mathbb{R}^d)$, we write ${\rm{div}}B = \bold{E}$ if
\begin{equation} \label{defndivma}
\int_{\mathbb{R}^d} \sum_{i,j=1}^d b_{ij} \partial_i \phi_j \, dx = - \int_{\mathbb{R}^d} \sum_{j=1}^d e_j \phi_j \, dx \quad \text{ for all $\phi_j \in C_0^{\infty}(\mathbb{R}^d)$, $j=1, \ldots, d$}.
\end{equation}
In other words, if we consider the family of column vectors
\[
\bold{b}_j=\begin{pmatrix} 
b_{1j} \\ \vdots \\  b_{dj}
\end{pmatrix},\quad 1\le j\le d,\  \text{ i.e.}\  B=(\bold{b}_1 | \ldots |\bold{b}_d),
\]
then \eqref{defndivma} is equivalent to 
\begin{equation*}
\int_{\mathbb{R}^d} \sum_{j=1}^d \langle \bold{b}_j, \nabla \phi_j \rangle dx =-\int_{\mathbb{R}^d} \langle \bold{E}, \Phi \rangle dx, \text{ \;\;\; for all $\Phi=(\phi_1, \ldots, \phi_d) \in C_0^{\infty}(\mathbb{R}^d)^d$.}
\end{equation*}
Let $B^T=(b^T_{ij})_{1 \leq i,j \leq d}=(b_{ji})_{1 \leq i,j \leq d}$ denote the transpose of $B$. For $\rho \in H^{1,1}_{loc}(\mathbb{R}^d)$ with $\rho>0$ a.e., we write
\begin{equation*}\label{logder 2}
\beta^{\rho, B^T} := \frac12{\rm{div}} B + \frac{1}{2\rho} B^T \nabla \rho.
\end{equation*}
Here, we adopt the convention that for $\varphi \in L^1_{loc}(\mathbb{R}^d)$,with $\varphi\not=0$ a.e., $\frac{1}{\varphi}$ denotes any Borel measurable function satisfying $\varphi \cdot \frac{1}{\varphi}=1$ a.e. For $\psi \in L^1_{loc}(\mathbb{R}^d)$, $\psi>0$ a.e. with $\frac{1}{\psi} \in L^{\infty}_{loc}(\mathbb{R}^d)$, we write
\begin{equation}\label{logder 1}
\beta^{\rho, B^T, \psi}:= \frac{1}{\psi} \beta^{\rho, B^T} = \frac{1}{2\psi} {\rm{div}} B + \frac{1}{2\psi\rho }  B^T \nabla \rho
\end{equation}
(ii) For $B=(b_{ij})_{1 \leq i,j \leq d}$, with $b_{ij} \in H^{1,1}_{loc}(\mathbb{R}^d)$, $1\le i,j\le d$, we define a vector field $\nabla B =\big(  (\nabla B)_1,  \ldots, (\nabla B)_d \big) $ on $\mathbb{R}^d$ by
$$
(\nabla B)_i = \sum_{j=1}^d \partial_j b_{ij}, \quad \; i=1, \ldots, d.
$$
\end{defn}

\begin{remark}  \label{rem1.2}
(i) For a matrix of locally integrable functions $B = (b_{ij})_{1 \leq i,j \leq d}$ on $\mathbb{R}^d$ and a vector field $\bold{E}=(e_1,\ldots, e_d) \in L^1_{loc}(\mathbb{R}^d, \mathbb{R}^d)$, assume that ${\rm{div}} B = \bold{E}$.  Let $\Phi=(\phi_1, \ldots, \phi_d) \in C_0^{\infty}(\mathbb{R}^d)^d$ be arbitrary. Then, there exists an open ball  $U$ in $\mathbb{R}^d$ such that 
$\text{supp} \, \phi_j \subset U$ for any $j=1, \ldots, d$. For each $1\le i,j \le d$, choose an arbitrary a sequence of functions $(b^n_{ij})_{n \geq 1}$ in $C^1(\overline{U})$ such that
$\lim_{n \rightarrow \infty} b^n_{ij} =b_{ij}$ in $L^1(U)$. Then, setting $B_n=(b^n_{ij})_{1\le i,j\le d}$, $n\ge 1$,
\begin{align*}
- \int_{\mathbb{R}^d} \sum_{j=1}^d e_j \phi_j \, dx&=\int_{\mathbb{R}^d} \sum_{i,j=1}^d b_{ij} \partial_i \phi_j \, dx = \lim_{n \rightarrow \infty} \int_{\mathbb{R}^d} \sum_{i,j=1}^d b^n_{ij} \partial_i \phi_j \, dx\\
&= \lim_{n \rightarrow \infty} -\int_{\mathbb{R}^d} \sum_{j=1}^d \Big(\sum_{i=1}^d \partial_i b^n_{ij}\Big)  \phi_j dx, 
\end{align*}
i.e. for $\Phi = (\phi_1, \ldots, \phi_d )$, it holds
\begin{equation} \label{divequation1}
\lim_{n \rightarrow \infty} \int_{\mathbb{R}^d} \langle \nabla B^T_n, \Phi \rangle dx =  \int_{\mathbb{R}^d} \langle \bold{E}, \Phi \rangle dx.
\end{equation}
In particular, if $b_{ij} \in H^{1,1}_{loc}(\mathbb{R}^d)$ for all $1\le i,j \le d$, we may choose $B_n=(b^n_{ij})_{1 \leq i,j \leq d}$ in such a way that additionally 
$\lim_{n \rightarrow \infty} \partial_i b^n_{ij} =\partial_i b_{ij}$ in $L^1(U)$ for any $1\le i,j \le d$. Hence in that case, we can conclude from \eqref{divequation1}, holding for arbitrary  $\Phi\in C_0^{\infty}(\mathbb{R}^d)^d$, that $\nabla B^T = {\rm{div}}\, B$.\\
(ii)
In order to define ${\rm{div}}\, B$ for $B=(b_{ij})_{1 \leq i,j \leq d}$, we do not need to impose the condition that $b_{ij} \in H^{1,1}_{loc}(\mathbb{R}^d)\cup C(\mathbb{R}^d)$ for each $1\le i,j \le d$. For instance, let $\zeta \in L^{\infty}(\mathbb{R})$ and assume that $\zeta$ is discontinuous and $\zeta \notin H^{1,1}_{loc}(\mathbb{R})$. Let $b_{11}(x_1, \ldots, x_d)=\zeta(x_2)$, $b_{22}(x_1, \ldots, x_d)=\zeta(x_3)$, \ldots, $b_{d-1 \, d-1}(x_1, \ldots, x_d)=\zeta(x_d)$, $b_{dd}(x_1, \ldots, x_d)=\zeta(x_1)$ and $b_{ij}=0$ if $i \neq j$. Then, it holds that $b_{ii} \notin H^{1,1}_{loc}(\mathbb{R}^d)$, $1\le i \le d$, but ${\rm{div}}\, B = 0$.
\end{remark}

\noindent
Consider the following set of assumptions:\\ \\
{\bf (A)}:\; For some $d \geq 1$, $\rho  = \varphi^2$ with $\varphi \in H^{1,2}_{loc}(\mathbb{R}^d)$, $\rho>0$ a.e. on $\mathbb{R}^d$ and $\mu:=\rho dx$. $\psi \in L^1_{loc}(\R^d)$ satisfies $\psi>0$ a.e., $\frac{1}{\psi} \in L_{loc}^{\infty}(\R^d)$, $\psi\rho\in L^1_{loc}(\mathbb{R}^d)$, and $\widehat{\mu}:=\psi\,d\mu$.
$A=(a_{ij})_{1 \leq i,j \leq d}$ satisfying $\text{div}\, A \in L^2_{loc}(\mathbb{R}^d, \mathbb{R}^d, \mu)$ is a symmetric matrix of measurable functions that is locally uniformly strictly elliptic, i.e.  for every open ball $B$, there exist constants $\lambda_B, \Lambda_B>0$ such that 
\begin{equation*}\label{ellipticity}
\lambda_{B} \|\xi\|^2 \leq \langle A(x) \xi, \xi \rangle \leq \Lambda_B \|\xi\|^2, \quad \text{ for all } \xi \in \R^d, \; x \in B.
\end{equation*}
Moreover, $\mathbf{B} \in L_{loc}^2(\R^d, \R^d, \widehat{\mu})$ satisfying $\psi \bold{B} \in L^2_{loc}(\mathbb{R}^d, \mathbb{R}^d, \mu)$
is weakly divergence free with respect to $\widehat{\mu}$, i.e.
\begin{equation} \label{divfree}
\int_{\R^d} \langle \mathbf{B}, \nabla u \rangle d\widehat{\mu} = 0, \quad \text{ for all } u \in C_0^{\infty}(\R^d),
\end{equation}
and as fixed notations throughout the text, we set
\begin{equation*} \label{defnhata}
\widehat{A} := \frac{1}{\psi} A, \quad \bold{G}:=\beta^{\rho, A,\psi}+\bold{B}.
\end{equation*}
\centerline{}
\begin{remark} \label{inclupro}
(i) Both measures $\mu$ and $\widehat{\mu}$ occurring in {\bf (A)} are locally finite and equivalent to the Lebesgue measure $dx$. Thus $L^{\infty}(\mathbb{R}^d, \mu)=L^{\infty}(\mathbb{R}^d, \widehat{\mu})=L^{\infty}(\mathbb{R}^d)$ and $L_{loc}^{\infty}(\mathbb{R}^d, \mu)=L_{loc}^{\infty}(\mathbb{R}^d, \widehat{\mu})=L_{loc}^{\infty}(\mathbb{R}^d)$.  {\bf Therefore from now on, we will consistently only use the notations} a.e., $L^{\infty}(\mathbb{R}^d)$, and $L_{loc}^{\infty}(\mathbb{R}^d)$. \\
(ii) Let $r \in [1, \infty)$. Since $\frac{1}{\psi} \in L^{\infty}_{loc}(\mathbb{R}^d)$, we get $L^r_{loc}(\mathbb{R}^d, \widehat{\mu}) \subset L^r_{loc}(\mathbb{R}^d, \mu)$. Indeed,  for any bounded open subset $V$ of $\mathbb{R}^d$ and $u \in L^{r}(V, \widehat{\mu})$, it holds $u \in L^r(V, \mu)$ and
$$
\|u\|_{L^r(V, \mu)} \leq  \left \| \frac{1}{\psi}\right \|_{L^{\infty}(V)}^{1/r} \|u \|_{L^{r}(V, \widehat{\mu})}.
$$
\end{remark}
Assume {\bf (A)} and let $U$ be an open subset of $\R^d$. Then, the following bilinear form 
$$
\frac12 \int_{U} \langle \nabla u, \nabla v \rangle \,d\mu, \quad u,v \in C_0^{\infty}(U)
$$ 
is closable in $L^2(U, \widehat{\mu})$ by \cite[I. Proposition 3.3]{MR}. Indeed, for each $u, v \in C_0^{\infty}(U)$
\begin{align*}
\frac12 \int_{U} \langle \nabla u, \nabla v \rangle d\mu &= -\int_{U} \Big(\frac12 \Delta u + \langle \frac{1}{2\rho} \nabla \rho, \nabla u \rangle \Big) v \rho\, dx  \\
&= -\int_{U} \Big(\frac{1}{2\psi} \Delta u + \langle  \frac{1}{2\rho \psi} \nabla \rho, \nabla u \rangle  \Big) v \,d \widehat{\mu}
\end{align*}
and $\frac{1}{2\psi} \Delta u + \langle  \frac{1}{2\rho \psi} \nabla \rho, \nabla u \rangle \in L^2(U, \widehat{\mu})$. 
\begin{defn}\label{defmaydirichletspace}
Let $\widehat{H}_0^{1,2}(U, \mu)$ be the closure of $C_0^{\infty}(U)$ in $L^2(U, \widehat{\mu})$ with respect to the norm 
\[
\left(\int_U \|\nabla u\|^2 d \mu + \int_{U} u^2 d \widehat{\mu} \right)^{1/2}.
\] 
\end{defn}
Consequently, $u \in \widehat{H}_0^{1,2}(U, \mu)$, if and only if there exists $(u_n)_{n \geq 1} \subset C_0^{\infty}(U)$  such that 
\begin{equation} \label{closablet}
\lim_{n \rightarrow \infty} u_n =u \; \text{ in } L^2(U, \widehat{\mu}), \;\quad \lim_{n,m \rightarrow \infty} \int_U \| \nabla (u_n-u_m) \|^2 d \mu =0,
\end{equation}
and moreover $\widehat{H}^{1,2}_0(U, \mu)$ is a Hilbert space equipped
with the inner product 
$$
\langle u, v \rangle_{\widehat{H}^{1,2}_0(U, \mu)}= \lim_{n \rightarrow \infty} \int_{U} \langle \nabla u_n, \nabla v_n \rangle d \mu + \int_{U} uv \,d\widehat{\mu},
$$
where $(u_n)_{n \geq 1}$ and $(v_n)_{n \geq 1} \subset C_0^{\infty}(U)$ are arbitrary sequences such that $(u_n)_{n \geq 1}$ satisfies
\eqref{closablet} and $(v_n)_{n \geq 1}$ satisfies the analogue of \eqref{closablet} with $u$, $u_n$, replaced by $v$, $v_n$, $n \geq 1$.
Likewise, considering $\psi \equiv 1$, we can define $H^{1,2}_0(U, \mu)$ as the closure of $C_{0}^{\infty}(U)$ in $L^2(U, \mu)$  with respect to the norm
$\left(\int_{U} \big (\|\nabla u \|^2 +u^2 \big )d\mu\right)^{1/2}$.\\
\noindent
Now consider a bounded open subset $V$ of $\R^d$. If $u \in \widehat{H}^{1,2}_0(V, \mu)$, then using Remark \ref{inclupro}(ii), $u \in H_0^{1,2}(V, \mu) \cap L^2(V, \widehat{\mu})$ and there exists $(u_n)_{n \geq 1} \subset C_0^{\infty}(V)$ such that
$$
\lim_{n \rightarrow \infty} u_n = u \;\; \text{ in } H_0^{1,2}(V, \mu)  \; \text{ and } \text{ in } L^2(V, \widehat{\mu}).
$$
Consider the symmetric bilinear form $(\mathcal{E}^0, C_0^{\infty}(\mathbb{R}^d))$ defined by
\begin{equation*} \label{weofijfh}
\mathcal{E}^0(f,g) := \frac12 \int_{\R^d} \langle \widehat{A} \nabla f, \nabla g \rangle d \widehat{\mu}, \quad f,g \in C_0^{\infty}(\R^d).
\end{equation*}
Then, by Lemma \ref{basicprop}
\begin{eqnarray*}
\mathcal{E}^0(f,g) =   - \int_{\R^d} \big( \frac12 {\rm trace} (\widehat{A}\nabla^2f ) + \langle \beta^{\rho, A, \psi}, \nabla f \rangle \big) g \,d \widehat{\mu}, \quad \text{  for all }  f, g \in C_0^{\infty}(\mathbb{R}^d).
\end{eqnarray*}
For each $f \in C_0^{\infty}(\mathbb{R}^d)$, $\frac12 \text{trace} (\widehat{A}\nabla^2f ) + \langle \beta^{\rho, A, \psi}, \nabla f \rangle \in L^2(\mathbb{R}^d, \widehat{\mu})$, hence $(\mathcal{E}^0, C_0^{\infty}(\mathbb{R}^d))$ is closable in $L^{2}(\R^d, \widehat{\mu})$ by \cite[I. Proposition 3.3]{MR} and its closure denoted by $(\mathcal{E}^0, D(\mathcal{E}^0))$ is a symmetric  Dirichlet form on $L^2(\R^d, \widehat{\mu})$ (see \cite[II. (2.18)]{MR}). For each $\alpha>0$, let 
$$
\mathcal{E}_{\alpha}^0(f,g) = \mathcal{E}^0(f,g) + \alpha \int_{\mathbb{R}^d} f g \,d\widehat{\mu},\,\quad f,g \in D(\mathcal{E}^0)
$$
and let $\|u\|_{D(\mathcal{E}^{0})}=\mathcal{E}_1^{0}(u,u)^{1/2}$, $u \in D(\mathcal{E}^0)$.
Denote the corresponding generator of $(\mathcal{E}^0, D(\mathcal{E}^0))$ by $(L^0,D(L^0))$. Then we obtain that $C_0^{2}(\R^d) \subset D(L^0)$ and for each $f \in C_0^{2}(\mathbb{R}^d)$
$$
L^0f =\frac12\text{trace}(\widehat{A} \nabla^2 f) + \langle \beta^{\rho, A, \psi}, \nabla f \rangle  \in L^2(\R^d, \widehat{\mu}).
$$ 
Let $(T^0_t)_{t>0}$ be the sub-Markovian $C_0$-semigroup of contractions on $L^2(\R^d, \widehat{\mu})$ associated with $(L^0,D(L^0))$. By Proposition \ref{extlempt}(i), \,$(T^0_t)_{t>0}|_{L^1(\R^d, \widehat{\mu})_b}$ can be uniquely extended to a sub-Markovian  $C_0$-semigroup of contractions $(\overline{T}^0_t)_{t>0}$ on $L^1(\R^d, \widehat{\mu})$.
Using Lemma \ref{applempn},  \eqref{divfree} can be extended to all $u \in \widehat{H}_0^{1,2}(\R^d, \mu)_{0,b}$ (where the latter subindex \lq\lq $0$\rq\rq\ means with compact $dx$-support and \lq\lq $b$\rq\rq\ means a.e. bounded, cf. \cite{LST22}), i.e.
\begin{equation} \label{divfreeext}
\int_{\R^d} \langle \mathbf{B}, \nabla u \rangle d\widehat{\mu} = 0, \quad \text{ for all } u \in \widehat{H}_0^{1,2}(\R^d, \mu)_{0,b},
\end{equation}
 and moreover
$$
\int_{\R^d} \langle \mathbf{B}, \nabla u \rangle v d \widehat{\mu} = -\int_{\R^d} \langle \mathbf{B}, \nabla v \rangle u d \widehat{\mu}, \quad \text{ for all } u,v \in \widehat{H}_0^{1,2}(\R^d, \mu)_{0,b}.
$$
Define 
$$
Lu := L^0 u+ \langle \mathbf{B}, \nabla u \rangle, \;\; \; u \in D(L^0)_{0,b}. 
$$
Then,  $(L, D(L^0)_{0,b})$ is an extension of
$$
\frac12 \text{trace}(\widehat{A} \nabla^2 u) + \langle \beta^{\rho, A,\psi} + \mathbf{B}, \nabla u \rangle, \quad u \in C_0^{\infty}(\R^d).
$$
For any bounded open subset $V$ of $\R^d$, define $(\mathcal{E}^{0,V}, C_0^{\infty}(V))$ by
$$
\mathcal{E}^{0,V}(f,g) := \frac12 \int_{V} \langle \widehat{A} \nabla f, \nabla g \rangle d \widehat{\mu}, \quad f,g \in C_0^{\infty}(V).
$$
Then, $(\mathcal{E}^{0,V}, C_0^{\infty}(V))$ is closable on $L^2(V, \widehat{\mu})$ by \cite[I. Proposition 3.3]{MR}. Denote by $(\mathcal{E}^{0,V}, D(\mathcal{E}^{0,V}))$ the closure of $(\mathcal{E}^{0,V}, C_0^{\infty}(V))$ in $L^2(V, \widehat{\mu})$.
Then, $(\mathcal{E}^{0,V}, D(\mathcal{E}^{0,V}))$ is a symmetric Dirichlet form (see \cite[II. (2.18)]{MR}).
Let 
$$
\mathcal{E}_{\alpha}^{0,V}(f,g) = \mathcal{E}^{0,V}(f,g) + \alpha \int_{V} f g \,d\widehat{\mu}, \;\quad  f,g \in D(\mathcal{E}^{0,V})
$$
and 
$\|u\|_{D(\mathcal{E}^{0,V})}=\mathcal{E}_1^{0,V}(u,u)^{1/2}$, $u \in D(\mathcal{E}^{0,V})$. Then, $D(\mathcal{E}^{0,V})= \widehat{H}_0^{1,2}(V, \mu)$ since the norms $\| \cdot \|_{D(\mathcal{E}^{0,V})}$ and  $\|\cdot\|_{\widehat{H}^{1,2}_0 (V, \mu)}$ are equivalent by the uniform strict ellipticity of $A$ on $V$ .
 Denote by $(L^{0,V}, D(L^{0,V}))$ the generator of $(\mathcal{E}^{0,V}, D(\mathcal{E}^{0,V}))$, by $(G^{0,V}_{\alpha})_{\alpha>0}$ the associated strongly continuous sub-Markovian resolvent of contractions on $L^2(V, \widehat{\mu})$, by $(T_t^{0,V})_{t>0}$ the associated sub-Markovian $C_0$-semigroup of contractions on $L^2(V, \widehat{\mu})$ 
(see \cite[I.2]{MR}). Let $(\overline{T}_t^{0,V})_{t>0}$ be the unique extension of $(T^{0,V}_t)_{t>0}|_{L^1(V, \widehat{\mu})_b}$  on $L^1(V, \widehat{\mu})$, which is a sub-Markovian  $C_0$-semigroup of contractions on $L^1(V, \widehat{\mu})$ (see Proposition \ref{extlempt}). Let $(\overline{L}^{0,V}, D(\overline{L}^{0,V}))$ be the generator corresponding to $(\overline{T}_t^{0,V})_{t>0}$.
By Proposition \ref{extlempt}, $(\overline{L}^{0,V}, D(\overline{L}^{0,V}))$ is the closure of $(L^{0,V}, D(L^{0,V}))$  on $L^1(V, \widehat{\mu})$. 
\centerline{}

\section{Existence of an $L^1(\mathbb{R}^d, \widehat{\mu})$-closed extension} \label{l1closed}
In this section, we show that there exists a closed extension $(\overline{L}, D(\overline{L}))$ of 
$$
Lu = L^0 u + \langle \mathbf{B}, \nabla u \rangle, \quad u \in D(L^0)_{0, b},
$$
that generates a sub-Markovian $C_0$-semigroup of contractions on $L^1(\mathbb{R}^d, \widehat{\mu})$.
Since $C_0^{\infty}(\mathbb{R}^d) \subset D(L^0)_{0,b}$, we thus obtain in particular a closed extension of $(L, C_0^{\infty}(\mathbb{R}^d))$. The proofs of the following Lemma \ref{existlem}, Proposition \ref{first1}, Lemma \ref{increlem}, Theorems \ref{mainijcie} and \ref{pojjjde} in order to obtain a closed extension are similar to those of \cite[Lemma 1.2, Proposition 1.1, Lemma 1.6, Theorem 1.5 and Remark 1.7(iii)]{St99}, respectively. But it is nonetheless necessary to check the detailed calculations related to the measure $\widehat{\mu}$ and therefore the presentation of the proofs is necessary. 

\begin{lemma} \label{existlem}
Assume {\bf (A)}. Let $V$ be a bounded open subset of $\R^d$. Then, the following hold:
\begin{itemize}
\item[(i)]
$D(\overline{L}^{0,V})_b \subset \widehat{H}_0^{1,2}(V, \mu)$.
\item[(ii)]
$\lim_{t \rightarrow 0+}T_t^{0,V}u = u$ \;\, in $\widehat{H}_0^{1,2}(V, \mu)$  \;for all  $u \in D(\overline{L}^{0,V})_b$.

\item[(iii)]
$\mathcal{E}^0(u,v) = - \int_{V} \overline{L}^{0,V} u\, v \,d\widehat{\mu}$ \,\;for all $u \in D(\overline{L}^{0,V})_b$ and $v \in \widehat{H}^{1,2}_0(V, \mu)_b$.

\item[(iv)]
Let $\varphi \in C^2(\R)$, $\varphi(0)=0$ and $u \in D(\overline{L}^{0,V})_b$. Then $\varphi(u) \in D(\overline{L}^{0,V})_b$ and
\begin{equation*} \label{cdsid}
\overline{L}^{0,V} \varphi(u) = \varphi'(u) \overline{L}^{0,V} u +  \frac12\varphi ''(u) \langle\widehat{A} \nabla u, \nabla u \rangle.
\end{equation*}
\end{itemize}
\end{lemma}

\noindent
\begin{proof}
(i) and (ii): Let $u \in D(\overline{L}^{0,V})_b$. Since $(T^{0,V}_t)_{t>0}$ is an analytic semigroup on $L^2(V, \widehat{\mu})$, we get
$$
\overline{T}_t^{0,V} u =T^{0,V}_t u \in D(L^{0,V}) \;\; \text{ for all } t>0,
$$ 
and hence by Proposition \ref{extlempt},
$$
L^{0,V} T_t^{0,V}u = L^{0,V} \overline{T}_t^{0,V} u = \overline{L}^{0,V}\overline{T}_t^{0,V}u = \overline{T}_t^{0,V} \overline{L}^{0,V} u.
$$
Therefore
\begin{eqnarray*}
&&\mathcal{E}^{0,V}(T^{0,V}_t u - T_s^{0,V} u, T^{0,V}_t u - T_s^{0,V} u ) \\
&& \qquad \qquad = - \int_{V} L^{0,V} ( T^{0,V}_t u - T_s^{0,V} u ) \cdot (T^{0,V}_t u - T_s^{0,V} u ) d \widehat{\mu} \\
&&\qquad \qquad   =-\int_{V} ( \overline{T}^{0,V}_t \overline{L}^{0,V} u - \overline{T}_s^{0,V} \overline{L}^{0,V}  u ) \cdot (T^{0,V}_t u - T_s^{0,V} u) d \widehat{\mu} \\
&&  \qquad \qquad  \leq \| \overline{T}^{0,V}_t \overline{L}^{0,V} u - \overline{T}_s^{0,V} \overline{L}^{0,V}  u \|_{L^1(V, \widehat{\mu})} \cdot 2 \|u\|_{L^{\infty}(V)}\ \longrightarrow \ 0 \;\; \; \text{ as } \; t, s \rightarrow 0+.
\end{eqnarray*}
Thus $(T_t^{0,V} u)_{t>0}$ is an $\widehat{H}_0^{1,2}(V, \mu)$-Cauchy sequence as $t \rightarrow 0+$, which implies $u \in \widehat{H}_0^{1,2}(V, \mu)$ and  $\lim_{t \rightarrow 0+} T^{0,V}_t u = u \; \text{ in } \;  \widehat{H}_0^{1,2}(V, \mu)$.
Thus (i), (ii) are proved. \\ \\
(iii): Let $v \in \widehat{H}_0^{1,2}(V, \mu)_b$. Then by (ii)
\begin{eqnarray*}
\mathcal{E}^{0,V}(u,v) &=& \lim_{t \rightarrow 0+} \mathcal{E}^{0,V} (T^{0,V}_t u, v) =-\lim_{t \rightarrow 0+} \int_{V} \big(L^{0,V} T_t^{0,V} u \big) v\, d \widehat{\mu} \\
&=& -\lim_{t\rightarrow 0+} \int_V \big(  \overline{T}^{0,V}_t \overline{L}^{0,V} u \big) \; v \,d \widehat{\mu}   =  -\int_V \overline{L}^{0,V}u  \; v \,d \widehat{\mu},
\end{eqnarray*}
hence (iii) is proved.\\
\centerline{}
\quad (iv): By (i) $u \in D(\overline{L}^{0,V})_b \subset \widehat{H}_0^{1,2}(V, \mu)_b$.  Since normal contractions operate on $\widehat{H}_0^{1,2}(V, \mu)$, we obtain 
$\varphi(u) \in \widehat{H}^{1,2}_0(V, \mu)$, and $\varphi'(u)- \varphi'(0)\in \widehat{H}_0^{1,2}(V, \mu)$. Let $v \in \widehat{H}_0^{1,2}(V, \mu)_b$.
Then by \cite[I. Corollary 4.15]{MR}, $v \big (\varphi'(u)- \varphi'(0)\big ) \in \widehat{H}_0^{1,2}(V, \mu)_b$ and 
\begin{eqnarray*}
\mathcal{E}^{0,V}(\varphi(u), v)  &=&  \frac12 \int_V \langle \widehat{A} \nabla \varphi(u), \nabla v \rangle d \widehat{\mu} \\
&=&  \frac12 \int_V \langle \widehat{A} \nabla u, \nabla v \rangle \varphi'(u)d  \widehat{\mu} \\
&=&  \frac12 \int_V \langle \widehat{A} \nabla u, \nabla (v \varphi'(u))  \rangle d \widehat{\mu}  -  \frac12 \int_V \langle \widehat{A} \nabla u, \nabla u \rangle \varphi''(u)v \,d \widehat{\mu} \\
&=& -\int_V  \Big(\varphi'(u)\overline{L}^{0,V} u   +   \frac12\varphi''(u) \langle\widehat{A} \nabla u, \nabla u \rangle \Big) v\, d \widehat{\mu}.
\end{eqnarray*}
Since $ \varphi'(u)\overline{L}^{0,V} u   +   \frac12 \varphi''(u) \langle\widehat{A} \nabla u, \nabla u \rangle  \in L^1(V, \widehat{\mu})$, (iv) holds by \cite[I. Lemma 4.2.2.1]{BH}.
\end{proof}

\noindent Recall that a densely defined operator $(S,D(S))$ on a Banach space $X$ is called {\it dissipative} \;if for any $u \in D(S)$, there exists $l_u \in X'$ such that 
\begin{equation} \label{dissipativei}
\|l_u\|_{X'} = \|u\|_X,  \,\; l_u(u) = \|u\|_X^2  \;\text{ and  }\,  l_u(Su) \leq 0.
\end{equation}
\medskip
\begin{proposition} \label{first1}
Assume {\bf (A)}. Let $V$ be a bounded open subset of $\R^d$.  Then, the following hold:
\begin{itemize}
\item[(i)]
The operator $(L^V, D(L^{0,V})_b)$  defined by
$$
L^{V} u  := L^{0,V} u  + \langle \mathbf{B}, \nabla u \rangle,\;\;  u \in D(L^{0,V})_b,
$$ 
is dissipative, hence closable on $L^1(V, \widehat{\mu})$. The closure $(\overline{L}^V, D(\overline{L}^V))$ generates a sub-Markovian $C_0$-semigroup of contractions $(\overline{T}^V_t)_{t>0}$ on $L^1(V, \widehat{\mu})$ with associated resolvent $(\overline{G}_{\alpha}^V)_{\alpha>0}$.

\item[(ii)]
$D(\overline{L}^V)_b \subset \widehat{H}^{1,2}_0(V, \mu)$ and
\begin{equation} \label{stamid}
\mathcal{E}^{0,V}(u,v) - \int_{V} \left \langle \mathbf{B}, \nabla u  \right \rangle v d \widehat{\mu}  = -\int_V \overline{L}^V u \cdot v d \widehat{\mu} \; \;  \; \text{ for all } u \in D(\overline{L}^V)_b \; \text{ and } \; v \in \widehat{H}_0^{1,2}(V, \mu)_b.
\end{equation}

\item[(iii)]
If $f \in L^{\infty}(V)$ and $\alpha>0$, then it holds that  $\overline{G}_{\alpha}^V f \in D(\overline{L}^V)_b \subset \widehat{H}_0^{1,2}(V, \mu)_b$ and that  
$$
\mathcal{E}^{0, V}_{\alpha}(\overline{G}_{\alpha}^V f, \varphi)- \int_{V} \langle \mathbf{B}, \nabla \overline{G}_{\alpha}^V f \rangle \varphi \,d\widehat{\mu} = \int_{V} f \varphi d \widehat{\mu}, \quad \text{ for any $\varphi \in \widehat{H}^{1,2}_0(V, \mu)_b$}.
$$
\end{itemize}
\end{proposition}

\noindent
\begin{proof} (i)
\; {\bf Step 1:} For $u \in D(L^{0,V})_b$, we claim that $\int_V L^V u 1_{\{ u >1 \}} d \widehat{\mu} \leq 0$. \\
Let $\varphi_{\varepsilon}  \in C^2(\R)$, $\varepsilon>0$, be such that $\varphi_{\varepsilon}''\geq 0$,  $0 \leq \varphi_{\varepsilon}' \leq 1$ and  $\varphi_{\varepsilon}(t) = 0$ if $t<1$, $\varphi'_{\varepsilon}(t) = 1$ if $t \geq 1+\varepsilon$. Then $\varphi_{\varepsilon}(u) \in D(\overline{L}^{0,V})$ by Lemma \ref{existlem}(iv) and 
\begin{eqnarray}
\int_{V} L^{0,V} u \, \varphi'_{\varepsilon}(u) d \widehat{\mu} 
&\leq& \int_{V} L^{0,V}u \,\varphi'_{\varepsilon}(u)\,d \widehat{\mu} +\int_{V} \frac12 \varphi''_{\varepsilon}(u) \langle \widehat{A} \nabla u, \nabla u \rangle d \widehat{\mu}  \nonumber \\
& = & \int_{V} \overline{L}^{0,V} \varphi_{\varepsilon} (u) d \widehat{\mu}  
= \lim_{t \rightarrow 0+}\int_V \frac{\overline{T}^{0,V}_t \varphi_{\varepsilon} (u)- \varphi_{\varepsilon} (u)     }{t} d \widehat{\mu}
\ \leq\  0, \label{subinvariant1}
\end{eqnarray}
where the last inequality followed by the $L^1(V, \widehat{\mu})$-contraction property of $(\overline{T}_t^{0,V})_{t>0}$. \\
Since $\lim_{\varepsilon \rightarrow 0+} \varphi'_{\varepsilon}(t) = 1_{(1,\infty)}(t)$ for every $t \in \R$, we have
$$
\lim_{\varepsilon \rightarrow 0+}\varphi'_{\varepsilon}(u) = 1_{\{u>1\}} \,\;\; \text{a.e. on } V \; \text{ and } \; \| \varphi'_{\varepsilon} (u)\|_{L^{\infty}(V)} \leq 1.
$$
Thus by Lebesgue's Theorem 
$$
\int_V L^{0,V} u\, 1_{\{ u >1 \}} d \widehat{\mu} = \lim_{\varepsilon \rightarrow 0+} \int_{V} L^{0,V} u\;  \varphi_{\varepsilon}'(u) d \widehat{\mu} \leq 0.
$$
Similarly, since $\varphi_{\varepsilon}(u) \in \widehat{H}^{1,2}_0(V, \mu)$, using \eqref{divfree} we get
$$
\int_{V}\langle \mathbf{B}, \nabla u \rangle  1_{\{u >1\}} d \widehat{\mu} =  \lim_{\varepsilon \rightarrow 0+} \int_{V} \langle \mathbf{B}, \nabla u \rangle \varphi'_{\varepsilon}(u) d  \widehat{\mu}  = \lim_{\varepsilon \rightarrow 0+} \int_{V} \langle \mathbf{B}, \nabla \varphi_{\varepsilon}(u) \rangle d \widehat{\mu}  = 0.
$$
Therefore $\int_{V} L^{V} u 1_{\{u >1 \}} d \widehat{\mu} \leq 0$ and Step 1 is proved. \\[3pt]
Observe that by Step 1, for any $n \geq 1$
$$
\int_{V} \left(L^V n u \right)1_{\{ nu >1 \}} d \widehat{\mu} \leq 0 \, \Longrightarrow \, \int_{V} L^V  u \,1_{\{ u >\frac{1}{n} \}} d \widehat{\mu} \leq 0.
$$
Letting $n \rightarrow \infty$  it follows from Lebesgue's Theorem that
$\int_{V} L^V u \,1_{\{u >0 \}} d \widehat{\mu} \leq 0$. Replacing $u$ with $-u$, we have 
$$
-\int_{V} L^V u \,1_{\{u <0 \}} \,d \widehat{\mu} =  \int_{V} L^V (-u) \,1_{\{-u >0 \}} \,d \widehat{\mu} \leq 0,
$$
and hence
$$
\int_{V} L^V u \,(1_{\{u >0 \}} - 1_{ \{ u <0 \}}) \, d \widehat{\mu} \leq 0.
$$
Setting $l_u := \|u\|_{L^1(V, \widehat{\mu})}  (1_{\{u >0 \}} - 1_{ \{ u <0 \}})  \in L^{\infty}(V) =L^{\infty}(V, \widehat{\mu}) = (L^1(V, \widehat{\mu}))'$, \eqref{dissipativei} is satisfied. Since $C_0^{\infty}(V) \subset   D(L^{0,V})_b$, $(L^{0,V}, D(L^{0,V})_b)$ is densely defined on $L^1(V, \widehat{\mu})$. Consequently,
 $(L^{0,V}, D(L^{0,V})_b)$ is dissipative. \\
\centerline{}
{\bf Step 2:} \;We claim that $(1-L^V)(D(L^{0,V})_b) \subset L^1(V, \widehat{\mu})$  densely. \\
Let $h \in L^{\infty}(V) = L^{\infty}(V, \widehat{\mu})  = (L^1(V, \widehat{\mu}))'$ be such that $\int_{V} (1-L^V) u \, h d \widehat{\mu} = 0$ for all $u \in D(L^{0,V})_b$. Then,
\begin{equation} \label{operdivfree}
\int_{V} (1-L^{0,V}) u \cdot h \, d\widehat{\mu} = \int_{V} \langle \mathbf{B}, \nabla u \rangle h \, d\widehat{\mu}, \quad \text{ for all } u \in D(L^{0,V})_b.
\end{equation}
Now consider the map $\mathcal{T}: \widehat{H}^{1,2}_0(V, \mu)\to \R$ given by
$$
\mathcal{T} u=\int_{V} \langle \mathbf{B}, \nabla u \rangle h \, d\widehat{\mu}, \qquad u \in \widehat{H}^{1,2}_0(V, \mu).
$$
Then, $\mathcal{T}$ is continuous on $\widehat{H}^{1,2}_0(V, \widehat{\mu})$ since
\begin{eqnarray*}
\big | \int_{V} \langle \mathbf{B}, \nabla u \rangle h \, \psi\rho dx \big | &\leq& \|h\|_{L^{\infty}(V)} \| \psi \mathbf{B} \|_{L^2(V, \mu)} \| \nabla u\|_{L^2(V, \mu)} \\
&\leq& \|h\|_{L^{\infty}(V)} \| \psi \mathbf{B} \|_{L^2(V, \mu)} \|u\|_{\widehat{H}_0^{1,2}(V, \mu)}.
\end{eqnarray*}
Thus, by the Riesz representation Theorem, there exists $v \in \widehat{H}_0^{1,2}(V, \mu)$ such that
$$
\mathcal{E}_1^{0,V}(u,v) = \mathcal{T}u \quad \text{ for all } u \in \widehat{H}^{1,2}_0(V, \mu),
$$
which together with \eqref{operdivfree} implies that
\begin{equation} \label{den12}
\int_V (1-L^{0,V})u \cdot  (h-v) d \widehat{\mu} = 0  \quad \text{ for all } u \in D(L^{0,V})_b.
\end{equation}
Since $G^{0,V}_{1}(L^1(V, \widehat{\mu})_b)\subset D(L^{0,V})_b$ and $(1-L^{0,V})G^{0,V}_{1}g=g$ for any $g\in L^1(V, \widehat{\mu})_b$, \eqref{den12} implies $h-v = 0$. In particular, $h \in \widehat{H}_0^{1,2}(V, \mu)$ and by \cite[I. Theorem 2.13(ii)]{MR}
\begin{eqnarray*}
\mathcal{E}^{0,V}_1(h,h) &=& \lim_{\alpha \rightarrow \infty} \mathcal{E}^{0,V}_1(\alpha G^{0,V}_{\alpha} h, h) \\
&=& \lim_{\alpha \rightarrow \infty} \int_V (1-L^{0,V})(\alpha G^{0,V}_{\alpha} h) \, h d \widehat{\mu} \\
&=& \lim_{\alpha \rightarrow \infty} \int_{V} \big \langle \mathbf{B}, \nabla (\alpha G^{0,V}_{\alpha} h) \big \rangle h \psi\rho dx  \\
&=& \int_{V} \big \langle \psi \mathbf{B}, \nabla h \big \rangle h \rho dx =  \frac12 \int_{V} \big \langle  \mathbf{B}, \nabla h^2 \big \rangle d \widehat{\mu} = 0,
\end{eqnarray*}
therefore $h = 0$. Then applying the Hahn-Banach Theorem (see, e.g.\cite[Proposition 1.9]{BRE}), Step 2 is proved.  By the Lumer-Phillips Theorem \cite[Theorem 3.1]{LP61}, the closure $(\overline{L}^V, D(\overline{L}^V))$ of $(L^V, D(L^{0,V})_b)$ generates a $C_0$-semigroup $(\overline{T}^V_t)_{t>0}$ of contractions on $L^1(V, \mu)$.\\
\centerline{}
{\bf Step 3:} We claim that $(\overline{T}^V_t)_{t > 0}$ is sub-Markovian, i.e. we have $0 \leq \overline{T}^V_t f \leq 1$ if $f \in L^1(\mathbb{R}^d, \widehat{\mu})$ satisfies $0 \leq f \leq 1$. \\
Let $(\overline{G}_{\alpha}^V)_{\alpha>0}$ be the associated resolvent.  Observe that by construction 
$$
D(L^{0,V})_b \subset D(\overline{L}^V)  \text{ densely with respect to the graph norm } \| \cdot \|_{D(\overline{L}^V)}.
$$
Let $u \in D(\overline{L}^V)$ and take $u_n \in D(L^{0,V})_b$, $n\ge 1$, satisfying $\lim_{n \rightarrow \infty}u_n =u$ in $D(\overline{L}^V)$ and $\lim_{n \rightarrow \infty} u_n =u$, a.e. on $V$. Let $\varepsilon>0$ and $\varphi_{\varepsilon}$ be as in Step 1. Then by \eqref{subinvariant1}
$$
\int_{V} \overline{L}^V u \,1_{\{ u>1 \}} d \widehat{\mu} =\lim_{\varepsilon \rightarrow 0+}\int_{V} \overline{L}^V  u \, \varphi_{\varepsilon}'(u) d \widehat{\mu} \leq 0.
$$
Let $f \in L^1(V,  \widehat{\mu})$ and $u:= \alpha \overline{G}^V_{\alpha} f \in D(\overline{L}^V)$. If $f \leq 1$, then
\begin{eqnarray*}
\alpha \int_V u 1_{\{ u >1 \}} d \widehat{\mu} \leq \int_V (\alpha u - \overline{L}^V u )1_{\{ u >1 \}} d \widehat{\mu} = \alpha \int_V f 1_{\{ u >1 \}} d  \widehat{\mu} \leq \alpha \int_V 1_{\{ u >1 \}} d \widehat{\mu}. 
\end{eqnarray*}
Therefore, $\alpha \int_V (u-1)1_{\{ u >1 \}} d \widehat{\mu} \leq 0$, which implies $u \leq 1$. If $f \geq 0$, then $-nf \leq 1$ for all $n \in \N$, and hence $-nu \leq 1$ for all $n \in \N$. Thus $u \geq 0$. Therefore  $(\overline{G}^V_{\alpha})_{\alpha>0}$ is sub-Markovian. Since $\overline{T}^V_t u = \displaystyle \lim_{\alpha \rightarrow \infty} \exp\big( t \alpha ( \alpha \overline{G}^V_{\alpha}u -u )   \big)$ in $L^1(V, \widehat{\mu})$ for any $u\in L^1(V, \widehat{\mu})$ by the proof of Hille-Yosida (cf. \cite[I. Theorem 1.12]{MR}), we obtain that $(\overline{T}^V_t)_{t > 0}$ is sub-Markovian too.\\
\centerline{}
(ii) \; {\bf Step 1: }\; We claim that $D(\overline{L}^{0,V})_b \subset D(\overline{L}^V) \text{ \;and\; } \overline{L}^V u = \overline{L}^{0,V} u + \langle   \mathbf{B}, \nabla u \rangle, \; u \in D(\overline{L}^{0,V})_b$.\\
 Let $u \in D(\overline{L}^{0,V})_b$. \;Since $(T_t^{0,V})_{t>0}$ is an analytic semigroup, $T_t^{0,V} u \in D(L^{0,V})_b \subset D(\overline{L}^V)$ and $\overline{L}^V T_t^{0,V} u = L^{0,V} T_t^{0,V} u + \langle \mathbf{B}, \nabla T_t^{0,V}u \rangle  = \overline{T}_t^{0,V} \overline{L}^{0,V} u + \langle  \mathbf{B}, \nabla T_t^{0,V} u  \rangle$. By Lemma \ref{existlem}(ii), $\lim_{t \rightarrow 0+}T^{0,V}_t u =u$ in $\widehat{H}^{1,2}_0(V, \mu)$, which implies that 
\begin{equation*}
\lim_{t \rightarrow 0+} \overline{L}^V T_t^{0,V} u = \overline{L}^{0,V} u + \langle \mathbf{B}, \nabla u \rangle \quad \text{ in }\; L^1(V, \widehat{\mu}),
\end{equation*}
because $\psi \bold{B} \in L^2_{loc}(\mathbb{R}^d, \mathbb{R}^d, \mu)$. \,Since $\lim_{t \rightarrow 0+} T^{0,V}_t u = \lim_{t \rightarrow 0+} \overline{T}^{0,V}_t u =u$\;  in \,$L^1(V, \widehat{\mu})$ and $(\overline{L}^V, D(\overline{L}^V))$ is a closed operator on $L^1(V, \widehat{\mu})$, we obtain
$$
u \in D(\overline{L}^V), \quad  \; \overline{L}^V u = \overline{L}^{0,V} u + \langle \mathbf{B}, \nabla u \rangle \quad \text{ on \;$L^1(V, \widehat{\mu})$}.
$$ 
Thus Step 1 is proved.\\
{\bf Step 2: } Let $u \in D(\overline{L}^V)_b$ and take $u_n \in D(L^{0,V})_b$, $n\ge 1$,  as in the proof of (i) Step 3. Let $M_1, M_2>0$ be such that $\|u\|_{L^{\infty}(V)}<M_1 < M_2$.  Then, we claim that
\begin{equation} \label{imlemce}
\lim_{n \rightarrow \infty} \int_{\{ M_1 \leq |u_n| \leq M_2 \}} \langle \widehat{A} \nabla u_n, \nabla u_n \rangle d \widehat{\mu} = 0.
\end{equation}
Indeed, let $\varphi \in C^1(\R)$ be such that $\varphi'(t)= (t-M_1)^+ \wedge (M_2-M_1)$ with $\varphi(0)=0$. Then by Lemma \ref{existlem}(i) and (iv), we have $\varphi (u_n) \in \widehat{H}_0^{1,2}(V, \mu)$. Observe that $\varphi'(u) = 0$, a.e. on $V$ and
\begin{eqnarray*}
&&\int_{\{ M_1 \leq u_n \leq M_2 \}} \langle \widehat{A} \nabla u_n, \nabla u_n \rangle d \widehat{\mu}  = \int_{V} \big \langle \widehat{A} \nabla u_n, \nabla \varphi'(u_n) \big \rangle d \widehat{\mu}  \\
&&= \mathcal{E}^{0,V}(u_n, \varphi'(u_n)) = -\int_V L^{0,V} u_n \, \varphi'(u_n) d \widehat{\mu} \\
&&= - \int_V L^{0,V} u_n \,\varphi '(u_n)d  \widehat{\mu} - \int_V \langle \mathbf{B}, \nabla \varphi(u_n) \rangle d \widehat{\mu}  \\
&&= - \int_V L^V u_n\, \varphi '(u_n)d \widehat{\mu} \longrightarrow - \int_V  \overline{L}^V u\, \varphi '(u)d  \widehat{\mu}  =  0, \quad \text{ as } n \rightarrow \infty,
\end{eqnarray*}
where the convergence of the last limit holds by Lebesgue's Theorem, since
$$
\lim_{n \rightarrow \infty} \varphi'(u_n) = \varphi'(u) = 0, \; \text{ a.e. and boundedly on $\R^d$ }
$$
and
\begin{eqnarray*}
&&\left | \int_V L^V u_n\cdot \varphi '(u_n)d \widehat{\mu} - \int_V \overline{L}^V u\cdot \varphi '(u)d  \widehat{\mu} \right | \\
&\leq& \|\varphi' \|_{L^{\infty}(V)}\int_{V}   \big|  (L^V u_n-\overline{L}^V u)\big|  d \widehat{\mu} + \int_{V} |\overline{L}^V u|\cdot |\varphi'(u_n)-\varphi'(u)| d \widehat{\mu} \\
&& \; \longrightarrow 0 \; \text{ as } n \rightarrow \infty.
\end{eqnarray*}
Similarly, 
$$
\int_{\{-M_2 \leq u_n \leq -M_1\}} \langle \widehat{A} \nabla u_n, \nabla u_n \rangle d \widehat{\mu} =\int_{\{M_1 \leq -u_n \leq M_2\}} \langle \widehat{A} \nabla (-u_n), \nabla (-u_n) \rangle d \widehat{\mu} =0,
$$
hence \eqref{imlemce} is proved. \\
\centerline{}
{\bf Step 3: } Let $u \in D(\overline{L}^V)_b$ and take $u_n \in D(L^{0,V})_b$, $n\ge 1$, as in the proof of (i) Step 3.  Let $\varphi \in C_0^2(\R)$ be such that $\varphi(t) = t$ if $|t|< \|u\|_{L^{\infty}(V)}+1$ and $\varphi(t) = 0$ if $|t| \geq \|u\|_{L^{\infty}(V)}+2$.  Using  Step 2 with $M_1 = \|u\|_{L^{\infty}(V)}+1$, $M_2=\|u\|_{L^{\infty}(V)}+2$, Lemma \ref{existlem}(iv), and Lebesgue's Theorem
\begin{eqnarray*}
\overline{L}^V \varphi (u_n)  = \varphi'(u_n)L^V u_n + \frac12\varphi''(u_n) \langle \widehat{A} \nabla u_n, \nabla u_n \rangle \longrightarrow \overline{L}^V u \;\; \text{ in } L^1(V, \widehat{\mu}) \; \text{ as } n \rightarrow \infty.
\end{eqnarray*}
Therefore,
\begin{eqnarray*}
&&\mathcal{E}^{0,V}(\varphi(u_n)-\varphi(u_m), \,\varphi(u_n)-\varphi(u_m)) \\
&& \qquad \quad = - \int_V \overline{L}^V \big(  \varphi(u_n)- \varphi(u_m)  \big)   \cdot \big(  \varphi(u_n)- \varphi(u_m)  \big) d \widehat{\mu} \\
&& \qquad \quad \leq 2 \|\varphi \|_{L^{\infty}(\R^d)}  \| \overline{L}^V \varphi(u_n) - \overline{L}^V \varphi(u_m)  \|_{L^1(V,  \widehat{\mu})} \longrightarrow 0 \;\; \text{ as } \; n,m \rightarrow \infty.
\end{eqnarray*}
Thus $\lim_{n \rightarrow \infty} \varphi(u_n) = u$  \,in\, $\widehat{H}_0^{1,2}(V, \mu)$ by the completeness of $\widehat{H}_0^{1,2}(V, \mu)$. Then, for any $v \in \widehat{H}_0^{1,2}(V, \mu)_b$, 
\begin{eqnarray*}
&&\mathcal{E}^{0,V}(u,v) - \int_{V} \langle  \mathbf{B}, \nabla u  \rangle \, v d \widehat{\mu} = \lim_{n \rightarrow \infty} (\mathcal{E}^{0,V}\left(\varphi(u_n), v \right) - \int_{V} \langle  \mathbf{B}, \nabla \varphi(u_n) \rangle d \widehat{\mu}  \\
&& \qquad \qquad = - \lim_{n \rightarrow \infty} \int_{V} \overline{L}^V \varphi(u_n)  \cdot v d \widehat{\mu} = - \int_{V} \overline{L}^V u \cdot v d \widehat{\mu},
\end{eqnarray*}
which completes the proof of (ii).\\
(iii)
Replacing $u$ and $v$ by $\overline{G}_{\alpha}^V f$ and $\varphi$, respectively and using the fact that $(\alpha-\overline{L}^V) \overline{G}_{\alpha}^V f = f$, the assertion follows.
\end{proof}

\begin{remark} \label{remimpco}
Assume {\bf (A)}. Let $V$ be a bounded open subset of $\R^d$. Define 
$$
{L'^{,V}} u:= L^{0,V}+\langle -\mathbf{B}, \nabla u \rangle, \quad u \in D(L^{0,V})_b.
$$
Note that $-\mathbf{B}$ has the same structural properties as $\mathbf{B}$ since $\psi \mathbf{B} \in L_{loc}^2(\R^d, \R^d, \mu)$  and \eqref{divfree} holds. Thus Proposition \ref{first1} holds equally with $\mathbf{B}$ replaced by $-\mathbf{B}$. In particular, there exists a  closed extension $(\overline{L}'^{,V}, D(\overline{L}'^{,V}))$ of $({L'^{,V}}, D(L^{0,V})_b)$ on $L^1(V, \widehat{\mu})$, which generates a strongly continuous sub-Markovian resolvent of contractions $(\overline{G}_{\alpha}'^{,V} )_{\alpha>0}$ on $L^1(V, \widehat{\mu})$ and
\begin{equation*}
\mathcal{E}^{0,V}(u,v) + \int_{V} \langle \mathbf{B}, \nabla u \rangle v d \widehat{\mu}= - \int_{V} \overline{L}'^{,V} u\, \cdot v d\widehat{\mu}, \quad u \in D(\overline{L}'^{,V} )_b, \; v \in \widehat{H}^{1,2}_0(V, \mu).
\end{equation*}
Let $(\overline{G}_{\alpha}^V)_{\alpha>0}$ and $(\overline{G}_{\alpha}'^{,V})_{\alpha>0}$ 
 be the strongly continuous sub-Markovian resolvents of contractions on $L^1(V, \widehat{\mu})$ associated with
 $(\overline{L}^V, D(\overline{L}^V))$ and  $(\overline{L}'^{,V}, D(\overline{L}'^{,V}))$, respectively.
 By Riesz-Thorin interpolation, $(\overline{G}_{\alpha}^V)_{\alpha>0}|_{L^1(V, \widehat{\mu})_b}$ and $(\overline{G}_{\alpha}'^{, V})_{\alpha>0}|_{L^1(V, \widehat{\mu})_b}$ are extended to strongly continuous sub-Markovian resolvents of contractions on $L^2(V, \widehat{\mu})$, denoted by
$(G^{V}_{\alpha})_{\alpha>0}$ and $(G'^{, V}_{\alpha})_{\alpha>0}$, respectively. Let  $(L^V, D(L^V))$ and $(L'^{,V}, D(L'^{,V}))$ be the generators on 
$L^2(V, \widehat{\mu})$ associated with $(G^{V}_{\alpha})_{\alpha>0}$ and $(G'^{, V}_{\alpha})_{\alpha>0}$, respectively. Then, using the same arguments as in the proof of Proposition \ref{extlempt}, we obtain that
$(L^V, D(L^V)) \subset (\overline{L}^V, D(\overline{L}^V))$ and $(L'^{,V}, D(L'^{,V})) \subset (\overline{L}'^{,V}, D(\overline{L}'^{,V}))$, hence
for any $u \in D(L^V)_b$, $v \in D(L'^{,V})_b$
\begin{eqnarray}
-\int_{V} L^V u\cdot v d \widehat{\mu} &=& \mathcal{E}^{0,V}(u,v) - \int_{V} \langle \mathbf{B}, \nabla u \rangle vd \widehat{\mu} \nonumber \\
&=& \mathcal{E}^{0,V}(v,u) + \int_{V} \langle \mathbf{B}, \nabla v \rangle u \, d \widehat{\mu}  \nonumber\\
&=& - \int_{V}  L'^{,V} v \cdot ud \widehat{\mu}. \label{fewstard}
\end{eqnarray}
Thus, for any $f, g \in L^{\infty}(V)$,
\begin{eqnarray}
\int_{V} G_{\alpha}^V f \cdot g d\mu  &=& \int_{V} G_{\alpha}^V f \cdot (\alpha - L'^{,V} )\, G_{\alpha}'^{,V} g \,d \widehat{\mu} \nonumber  \\
&\underset{\rm{by } \eqref{fewstard}}{=}& \int_{V} (\alpha-L^V) G^V_{\alpha} f \cdot G_{\alpha}'^{,V} g \,d\widehat{\mu} \nonumber  \\
&=& \int_{V} f \cdot G_{\alpha}'^{,V} g \,d \widehat{\mu}. \label{adjoinidko}
\end{eqnarray}
By the denseness of $L^{\infty}(V)$ in $L^2(V,\widehat{\mu})$, \eqref{adjoinidko} extends to all $f,g \in L^2(V, \widehat{\mu})$.  Thus for each $\alpha>0$,\; $G_{\alpha}'^{,V}$ is the adjoint operator of $G_{\alpha}^V$ on $L^2(V, \widehat{\mu})$.
\end{remark}
Now let $V$ be a bounded open subset of $\R^d$.
Denote by $(\overline{G}_{\alpha}^V)_{\alpha>0}$ the resolvent associated with $(\overline{L}^V, D(\overline{L}^V))$ on $L^1(V, \widehat{\mu})$. Then $(\overline{G}^V_{\alpha})_{\alpha>0}$ can be extended  to $L^1(\R^d, \widehat{\mu})$ by
\begin{equation} \label{extdefpj}
\overline{G}_{\alpha}^{V} f := 
\; \; \left\{\begin{matrix}
&\overline{G}^V_{\alpha} (f 1_V) &\text{ on }\; V \\ 
&\quad 0   &\qquad \text{ on } \, \R^d \setminus V,
\end{matrix}\right. \qquad \;\;
f \in L^1(\R^d, \widehat{\mu}).
\end{equation}
Let $g \in L^1(\R^d, \widehat{\mu})_b$. Then $\overline{G}^V_{\alpha}(g1_V)\in D(\overline{L}^V)_b \subset \widehat{H}^{1,2}_0(V, \mu)$, hence $\overline{G}^V_{\alpha} g \in \widehat{H}^{1,2}_0(V, \mu)$. Note that if $u \in D(\mathcal{E}^{0,V})$, then by definition it holds $u \in D(\mathcal{E}^0)$ and $\mathcal{E}^{0,V}(u,u) = \mathcal{E}^{0}(u,u)$. Therefore, we obtain that
\begin{equation} \label{fkppmeew}
\mathcal{E}^0(\overline{G}_{\alpha}^{V}g,  \, \overline{G}_{\alpha}^{V}g ) = \mathcal{E}^{0,V}\big(\overline{G}_{\alpha}^{V}(g1_{V}),  \, \overline{G}_{\alpha}^{V} (g1_{V}) \big).
\end{equation}

\begin{lemma} \label{increlem}
Assume {\bf (A)}.
Let $V_1$, $V_2$ be bounded open subsets of $\R^d$ and $\overline{V}_1 \subset V_2$. Let $u \in L^1(\R^d, \widehat{\mu})$, $u \geq 0$ and $\alpha>0$. Then $\overline{G}^{V_1}_{\alpha} u \leq \overline{G}^{V_2}_{\alpha} u$.
\end{lemma}
\begin{proof}
By approximation in $L^1(\R^d, \widehat{\mu})$, we may assume $u \in L^1(\R^d, \widehat{\mu})_b$. Let $w_{\alpha}:= \overline{G}_{\alpha}^{V_1} u- \overline{G}^{V_2}_{\alpha} u$.  Then, since $\widehat{H}_0^{1,2}(V_1, \mu)\subset \widehat{H}_0^{1,2}(V_2, \mu)$, clearly $w_{\alpha} \in \widehat{H}_0^{1,2}(V_2, \mu)$. Observe that $w^+_{\alpha} \leq \overline{G}_{\alpha}^{V_1}u$ on $\R^d$, so that $w^+_{\alpha} \in \widehat{H}^{1,2}_0(V_1, \mu)$ by Lemma \ref{tecpjp}.
By Lemma \ref{teclemma}, we obtain
\begin{equation*} \label{evpwoek}
\int_{V_2} \langle \mathbf{B}, \nabla w_{\alpha} \rangle w_{\alpha}^+ d \widehat{\mu} = \int_{V_2} \langle \mathbf{B}, \nabla w_{\alpha}^+ \rangle w^{+}_{\alpha}d  \widehat{\mu} = 0.
\end{equation*}
Since $\mathcal{E}^{0,V_2}$ is a symmetric Dirichlet form, $\mathcal{E}^{0,V_2}(w_{\alpha}^{-}, w^+_{\alpha}) = \mathcal{E}^{0,V_2}(w_{\alpha}^{+}, w^-_{\alpha}) \leq 0$.
Therefore
\begin{eqnarray*}
\mathcal{E}^{0,V_2}( w^+_{\alpha},  w^+_{\alpha}) &\leq& \mathcal{E}^{0,V_2}_{\alpha}( w_{\alpha},  w^+_{\alpha}) - \int_{V_2} \langle \mathbf{B}, \nabla w_{\alpha} \rangle w_{\alpha}^+ d \widehat{\mu} \\
&\leq& \Big( \mathcal{E}^{0,V_1}_{\alpha}( \overline{G}_{\alpha}^{V_1} u,  w^+_{\alpha}) - \int_{V_1} \langle \mathbf{B}, \nabla \overline{G}_{\alpha}^{V_1}u \rangle w_{\alpha}^+ d \widehat{\mu} \Big) \\
&& \qquad \qquad - \Big(  \mathcal{E}^{0,V_2}_{\alpha}( \overline{G}_{\alpha}^{V_2} u,  w^+_{\alpha}) - \int_{V_2} \langle \mathbf{B}, \nabla \overline{G}_{\alpha}^{V_2} u \rangle w_{\alpha}^+ d \widehat{\mu} \Big) \\
&\leq& \int_{V_1}  (\alpha- \overline{L}^{V_1}) \overline{G}_{\alpha}^{V_1} u \, w^{+}_{\alpha} d \widehat{\mu} - \int_{V_2}  (\alpha- \overline{L}^{V_2}) \overline{G}_{\alpha}^{V_2} u \, w^{+}_{\alpha} d \widehat{\mu} \\
&=& \int_{V_1} u w^{+}_{\alpha} d \widehat{\mu} -\int_{V_2} u w^{+}_{\alpha} d \widehat{\mu} \le 0.
\end{eqnarray*}
Therefore $w_{\alpha}^+= 0$ in $\R^d$, hence $\overline{G}_{\alpha}^{V_1} u \leq   \overline{G}_{\alpha}^{V_2} u$ a.e. on $\R^d$.
\end{proof}\\
\centerline{}
We are now going to define a family of operators $(\overline{G}_{\alpha})_{\alpha>0}$ which will turn out to be a strongly continuous contraction resolvent on $L^1(\R^d, \widehat{\mu})$ that is associated to a closed extension  of $(L, D(L^0)_{0,b})$ (hence also to a closed extension  of $(L, C_0^{\infty}(\mathbb{R}^d))$). \\
First, let $(V_n)_{n \geq 1}$ be an arbitrary but fixed family of bounded open subsets of $\R^d$ satisfying $\overline{V}_n \subset V_{n+1}$ and $\R^d= \bigcup_{n \geq 1}  V_n$, and let $f \in L^1(\R^d, \widehat{\mu})$ with $f \geq 0$ \, a.e. Using Lemma \ref{increlem}, we can define for any $\alpha>0$
$$
\overline{G}_{\alpha} f := \lim_{n \rightarrow \infty} \overline{G}_{\alpha}^{V_n} f \quad \;  \text{a.e. on } \R^d. 
$$
Using the $L^1$-contraction property, $\int_{\R^d} \alpha \overline{G}_{\alpha}^{V_n} f \,d \widehat{\mu}  = \int_{V_n} \alpha \overline{G}_{\alpha}^{V_n} (f1_{V_n})\,d \widehat{\mu} \leq \int_{V_n} f d \widehat{\mu} \leq \int_{\R^d} f d \widehat{\mu}$. Thus by monotone integration, $\overline{G}_{\alpha} f \in L^1(\R^d, \widehat{\mu})$ with
\begin{eqnarray} \label{contracl1}
\int_{\R^d} \alpha \overline{G}_{\alpha} f d \widehat{\mu} \leq \int_{\R^d} f d \widehat{\mu},
\end{eqnarray}
and by Lebesgue's Theorem, $\lim_{n \rightarrow \infty} \overline{G}_{\alpha}^{V_n} f  =  \overline{G}_{\alpha} f$ in $L^1(\R^d, \widehat{\mu})$. For general $f \in L^1(\R^d, \widehat{\mu})$, define 
\[
\overline{G}_{\alpha} f := \overline{G}_{\alpha} f^+ - \overline{G}_{\alpha} f^-.
\] 

\begin{theorem} \label{mainijcie}
Assume {\bf (A)}. Then $(\overline{G}_{\alpha})_{\alpha>0}$ is a strongly continuous contraction resolvent on $L^1(\R^d, \widehat{\mu})$ and sub-Markovian. Its generator $(\overline{L}, D(\overline{L}))$ is a  closed extension of
$$
Lu:= L^0 u + \langle \mathbf{B}, \nabla u \rangle, \quad \text{ $u \in D(L^0)_{0,b}$  }
$$
on $L^1(\R^d, \widehat{\mu})$ and generates a sub-Markovian $C_0$-semigroup of contractions $(\overline{T}_t)_{t>0}$ on $L^1(\R^d, \widehat{\mu})$. Moreover, the following properties are satisfied:

\begin{itemize}
\item[(i)] For each $\alpha>0$ and $f \in L^1(\mathbb{R}^d, \widehat{\mu})_b$, it holds $\overline{G}_{\alpha} f \in D(\mathcal{E}^0)$, 
\begin{equation} \label{weakineqsom}
\mathcal{E}^{0}_{\alpha}(\overline{G}_{\alpha} f, \overline{G}_{\alpha} f )\leq \int_{\R^d} f \overline{G}_{\alpha} f d \widehat{\mu}
\end{equation}
and
\begin{equation} \label{stampale}
\mathcal{E}_{\alpha}^{0}(\overline{G}_{\alpha}f, v) - \int_{\R^d} \langle \mathbf{B}, \nabla \overline{G}_{\alpha} f \rangle v \, d \widehat{\mu}=\int_{\R^d}fv \,d \widehat{\mu}, \quad \text{ for any $v \in H^{1,2}_0(\mathbb{R}^d, \mu)_{0,b}$}.
\end{equation}

\item[(ii)]
Let $(U_n)_{n \geq 1}$ be a family of bounded open subsets of $\R^d$ satisfying $\overline{U}_n \subset U_{n+1}$ and $\R^d= \bigcup_{n \geq 1}  U_n$. Let $\alpha>0$ and $f \in L^1(\mathbb{R}^d, \widehat{\mu})$. Then, $\lim_{n \rightarrow \infty} \overline{G}_{\alpha}^{U_n} f = (\alpha- \overline{L})^{-1} f$ in $L^1(\R^d, \widehat{\mu})$ and a.e. Moreover, for $f \in L^1(\R^d, \widehat{\mu})_b$, we have $\lim_{n \rightarrow \infty} \overline{G}_{\alpha}^{U_n}f =\overline{G}_{\alpha}f$ weakly in $D(\mathcal{E}^0)$.

\item[(iii)] $D(\overline{L})_b \subset D(\mathcal{E}^0)$ and  $\text{ for all }  u \in D(\overline{L})_b,\, v \in \widehat{H}^{1,2}_0(\R^d, \mu)_{0,b}$ it holds 
\begin{eqnarray*}
\mathcal{E}^0(u,u)  &\leq& -\int_{\R^d} \overline{L}u\cdot u d  \widehat{\mu},  \qquad \qquad \\
\hspace{-5em}\mathcal{E}^{0}(u,v) - \int_{\R^d} \langle \mathbf{B}, \nabla u \rangle v d \widehat{\mu} &=& - \int_{\R^d} \overline{L} u\cdot v d \widehat{\mu}. \quad  
\end{eqnarray*}
Moreover, $\lim_{\alpha\to \infty}\alpha \overline{G}_{\alpha}u=u$ in $D(\mathcal{E}^0)$ for any $u\in D(\overline{L})_b$.
\end{itemize}
\end{theorem}

\noindent
\begin{proof}
First we will show (i) and that $(\overline{G}_{\alpha})_{\alpha>0}$ is a strongly continuous contraction resolvent on $L^1(\R^d, \widehat{\mu})$ and sub-Markovian. Let $\alpha>0$. Then $\alpha \overline{G}_{\alpha}$ is a contraction on $L^1(\R^d, \widehat{\mu})$, since by \eqref{contracl1}
$$
\int_{\R^d} | \alpha \overline{G}_{\alpha} f | d \widehat{\mu} \leq \int_{\R^d} (\alpha \overline{G}_{\alpha} f^+ + \alpha \overline{G}_{\alpha} f^- )d \widehat{\mu} \leq \int_{\R^d} f^+ d \widehat{\mu} + \int_{\R^d} f^- d \widehat{\mu} = \int_{\R^d} |f| d \widehat{\mu}.
$$
Thus,
$$
\lim_{n \rightarrow \infty} \overline{G}_{\alpha}^{V_n} f  =  \overline{G}_{\alpha} f \; \text{ in  }L^1(\R^d, \widehat{\mu}) \ \text{ and }\ \text{a.e. on } \R^d. 
$$ 
Clearly, $(\overline{G}_{\alpha})_{\alpha>0}$ is sub-Markovian, since $(\overline{G}_{\alpha}^{V_n})_{\alpha>0}$ is sub-Markovian on $L^1(V_n, \widehat{\mu})$ for any $n \geq 1$. By the $L^1(\R^d, \widehat{\mu})$-contraction property, for any $\alpha, \beta >0$
\begin{equation*} \label{limitidt}
\lim_{n \rightarrow \infty} \| \overline{G}_{\alpha}^{V_n} \overline{G}_{\beta} f  - \overline{G}_{\alpha}^{V_n} \overline{G}_{\beta}^{V_n} f \|_{L^1(\R^d, \widehat{\mu})} \leq \lim_{n \rightarrow \infty} \frac{1}{\alpha} \| \overline{G}_{\beta} f - \overline{G}^{V_n}_{\beta} f  \|_{L^1(\R^d, \widehat{\mu})} = 0.
\end{equation*}
Using the latter
and the resolvent equation for $(\overline{G}^{V_n}_{\alpha})_{\alpha>0}$, we obtain for any $\alpha, \beta>0$
\begin{eqnarray*}
(\beta- \alpha) \overline{G}_{\alpha} \overline{G}_{\beta} f &=& \lim_{n \rightarrow \infty} (\beta- \alpha) \overline{G}_{\alpha}^{V_n} \overline{G}_{\beta} f = \lim_{n \rightarrow \infty} (\beta-\alpha) \overline{G}_{\alpha}^{V_n} \overline{G}_{\beta}^{V_n} f  \\
&=& \lim_{n \rightarrow \infty} \big (\overline{G}^{V_n}_{\alpha} f - \overline{G}_{\beta}^{V_n} f\big ) = \overline{G}_{\alpha} f - \overline{G}_{\beta} f \quad\text{ in } L^1(\R^d, \widehat{\mu}).
\end{eqnarray*}
Let $f \in L^1(\R^d, \widehat{\mu})_b$ and $\alpha>0$.\; By \eqref{stamid}, $\overline{G}_{\alpha}^{V_n}(f 1_{V_n}) \in D(\overline{L}^V)_b \subset \widehat{H}^{1,2}_0(V_n, \mu)_b$. Using \eqref{fkppmeew},
\begin{eqnarray}
\mathcal{E}^0_{\alpha}(\overline{G}_{\alpha}^{V_n}f,  \, \overline{G}_{\alpha}^{V_n}f ) &=& \mathcal{E}^{0,V_n}_{\alpha}\big(\overline{G}_{\alpha}^{V_n}(f 1_{V_n}),  \, \overline{G}_{\alpha}^{V_n}(f 1_{V_n}) \big) \nonumber \\
&=& - \int_{V_n} \overline{L}^{V_n}\overline{G}_{\alpha}^{V_n}(f 1_{V_n}) \cdot \overline{G}_{\alpha}^{V_n}(f1_{V_n}) d \widehat{\mu} + \int_{V_n} \alpha \overline{G}_{\alpha}^{V_n}(f1_{V_n}) \cdot \overline{G}_{\alpha}^{V_n}(f1_{V_n}) d \widehat{\mu} \nonumber \\
&=& \int_{V_n}  (f1_{V_n}) \cdot \overline{G}_{\alpha}^{V_n}(f1_{V_n}) d \widehat{\mu} \nonumber \\
&\leq& \int_{\R^d} f  \cdot \overline{G}_{\alpha} f d \widehat{\mu} \label{boundwq} \\
&\leq& \frac{1}{\alpha} \|f\|_{L^{\infty}(\R^d)} \|f\|_{L^1(\R^d, \widehat{\mu})}.  \nonumber
\end{eqnarray}
Observe that $\lim_{ n \rightarrow \infty} \overline{G}^{V_{n}}_{\alpha} f  = \overline{G}_{\alpha} f$  \; in $L^2(\R^d, \widehat{\mu})$ by Lebesgue's Theorem. Thus by the Banach-Alaoglu Theorem, $\overline{G}_{\alpha} f \in D(\mathcal{E}^0)$ and
there exists subsequence of $(\overline{G}_{\alpha}^{V_n} f)_{n \geq 1}$, say again $(\overline{G}_{\alpha}^{V_n} f)_{n \geq 1}$, such that
\begin{equation} \label{weakcomp}
\lim_{n \rightarrow \infty} \overline{G}^{V_{n}}_{\alpha} f = \overline{G}_{\alpha} f \quad \text{ weakly in } D(\mathcal{E}^0).
\end{equation}
Since the above line of arguments works also for any subsequence of  $(\overline{G}_{\alpha}^{U_n}f)_{n \geq 1}$, i.e. any subsequence of $(\overline{G}_{\alpha}^{U_n}f)_{n \geq 1}$ has a further subsequence that converges weakly in $D(\mathcal{E}^0)$ to $\overline{G}_{\alpha} f$, \eqref{weakcomp} holds without referring to a subsequence.\\
Since the norm $\mathcal{E}^0_{\alpha}(\cdot,\cdot)^{\frac{1}{2}}$ is lower semicontinuous with respect to weak convergence, using \eqref{boundwq}
\begin{equation*} 
\mathcal{E}^{0}_{\alpha}(\overline{G}_{\alpha} f, \overline{G}_{\alpha} f ) \leq \liminf_{n \rightarrow \infty} \mathcal{E}^0_{\alpha} (\overline{G}^{V_{n}}_{\alpha} f,\overline{G}^{V_{n}}_{\alpha} f) \leq \int_{\R^d} f \overline{G}_{\alpha} f d \widehat{\mu},
\end{equation*}
hence \eqref{weakineqsom} follows. Let $v \in \widehat{H}^{1,2}_0(\R^d, \mu)_{0,b}$. Take $N \in \mathbb{N}$ so that $\text{supp}(v) \subset V_N$.
Then, by Lemma \ref{applempn}, $v \in \widehat{H}^{1,2}_0(V_n, \mu)$ for all $n \geq N$.
 Using \eqref{weakcomp},
\begin{eqnarray*} 
&&\qquad \mathcal{E}_{\alpha}^{0}(\overline{G}_{\alpha}f, v) - \int_{\R^d} \langle \mathbf{B}, \nabla \overline{G}_{\alpha} f \rangle v \, d \widehat{\mu}  \nonumber \\
&&=\lim_{n \rightarrow \infty}\big( \mathcal{E}_{\alpha}^{0}(\overline{G}^{V_{n}}_{\alpha}f, v) - \int_{\R^d} \langle \psi \mathbf{B}, \nabla \overline{G}^{V_{n}}_{\alpha} f \rangle v \, d\mu \big)   \nonumber \\
&&=\lim_{n \rightarrow \infty}\big( \mathcal{E}_{\alpha}^{0,V_{n}}(\overline{G}^{V_{n}}_{\alpha}(f 1_{V_n}), v) - \int_{V_n} \big \langle \mathbf{B}, \nabla \overline{G}^{V_{n}}_{\alpha} (f1_{V_n}) \big \rangle\, v \, d \widehat{\mu} \big) \nonumber \\
&&= \lim_{n \rightarrow \infty} \int_{V_n} (\alpha-\overline{L}^{V_n}) \overline{G}^{V_n}_{\alpha} (f1_{V_n}) \cdot v d \widehat{\mu} = \lim_{n \rightarrow \infty} \int_{V_n} f v  d \widehat{\mu}= \int_{\R^d} f v d \widehat{\mu},
\end{eqnarray*}
and hence \eqref{stampale} follows. Let $u \in D(L^0)_{0,b}$ be given and take $j \in \N$ satisfying $\text{supp}\, u \subset V_j$. Then by Lemma \ref{applempn}, $u \in \widehat{H}^{1,2}_0(V_j, \mu)$. Observe that $\text{supp}\, (Lu) \subset V_j$ and for any $n \geq j$, $u1_{V_n} \in D(L^{0, V_n})_b$, $L^{V_n} (u1_{V_n})  = L u$ on  $V_n$, hence $\overline{G}^{V_n}_{\alpha} (\alpha-L) u  = u$ on $\R^d$. Letting $n \rightarrow \infty$ we have
\begin{equation} \label{extiden}
 u = \overline{G}_{\alpha} (\alpha-L) u.
\end{equation}
Note that
\begin{eqnarray}
\| \alpha \overline{G}_{\alpha} u  - u \|_{L^1(\R^d, \widehat{\mu})} 
&=& \left \| \alpha \overline{G}_{\alpha} u - \overline{G}_{\alpha} (\alpha-L) u  \right \|_{L^1(\R^d, \widehat{\mu})} \ =\  \left \| \overline{G}_{\alpha} L u \right \|_{L^1(\R^d, \widehat{\mu})}  \nonumber \\
&\leq& \frac{1}{\alpha} \| Lu \|_{L^1(\R^d, \widehat{\mu})} \longrightarrow 0 \;\; \text{ as } \alpha \rightarrow \infty. \label{stconr}
\end{eqnarray}
Since $C_0^{\infty}(\R^d) \subset D(L^0)_{0,b}$, the strong continuity of $(\overline{G}_{\alpha})_{\alpha>0}$ on $D(L^0)_{0,b}$ of \eqref{stconr} extends to all $u \in L^1(\R^d, \widehat{\mu})$, which shows the strong continuity of $(\overline{G}_{\alpha})_{\alpha>0}$ on $L^1(\R^d, \widehat{\mu})$.
Let $(\overline{L}, D(\overline{L}))$ be the generator of $(\overline{G}_{\alpha})_{\alpha>0}$. Then \eqref{extiden} implies $\overline{L} u = Lu$  for all $u \in D(L^0)_{0,b}$. Thus $(\overline{L}, D(\overline{L}))$ is a closed extension of $(L, D(L^0)_{0,b})$ on $L^1(\R^d, \widehat{\mu})$. By the Hille-Yosida Theorem, $(\overline{L}, D(\overline{L}))$ generates a $C_0$-semigroup of contractions $(\overline{T}_t)_{t> 0}$ on $L^1(\R^d, \widehat{\mu})$. \\
Since $\overline{T}_t u = \lim_{\alpha \rightarrow \infty} \exp \big( t\alpha (\alpha \overline{G}_{\alpha}u - u ) \big) $ in $L^1(\R^d, \widehat{\mu})$,  $(\overline{T}_t)_{t>0}$ is also sub-Markovian, hence all statements including and up to (i) are proven.\\ \\
(ii)
Let $(U_n)_{n \geq 1}$ be a family of bounded open subsets of $\R^d$ such that $\overline{U}_{n} \subset U_{n+1}$ for all $n \in \N$ and $\R^d  = \bigcup_{n \geq 1} U_n$. Let $f \in L^1(\R^d, \widehat{\mu})$ with $f \geq 0$. By the compactness of $\overline{V}_n$ in $\R^d$, there exists $n_0 \in \N$ such that $\overline{V}_n \subset U_{n_0}$, so that $\overline{G}^{V_n} f \leq \overline{G}^{U_{n_0}} f \leq \lim_{n \rightarrow \infty} \overline{G}^{U_n}_{\alpha}f$. 
Letting $n \rightarrow \infty$, we obtain $\overline{G}_{\alpha} f \leq \lim_{n \rightarrow \infty} \overline{G}^{U_n}_{\alpha}f$. Similarly we have $\lim_{n \rightarrow \infty} \overline{G}^{U_n}_{\alpha}f \leq \overline{G}_{\alpha} f$. The rest follows as in  (i).
\\ \\
(iii)  Let $u \in D(\overline{L})_b$ and $(\alpha_n)_{n \geq 1}$ be any sequence of strictly positive real numbers with $\lim_{n \rightarrow \infty}\alpha_n = \infty$. Then by \eqref{weakcomp}, $\alpha_n \overline{G}_{\alpha_n} u \in D(\mathcal{E}^0)$ and by \eqref{weakineqsom}
\begin{eqnarray}
\mathcal{E}^{0}(\alpha_n \overline{G}_{\alpha_n} u, \alpha _n\overline{G}_{\alpha_n} u  ) &\leq& \int_{\R^d} \alpha_n u \cdot \alpha_n  \overline{G}_{\alpha_n} u d \widehat{\mu} - \alpha_n \int_{\R^d} \alpha_n \overline{G}_{\alpha_n} u \cdot \alpha_n \overline{G}_{\alpha_n} u \,d \widehat{\mu}  \nonumber \\
&=& \int_{\R^d} \alpha_n \big( u - \alpha_n \overline{G}_{\alpha_n} u  \big) \cdot \alpha_n \overline{G}_{\alpha_n} u \, d \widehat{\mu}  \nonumber \\
&=& \int_{\R^d} - \alpha_n \overline{L}\,  \overline{G}_{\alpha_n} u \cdot \alpha_n \overline{G}_{\alpha_n} u \,d \widehat{\mu} \nonumber \\
&=& \int_{\R^d} - \alpha_n \overline{G}_{\alpha_n} \overline{L} u \cdot \alpha_n \overline{G}_{\alpha_n} u d \widehat{\mu} \label{bddjwnvr14} \\
&\leq& \| \overline{L} u \|_{L^1(\R^d, \widehat{\mu})} \|u \|_{L^{\infty}(\R^d)}. \nonumber
\end{eqnarray}
Therefore $\sup_{n\ge 1} \mathcal{E}^{0}(\alpha_n \overline{G}_{\alpha_n}u, \alpha_n \overline{G}_{\alpha_n}u ) < \infty$. By the Banach-Alaoglu theorem, $u \in D(\mathcal{E}^0)$ and there exists a subsequence of $(\alpha_n \overline{G}_{\alpha_n}u)_{n\ge 1}$, say $(\alpha_{n_k} \overline{G}_{\alpha_{n_k}}u)_{k\ge 1}$, such $\lim_{k\rightarrow \infty} \alpha_{n_k} \overline{G}_{\alpha_{n_k}}u= u$ weakly in $D(\mathcal{E}^0)$. Since the argument works for any subsequence, we obtain that $\lim_{\alpha \rightarrow \infty} \alpha \overline{G}_{\alpha} u = u$ weakly in $D(\mathcal{E}^0)$. Consequently,  by the property of weak convergence, \eqref{bddjwnvr14} and Lebesgue's Theorem,
\begin{eqnarray*}
\mathcal{E}^{0}(u,u) &\leq& \liminf_{\alpha \rightarrow \infty} \mathcal{E}^{0}(\alpha \overline{G}_{\alpha} u, \alpha \overline{G}_{\alpha} u  )  \leq \liminf_{\alpha \rightarrow \infty} \big(-\int_{\R^d} \alpha \overline{G}_{\alpha} \overline{L} u \cdot \alpha \overline{G}_{\alpha} u \, d \widehat{\mu}\big)\\
&=& -\int_{\R^d} \overline{L}u \, u \,d \widehat{\mu}. \;\; \;
\end{eqnarray*}
The latter further implies
\begin{eqnarray*}
\mathcal{E}^0(\alpha \overline{G}_{\alpha} u - u, \alpha \overline{G}_{\alpha} u - u ) &\leq& -\int_{\R^d} \overline{L} (\alpha \overline{G}_{\alpha} u - u) \cdot (\alpha \overline{G}_{\alpha} u - u) d \widehat{\mu} \\
&\leq& 2\|u\|_{L^{\infty}(\R^d)} \|\alpha \overline{G}_{\alpha} \overline{L}u -\overline{L}u \|_{L^1(\R^d, \widehat{\mu})} \\
&& \;\longrightarrow 0 \quad \text{ as } \alpha \rightarrow \infty.
\end{eqnarray*}
Finally, if $v \in \widehat{H}^{1,2}(\R^d, \mu)_{0,b}$, then by \eqref{stampale}
\begin{eqnarray*}
&&\mathcal{E}^0(u,v) - \int_{\R^d} \langle \mathbf{B}, \nabla u \rangle v d \widehat{\mu} = \lim_{\alpha \rightarrow \infty} \big( \mathcal{E}^0(\alpha \overline{G}_{\alpha} u, v)- \int_{\R^d} \langle  \mathbf{B}, \nabla \alpha \overline{G}_{\alpha} u   \rangle v d\widehat{\mu}  \big) \\
&& =  \lim_{\alpha\rightarrow \infty} \big( \mathcal{E}_{\alpha}^0(\alpha \overline{G}_{\alpha} u, v)- \int_{\R^d} \langle  \mathbf{B}, \nabla \alpha \overline{G}_{\alpha} u   \rangle v d \widehat{\mu}   - \alpha \int_{\R^d} \alpha \overline{G}_{\alpha} u \cdot v d \widehat{\mu} \big) \\
&& = \lim_{\alpha \rightarrow \infty} \int_{\R^d} \alpha(u- \alpha \overline{G}_{\alpha} u) v d \widehat{\mu} =  \lim_{\alpha \rightarrow \infty} \int_{\R^d} -\alpha \overline{G}_{\alpha} \overline{L} u \cdot v d \widehat{\mu}  = - \int_{\R^d} \overline{L} u \cdot v d \widehat{\mu},
\end{eqnarray*}
as desired.
\end{proof}

\begin{remark} \label{resestdual}
In the same manner as in Theorem \ref{mainijcie}, one can construct an $L^1(\R^d, \widehat{\mu})$ closed extension $(\overline{L}', D(\overline{L}'))$ of $L^0 u + \langle -\mathbf{B}, \nabla u \rangle$, $u \in D(L^0)_{0,b}$ from the associated strongly continuous sub-Markovian resolvent of contractions $(\overline{G}_{\alpha}' )_{\alpha>0}$ on $L^1(\R^d, \widehat{\mu})$. Let  $(U_n)_{n \geq 1}$ be as in Theorem \ref{mainijcie}(ii). Observe that by Remark \ref{remimpco} 
\begin{equation} \label{localresid}
\int_{\R^d} \overline{G}_{\alpha}^{U_n} u \cdot v \, d \widehat{\mu}  = \int_{\R^d} u \cdot {\overline{G}_{\alpha}'^{,U_n}} v \,d \widehat{\mu}, \;\; \text{ for all } u,v \in L^1(\R^d, \widehat{\mu})_b,
\end{equation}
where $(\overline{G}_{\alpha}'^{,U_n})_{\alpha>0}$ is the resolvent associated to $(\overline{L}'^{,U_n}, D(\overline{L}'^{,U_n}))$ on $L^1(U_n, \widehat{\mu})$, which is trivially extended to $\R^d$ as in \eqref{extdefpj}.
\;Letting $n \rightarrow \infty$ in \eqref{localresid}, 
\begin{equation*} \label{glbresid}
\int_{\R^d} \overline{G}_{\alpha} u \, v \, d \widehat{\mu}  = \int_{\R^d} u \, \overline{G}_{\alpha}' v d \widehat{\mu}, \;\; \text{ for all } u,v \in L^1(\R^d, \widehat{\mu})_b. 
\end{equation*}
\end{remark}
\bigskip
The following Theorem \ref{pojjjde} which shows that $D(\overline{L})_b$ is an algebra is one of the main properties that are needed to construct a Hunt process corresponding to the strict capacity (see, {\bf SD3} in \cite{Tr5}) and will be used later. The proof of Theorem \ref{pojjjde} is based on a similar line of arguments as in \cite[Remark 1.7 (iii)]{St99}, but we include its proof since the framework considered here is different and we have to check in detail some approximation arguments.
\begin{theorem} \label{pojjjde}
Assume {\bf (A)}. Let $(\overline{L}, D(\overline{L}))$ be defined in Theorem \ref{mainijcie}. 
Then, $D(\overline{L})_b$ is an algebra and $\overline{L}u^2 = 2u \overline{L}u +  \langle \widehat{A} \nabla u, \nabla u \rangle$ for any $u \in D(\overline{L})_b$.
\end{theorem}
\begin{proof}
Let $u \in D(\overline{L})_b$. In order to show that $D(\overline{L})_b$ is an algebra, since $D(\overline{L})_b$ is a linear space, it suffices to show $u^2 \in D(\overline{L})_b$. Let $(\overline{L}', D(\overline{L}'))$, $(\overline{G}'_{\alpha} )_{\alpha>0}$ be as in Remark \ref{resestdual} and set $g:=2u \overline{L}u + \langle \widehat{A} \nabla u, \nabla u \rangle$. In order to show $u^2 \in D(\overline{L})_b$ and $\overline{L}u^2=g$, it suffices to show the following:
\begin{equation} \label{sufficeoce}
\int_{\R^d} (\overline{L}' \, \overline{G}'_1 h) \, u^2 d \widehat{\mu}  = \int_{\R^d} g\, \overline{G}'_1 h \,d \widehat{\mu}, \;\; \text{ for all } h \in L^1(\R^d, \widehat{\mu})_b.
\end{equation}
Indeed, if we can show \eqref{sufficeoce},
then 
\begin{eqnarray*}
\int_{\R^d} \overline{G}_1(u^2-g) \, h d \widehat{\mu}  &=& \int_{\R^d} (u^2-g) \overline{G}'_{1} h\, d \widehat{\mu} \underset{\eqref{sufficeoce}}{=} \int_{\R^d} u^2  (  \overline{G}'_{1} h - \overline{L}' \,\overline{G}_{1}' h  ) d \widehat{\mu}\\
&=& \int_{\R^d} u^2 h \,d \widehat{\mu}, \;\; \text{ for all } \; h \in L^1(\R^d,  \widehat{\mu})_b,
\end{eqnarray*}
hence $u^2  = \overline{G}_1(u^2-g) \in D(\overline{L})_b$  and $\overline{L}u^2= (1-\overline{L}) \overline{G}_1(g-u^2)- \overline{G}_1(g-u^2) = g-u^2+u^2=g$, as desired. \\
To prove \eqref{sufficeoce}, we proceed in two steps.\\
{\bf Step 1: } First assume $u = \overline{G}_1 f$ for some $f \in L^1(\R^d, \widehat{\mu})_b$. Fix $v =\overline{G}'_1 h$ for some $h \in L^1(\R^d, \widehat{\mu})_b$ with $h \geq 0$. Let $(U_n)_{n \geq 1}$ be as in Theorem \ref{mainijcie}(ii) and $u_n:= \overline{G}^{U_n}_1 f$, \;$v_n:= \overline{G}'^{, U_n}_{1} h$. 
Note that $u, v_n \in D(\mathcal{E}^0)$ and $u u_n \in D(\mathcal{E}^0)$ by Theorem \ref{mainijcie}(iii) and \cite[I. Corollary 4.15]{MR}.
Then by Proposition \ref{first1} and Theorem \ref{mainijcie},
\begin{eqnarray}
&&\int_{\R^d} (\overline{L}'^{, U_n} v_n) \, u u_n d \widehat{\mu} 
= - \mathcal{E}^0(v_n,  u u_n) - \int_{\R^d} \langle \mathbf{B}, \nabla v_n \rangle u u_n d \widehat{\mu} \nonumber \\
&=& - \frac12 \int_{\R^d} \langle \widehat{A} \nabla v_n, \nabla u \rangle u_n  d\mu- \frac12 \int_{\R^d} \langle \widehat{A} \nabla v_n, \nabla u_n \rangle u  d \widehat{\mu}  + \int_{\R^d} \langle \mathbf{B}, \nabla (u u_n) \rangle v_n d \widehat{\mu}  \nonumber \\
&=& - \frac12 \int_{\R^d} \langle \widehat{A} \nabla (v_n u_n), \nabla u    \rangle d \widehat{\mu} + \frac12 \int_{\R^d} \langle \widehat{A} \nabla u_n,  \nabla u \rangle v_n d \widehat{\mu} - \frac12\int_{\R^d} \langle \widehat{A} \nabla v_n, \nabla u_n \rangle u  d \widehat{\mu} \nonumber \\
&& \qquad +\int_{\R^d} \langle \mathbf{B}, \nabla u \rangle v_n u_n \,d \widehat{\mu}+ \int_{\R^d} \langle \mathbf{B}, \nabla u_n \rangle v_n u d \widehat{\mu} \nonumber \\
&=& - \frac12 \int_{\R^d} \langle \widehat{A} \nabla u, \nabla (v_n u_n)  \rangle d \widehat{\mu} +\int_{\R^d} \langle \mathbf{B}, \nabla u \rangle v_n u_n \,d \widehat{\mu}  + \frac12 \int_{\R^d} \langle \widehat{A} \nabla u_n, \nabla u \rangle v_n  d \widehat{\mu} \nonumber \\
&& \qquad  - \frac12 \int_{\R^d} \langle \widehat{A} \nabla u_n, \nabla (v_n u)\rangle d \widehat{\mu}+ \int_{\R^d} \langle \mathbf{B}, \nabla u_n \rangle v_n u d \widehat{\mu} + \frac12 \int_{\R^d} \langle \widehat{A} \nabla u_n, \nabla u \rangle v_n d \widehat{\mu} \nonumber \\
&=& \int_{\R^d} \overline{L} u \cdot v_n u_n d \widehat{\mu} + \int_{\R^d} \overline{L}^{U_n} u_n \cdot v_n u d \widehat{\mu} + \int_{\R^d} \langle \widehat{A} \nabla u_n, \nabla u \rangle v_n d \widehat{\mu}. \label{intepce}
\end{eqnarray}
Observe that
\begin{eqnarray}
&&  \big |\int_{\R^d} \langle \widehat{A} \nabla u, \nabla u   \rangle v d \widehat{\mu} - \int_{\R^d} \langle \widehat{A} \nabla u_n, \nabla u \rangle v_n d \widehat{\mu} \big | \nonumber \\
&& \leq \underbrace{\big | \int_{\R^d} \langle \widehat{A} \nabla (u-u_n), \nabla u \rangle v   d \widehat{\mu}  \big|}_{=:I_n} + \underbrace{\big| \int_{\R^d}  \langle \widehat{A} \nabla u_n, \nabla u    \rangle \, (v-v_n) \, d \widehat{\mu} \big| }_{=:J_n}. \label{weakconvlen} 
\end{eqnarray}
Since $\lim_{n \rightarrow \infty} u_n = u$ \, weakly in $D(\mathcal{E}^0)$ and $v$ is bounded on $\R^d$, we have $\lim_{n \rightarrow \infty} I_n =0$.
Note that $0 \leq v_n = \overline{G}'^{, \,U_n}_1 h \leq \overline{G}'_1 h=v$,\; $\sup_{n \geq 1}\mathcal{E}^0(u_n, u_n)<\infty$,\; $|v_n| \leq |v| \in L^{\infty}(\R^d)$ and 
$$
\lim_{n \rightarrow \infty}u_n = u \text{ a.e. on $\R^d$,  }
$$ 
hence we obtain from the Cauchy–Schwarz inequality,
\begin{eqnarray*}
J_n&&\leq   \big(  \int_{\R^d}  \langle \widehat{A} \nabla u_n, \nabla u_n    \rangle \, |v-v_n| \, d \widehat{\mu} \big)^{1/2} \big(  \int_{\R^d}  \langle \widehat{A} \nabla u, \nabla u    \rangle \, |v-v_n| \, d \widehat{\mu} \big)^{1/2} \\
&&\leq \sqrt{2}\|v\|_{L^{\infty}(\R^d)}^{1/2} \sup_{n \geq 1}\mathcal{E}^0(u_n, u_n)^{1/2} \big(  \int_{\R^d}  \langle \widehat{A} \nabla u, \nabla u    \rangle \, (v-v_n) \, d \widehat{\mu} \big)^{1/2} \\
&&\;\; \longrightarrow 0 \quad \text{ as } n \rightarrow \infty,
\end{eqnarray*}
where for the latter convergence to zero, which holds by Lebesgue's Theorem for which we use
$$
\big|\langle \widehat{A} \nabla u, \nabla u    \rangle \, (v-v_n)  \big| \leq 2\|v\|_{L^{\infty}(\R^d)}  \langle \widehat{A} \nabla u, \nabla u    \rangle \in L^1(\R^d, \widehat{\mu}),  \; \text{\,  a.e. on $\R^d$  }
$$
and
$$
\lim_{n \rightarrow \infty}  \langle \widehat{A} \nabla u, \nabla u    \rangle \, (v-v_n) = 0, \;\; \text{ a.e. on $\R^d$}.
$$
Therefore it follows by \eqref{weakconvlen} that
\begin{eqnarray} \label{limitcje}
\lim_{n \rightarrow \infty} \int_{\R^d} \langle \widehat{A} \nabla u_n, \nabla u \rangle v_n d \widehat{\mu} =\int_{\R^d} \langle \widehat{A} \nabla u, \nabla u \rangle v d \widehat{\mu}.
\end{eqnarray}
Using Lebesgue's Theorem 
\begin{eqnarray}
\int_{\R^d} \overline{L}' v\cdot u^2 d \widehat{\mu}  &= &\int_{\R^d} \big( \overline{G}'_1 h - h\big) u^2d \widehat{\mu} = \lim_{n \rightarrow \infty} \int_{\R^d} \big(\overline{G}'^{,  U_n}_1 h - h \big) u u_n \,d \widehat{\mu} = \lim_{n \rightarrow \infty} \int_{\R^d} (\overline{L}'^{,U_n} v_n)  u u_n d \widehat{\mu} \nonumber  \\
&\underset{\eqref{intepce}}{=}&\lim_{n \rightarrow \infty} \int_{\R^d} \overline{L} u \cdot v_n u_n d \widehat{\mu} + \int_{\R^d} (\overline{G}_1^{U_n} f -f) \cdot v_n u d \widehat{\mu} +\int_{\R^d} \langle \widehat{A} \nabla u_n, \nabla u \rangle v_n d \widehat{\mu}  \nonumber \\
&\underset{\eqref{limitcje}}{=}&  \int_{\R^d} \overline{L} u \cdot v u d \widehat{\mu} + \int_{\R^d} \overline{L}u\cdot vu d \widehat{\mu} + \int_{\R^d} \langle \widehat{A} \nabla u, \nabla u \rangle v d \widehat{\mu}  =\int_{\R^d} g v d \widehat{\mu}. \label{gvidendt}
\end{eqnarray}
In the case of general $h \in L^1(\R^d, \widehat{\mu})_b$, we also obtain \eqref{gvidendt} using $h=h^+-h^-$ and linearity.\\
\centerline{}
{\bf Step 2: } Let $ u \in D(\overline{L})_b$ be arbitrary. Set
$$
g_{\alpha}:= 2(\alpha \overline{G}_{\alpha} u) \overline{L} (\alpha \overline{G}_{\alpha} u) +  \langle \widehat{A} \nabla \alpha  \overline{G}_{\alpha} u, \nabla \alpha \overline{G}_{\alpha} u \rangle, \quad \alpha>0.
$$
By Theorem \ref{mainijcie}(iii), $\lim_{\alpha\to \infty}\alpha \overline{G}_{\alpha}u=u$ in $D(\mathcal{E}^0)$.
Moreover, since $\|\alpha G_{\alpha} u\|_{L^{\infty}(\mathbb{R}^d)} \leq \| u \|_{L^{\infty}(\mathbb{R}^d)}$ and $\lim_{\alpha \rightarrow \infty} \overline{L} (\alpha \overline{G}_{\alpha} u) = \lim_{\alpha \rightarrow \infty} \alpha \overline{G}_{\alpha} \overline{L}u = \overline{L}u$ in $L^1(\mathbb{R}^d, \widehat{\mu})$, it follows from Lebesgue's theorem that
$$
\lim_{\alpha \rightarrow \infty }2(\alpha \overline{G}_{\alpha} u) \overline{L} (\alpha \overline{G}_{\alpha} u) =2u \overline{L}u \quad \text{ in $L^1(\mathbb{R}^d, \widehat{\mu})$}.
$$
Therefore, we get $\lim_{\alpha \rightarrow \infty} g_{\alpha} = g$ \; in $L^1(\R^d, \widehat{\mu})$. Observe that by the resolvent equation
\begin{eqnarray*}
\overline{G}_{\alpha} u =\overline{G}_1\big( (1-\alpha) \overline{G}_{\alpha} u + u  \big)
\end{eqnarray*}
and  $(1-\alpha) \overline{G}_{\alpha} u +  u \in L^1(\R^d, \widehat{\mu})_b$. Replacing $u$ in \eqref{gvidendt} with $\alpha \overline{G}_{\alpha} u$
\begin{eqnarray*}
\int_{\R^d} \overline{L}'\, v\, \big(\alpha \overline{G}_{\alpha}u \big)^2 d \widehat{\mu}=\int_{\R^d} g_{\alpha} v d \widehat{\mu}.
\end{eqnarray*}
Letting $\alpha \rightarrow \infty$, we finally obtain by Lebesgue's Theorem
$$
\int_{\R^d} \overline{L}' v\cdot u^2 d \widehat{\mu}=\int_{\R^d} g v d \widehat{\mu},
$$
which concludes the proof.
\end{proof}
\centerline{}
\noindent
By Theorem \ref{mainijcie} there exists a closed extension $(\overline{L}, D(\overline{L}))$ of
$$
Lf = L^0 f + \langle \mathbf{B}, \nabla f \rangle, \;\quad  f \in D(L^0)_{0,b}
$$
on $L^1(\R^d, \widehat{\mu})$ which generates a sub-Markovian $C_0$-semigroup of contractions $(\overline{T}_t)_{t>0}$ on $L^1(\R^d, \widehat{\mu})$.  Restricting $(\overline{T}_t)_{t> 0}$ to $L^1(\R^d, \widehat{\mu})_b$, it follows by Riesz-Thorin interpolation that $(\overline{T}_t)_{t> 0}$ can be extended to a sub-Markovian $C_0$-semigroup of contractions $(T_t)_{t>0}$ on each $L^r(\R^d, \widehat{\mu})$, $r\in [1,\infty)$. Denote by $(L_r, D(L_r))$, $r \in [1, \infty)$  the corresponding closed generator with graph norm
$$
\|f\|_{D(L_r)}:=\|f\|_{L^r(\R^d, \widehat{\mu})}+ \|L_r f\|_{L^r(\R^d, \widehat{\mu})},
$$
and by $(G_{\alpha})_{\alpha>0}$ the corresponding strongly continuous sub-Markovian resolvent of contractions on $L^r(\mathbb{R}^d, \widehat{\mu})$. Also 
$(\overline{T}_t)_{t>0}$ and $(\overline{G}_{\alpha})_{\alpha>0}$ restricted to $L^1(\mathbb{R}^d, \widehat{\mu})_b$ extend to a sub-Markovian semigroup and a resolvent on $L^{\infty}(\mathbb{R}^d)$ which are denoted again by $(T_t)_{t>0}$ and $(G_{\alpha})_{\alpha>0}$, respectively.
\\
For $f \in C_0^{\infty}(\R^d)$, we have
\begin{equation}\label{definition of L}
Lf = L^0 f + \langle \mathbf{B}, \nabla f \rangle=\frac12 \text{trace}(\widehat{A} \nabla ^2 f) + \langle \beta^{\rho, A, \psi}+\mathbf{B}, \nabla f \rangle.
\end{equation}
Define
\begin{equation}\label{definition of L'}
L' f := L^0 f- \langle \mathbf{B}, \nabla f \rangle= \frac12\text{trace}(\widehat{A}\nabla^2 f)+\langle \beta^{\rho, A, \psi}-\mathbf{B}, \nabla f \rangle, \qquad f \in C_0^{\infty}(\R^d).
\end{equation}
We see that $L$ and $L'$ have the same structural properties, i.e. they are given as the sum of a symmetric second-order elliptic partial differential operator $L^0$ and a divergence-free first-order perturbation\,$\langle \mathbf{B}, \nabla \cdot \,\rangle$ \, or $\langle -\mathbf{B}, \nabla \cdot \, \rangle$, respectively, with same integrability condition\,$\psi \mathbf{B} \in L_{loc}^2(\R^d, \R^d,\mu)$. Therefore all that will be derived below and what has been derived above for $L$ holds analogously for $L'$ (see for instance  Remarks \ref{remimpco}, \ref{resestdual}). Let $(\overline{L}', D(\overline{L}'))$ be the closed extension of $(L', D(L^0)_{0,b})$ which generates a sub-Markovian $C_0$-semigroup and a resolvent of contractions on $L^1(\mathbb{R}^d, \widehat{\mu})$, denoted by $(\overline{T}'_t)_{t>0}$ and $(\overline{G}'_{\alpha})_{\alpha>0}$, respectively (cf. Theorem \ref{mainijcie} and Remark \ref{resestdual}). Analogously to the above $(\overline{G}'_{\alpha})_{\alpha>0}$ and $(\overline{T}'_t)_{t>0}$ restricted to $L^1(\mathbb{R}^d, \widehat{\mu})_b$ extend to a strongly continuous sub-Markovian resolvent and to a sub-Markovian $C_0$-semigroup of contractions on $L^r(\mathbb{R}^d, \widehat{\mu})$, $r \in [1, \infty)$, denoted by $(G'_{\alpha})_{\alpha>0}$ and $(T'_t)_{t>0}$, respectively. Denote by $(L'_r, D(L'_r))$ the generator corresponding to $(G'_{\alpha})_{\alpha>0}$ on $L^r(\mathbb{R}^d, \widehat{\mu})$. We obtain a corresponding bilinear form with domain 
$\big(D(L_2) \times L^2(\R^d, \widehat{\mu})\big) \cup \big(L^2(\R^d,\widehat{\mu}) \times D(L'_2)\big)$ by
\begin{equation} \label{fwpeokce}
{\mathcal{E}}(f,g):= \left\{ \begin{array}{r@{\quad\quad}l}
  -\int_{\R^d} L_2 f \cdot g \,d \widehat{\mu} & \mbox{ for}\ f\in D(L_2), \ g\in L^2(\R^d, \widehat{\mu}),  \\ 
            -\int_{\R^d} f\cdot L'_2 g \,d \widehat{\mu}  & \mbox{ for}\ f\in L^2(\R^d, \widehat{\mu}), \ g\in D(L'_2). \end{array} \right .
\end{equation}
 $\mathcal{E}$ with domain $\mathcal{F}:=D(L_2)$ is called the {\it generalized Dirichlet form associated with} $(L_2,D(L_2))$ (\cite[I.4.9(ii)]{St99diss}).\\
\subsection{The Hunt process associated to the generalized Dirichlet form $\mathcal{E}$}
In order to construct a Hunt process with informal generator $L$, we need to introduce some analytic potential theory related to the generalized Dirichlet form $\mathcal{E}$  with domain $\mathcal{F}=D(L_2)$.  Let us first introduce the notion of $1$-reduced function and some notation from \cite{St99diss, Tr5}. \\
$f\in L^2(\R^d, \widehat{\mu})$ is called $1$-excessive if $\beta G_{\beta+1}f\le f$ for all $\beta\ge 0$. Choosing $\beta=0$, we see that $1$-excessive functions are positive. For $V\subset \R^d$, $V$ open and  $f\in L^2(\R^d, \widehat{\mu})$ such that there exists $u\in D(L_2)$ with $u\ge f 1_V$, there exists a unique $1$-excessive function $f_V$ such that $f 1_V\le f_V\le v$ for any  $1$-excessive function $v$ with $v\ge f 1_V$ (\cite[III. Proposition 1.7]{St99diss}). $f_V$ is called the $1$-reduced function of $f$ on $V$. One easily deduces that for $f,g$ for which $f_V,g_V$  as described above exist, one has that $f\le g$ implies $f_V\le g_V$. Moreover, if $f$ is $1$-excessive, then $f_V=f$ on $V$. Furthermore, $G_1 h$ is $1$-excessive if $h\in L^2(\R^d,\widehat{\mu})$, $h\ge 0$.\\
Now let us introduce a space of functions for which we will consider $1$-reduced functions and which we will use crucially to construct the above mentioned process:
\begin{eqnarray}\label{defG}
{\cal{G}}:=\{u\in D(\overline{L})_b\,:\, \exists u_{\text{dom}}\in L^1(\R^d,\widehat{\mu})_b,\ | u|\le G_1 u_{\text{dom}}  \}\subset L^{2}(\R^d,\widehat{\mu})_b.
\end{eqnarray}
The function $u_{\text{dom}}$ above is not unique, so when we use this kind of notation below, we mean by it an arbitrary function with the given property. Obviously, ${\cal{G}}$ is a linear space and in order to show that ${\cal{G}}$ is an algebra of functions it is hence enough to show that with $u\in{\cal{G}}$ we also have $u^2\in{\cal{G}}$. But this is clear since by the sub-Markov property and linearity of $G_1$,
\[
|u^2|\le \big (G_1 u_{\text{dom}}\big )^2\le G_1\big ( \|u_{\text{dom}}\|_{L^{\infty}(\R^d)} u_{\text{dom}}\big ),
\]
so that we can choose $\|u_{\text{dom}}\|_{L^{\infty}(\R^d)} \cdot u_{\text{dom}}$ as $(u^2)_{\text{dom}}$, and $u^2\in D(\overline{L})_b$ by Theorem \ref{pojjjde}. \\
For $f\in {\cal{G}}$ and arbitrary $V\subset \R^d$, $V$ open, $f_V\in  L^{2}(\R^d,\widehat{\mu})$ exists, since $G_1(f_{\text{dom}})\in D(L_2)$ and $G_1(f_{\text{dom}})\ge |f|\ge f 1_V$. Let $\alpha >0$ be arbitrary. By \cite[III. Proposition 1.6]{St99diss}, let $f_V^{\alpha}\in D(L_2)$ be the unique solution to
\[
(1-L_2)f_V^{\alpha}=\alpha\big ( f_V^{\alpha} -f 1_V\big )^- \quad \text{in } L^2(\R^d,\widehat{\mu}).
\]
Then $f_V^{\alpha}=G_1(\alpha\big ( f_V^{\alpha} -f 1_V\big )^-)$ is $1$-excessive, hence $f_V^{\alpha}\ge 0$, and moreover $f_V^{\alpha}\nearrow$ as $\alpha\nearrow$,  and 
$\lim_{\alpha\to \infty}f_V^{\alpha}=f_V$ in $L^{2}(\R^d,\widehat{\mu})$ (cf. \cite[III. Proposition 1.7]{St99diss}) and its proof). In particular, we deduce
\[
0\le f_V^{\alpha}\le f_V\le \big ( G_1 f_{\text{dom}}  \big )_V\le G_1 \big ((f_{\text{dom}})^+\big )\le \|G_1 \big ((f_{\text{dom}})^+\big )\|_{L^{\infty}(\R^d)}\le const.
\]
In order to further exploit properties of the space  ${\cal{G}}$, we need the following lemma.
\begin{lemma}\label{SD3sqr}
(i) We have $D(L_2)\subset D(\mathcal{E}^0)$ and
\[
\mathcal{E}^0_1(f,f)\le \int_{\R^d}(1-L_2)f f\,d\widehat{\mu}, \quad \forall f\in D(L_2).
\]
In particular, since $G_1\big (L^2(\R^d,\widehat{\mu})_b\big )\subset G_1\big (L^2(\R^d,\widehat{\mu})\big )=D(L_2)$ densely and $C_0^{\infty}(\R^d)\subset D(L_2)$, we have that $G_1\big (L^2(\R^d,\widehat{\mu})_b\big)\subset D(\mathcal{E}^0)$ densely.\\
(ii) For $f\in {\cal{G}}$ (${\cal{G}}$ as in \eqref{defG}), we have $f_V\in D(\mathcal{E}^0)_b$, 
\begin{equation*} 
\mathcal{E}_{1}^{0}(f_V, g) - \int_{\R^d} \langle \mathbf{B}, \nabla f_V \rangle g \, d \widehat{\mu}=0, \quad \text{ for any $g \in D(\mathcal{E}^0)_{0,b}$ with $g=0$ on V},
\end{equation*}
and
\begin{equation*} 
\mathcal{E}_{1}^{0}(f_V, f_V)\le 2 \|(1-\overline{L})f\|_{L^{1}(\R^d,\widehat{\mu})}\big (  \|G_1 \big ((f_{\rm{dom}})^+\big )\|_{L^{\infty}(\R^d)} +2\|f\|_{L^{\infty}(\R^d)}\big ).
\end{equation*}\\
(iii) ${\cal{G}}$ (as in \eqref{defG})  satisfies condition {\bf SD3} of \cite{Tr5}. 
\end{lemma}
\begin{proof}
(i) This follows easily from  Theorem \ref{mainijcie}(iii) using the denseness of $G_1\big (L^{1}(\R^d,\widehat{\mu})_b\big )$ in $D(L_2)$. Indeed for $f\in D(L_2)$, we may choose $g_n\in L^{1}(\R^d,\widehat{\mu})_b$ such that $\lim_{n\to\infty}g_n=(1-L_2)f$ in $L^{2}(\R^d,\widehat{\mu})$, hence $\lim_{n\to\infty}G_1g_n=f$ in $L^{2}(\R^d,\widehat{\mu})$, replace $u$ by $G_1g_n$ in  Theorem \ref{mainijcie}(iii) and let $n\to\infty$.\\
(ii) Fix arbitrary $\alpha>0$. By (i) and  the preliminary facts about $f_V^{\alpha}$, it holds  $f_V^{\alpha}\in D(L_2)_b\subset D(\mathcal{E}^0)_b$, and $\sup_{n\ge1}\mathcal{E}_{1}^{0}(nG_n f_V^{\alpha},nG_n f_V^{\alpha})<\infty$, hence $\lim_{n\to \infty}nG_n f_V^{\alpha}=f_V^{\alpha}$ weakly in $D(\mathcal{E}^0)$.  By Theorem \ref{mainijcie}(iii), $f\in D(\overline{L})_b\subset D(\mathcal{E}^0)_b$ and by {\bf Step 2} in the proof of Theorem \ref{pojjjde}, $\lim_{n\to \infty}nG_n f=f$ weakly in $D(\mathcal{E}^0)$. Thus, using (i)
\begin{eqnarray*}
\mathcal{E}_{1}^{0}(f_V^{\alpha}-f, f_V^{\alpha}-f)&\le&\liminf_{n\to\infty}\mathcal{E}_{1}^{0}(nG_n(f_V^{\alpha}-f),nG_n( f_V^{\alpha}-f))\\
&\le&\liminf_{n\to\infty}\int_{\R^d}(1-L_2)nG_n(f_V^{\alpha}-f) \, nG_n(f_V^{\alpha}-f)\,d\widehat{\mu}\\
&=& \alpha\int_{\R^d}\big ( f_V^{\alpha} -f 1_V\big )^-\big ( f_V^{\alpha}-f\big )\,d\widehat{\mu}-\int_{\R^d}(1-\overline{L})f\, \big ( f_V^{\alpha}-f\big )\,d\widehat{\mu}\\
&\le&\|(1-\overline{L})f\|_{L^{1}(\R^d,\widehat{\mu})}\big(   \|G_1 \big ((f_{\rm{dom}})^+\big )\|_{L^{\infty}(\R^d)} +\|f\|_{L^{\infty}(\R^d)}\big ).
\end{eqnarray*}
It follows $\sup_{\alpha>0}\mathcal{E}_{1}^{0}(f_V^{\alpha}-f, f_V^{\alpha}-f)<\infty$. Thus $f_V\in D(\mathcal{E}^0)_b$ and $\lim_{\alpha\to\infty}f_V^{\alpha}=f_V$ weakly in $D(\mathcal{E}^0)$. Subsequently,
\begin{eqnarray*}
\mathcal{E}_{1}^{0}(f_V, f_V)&\le&\liminf_{\alpha\to\infty}\mathcal{E}_{1}^{0}(f_V^{\alpha},f_V^{\alpha})\\
&\le&\liminf_{\alpha\to\infty}2\mathcal{E}_{1}^{0}(f_V^{\alpha}-f,f_V^{\alpha}-f)+2\mathcal{E}_{1}^{0}(f,f)\\
&\le&2\|(1-\overline{L})f\|_{L^{1}(\R^d,\widehat{\mu})}\big(   \|G_1 \big ((f_{\rm{dom}})^+\big )\|_{L^{\infty}(\R^d)} +\|f\|_{L^{\infty}(\R^d)}\big )+2\int_{\R^d}(1-\overline{L})f\, f\, d\widehat{\mu}\\
&\le& 2 \|(1-\overline{L})f\|_{L^{1}(\R^d,\widehat{\mu})}\big (  \|G_1 \big ((f_{\rm{dom}})^+\big )\|_{L^{\infty}(\R^d)} +2\|f\|_{L^{\infty}(\R^d)}\big ).
\end{eqnarray*}
Finally, let $g$ be as in the statement of (ii). Then, using \eqref{stampale}
\[
\mathcal{E}_{1}^{0}(f_V, g) - \int_{\R^d} \langle \mathbf{B}, \nabla f_V \rangle g\, d \widehat{\mu}=\lim_{\alpha\to\infty}\Big (\mathcal{E}_{1}^{0}(f_V^{\alpha}, g) - \int_{\R^d} \langle \psi\mathbf{B}, \nabla f_V^{\alpha} \rangle g\, d \mu\Big )
\]
\[
=\lim_{\alpha\to\infty}\int_{\R^d}\big (1-L_2\big)f_V^{\alpha}g\, d \widehat{\mu}=\lim_{\alpha\to\infty}\alpha\int_{\R^d}\big ( f_V^{\alpha} -f 1_V\big )^-\, g \, d \widehat{\mu}=0.
\]
(iii) We already know that ${\cal{G}}$ is an algebra. Since ${\cal{G}}\cap D(L_2)\supset G_1\big (L^1(\R^d,\widehat{\mu})_b\big )$ and $G_1\big (L^1(\R^d,\widehat{\mu})_b\big)\subset G_1\big (L^2(\R^d,\widehat{\mu})\big))=D(L_2)$ densely, we obtain that ${\cal{G}}\cap D(L_2)$ is dense in  $D(L_2)$. Noting that for $u\in {\cal{G}}$, we also have that $u-\alpha G_\alpha u\in {\cal{G}}$, with $|u-\alpha G_\alpha u|\le G_1 (u_{\text{dom}}+\alpha G_\alpha u_{\text{dom}} )$, we thus get by (ii), 
\begin{eqnarray*}
&&\mathcal{E}_{1}^{0}((u-\alpha G_\alpha u)_{\R^d}, (u-\alpha G_\alpha u )_{\R^d})\\
&\le& 2 \|(1-\overline{L})(u-\alpha G_\alpha u)\|_{L^{1}(\R^d,\widehat{\mu})}\big (  \|G_1 \big ((u_{\text{dom}}+\alpha G_\alpha u_{\text{dom}} )^+\big )\|_{L^{\infty}(\R^d)} +2\|u-\alpha G_\alpha u\|_{L^{\infty}(\R^d)}\big )\\
&\le&  const.\|\alpha\overline{G}_\alpha(1-\overline{L})u-(1-\overline{L})u\|_{L^{1}(\R^d,\widehat{\mu})}\longrightarrow 0  \ \text{ as } \alpha\to\infty.
\end{eqnarray*}
Thus {\bf SD3} of \cite{Tr5} holds.
\end{proof}
\begin{defn}
An increasing sequence of closed  subsets $(F_k)_{k\ge 1}$ of $\R^d$ is called an $\mathcal{E}$-nest (resp. $\mathcal{E}^0$-nest), if $\lim_{k\to \infty}f_{F_k^c}=0$ in $L^{2}(\R^d,\widehat{\mu})$ for all $f\in D(L_2)$, $f$ $1$-excessive (resp.  $\bigcup_{k\ge 1}D(\mathcal{E}^0)_{F_k}\subset D(\mathcal{E}^0)$ densely, where $D(\mathcal{E}^0)_{F_k}:=\{v\in D(\mathcal{E}^0)\, :\, v=0 \text{ on } F_k\}$, $k\ge 1$).
\end{defn}
Now exactly as in \cite[Lemma 3.4]{St99}, we can show:
\begin{lemma}\label{e-nest/e0-nest}
An increasing sequence of closed  subsets $(F_k)_{k\ge 1}$ of $\R^d$ is an $\mathcal{E}$-nest, if and only if it is an 
$\mathcal{E}^0$-nest.
\end{lemma}In the following proposition, notions related to a strict capacity as defined in \cite{Tr5} are introduced. For further reading, we suggest \cite{Tr5} and \cite{PTr11} (see also the proof of the following proposition).
\begin{proposition}\label{Huntex}
Assume {\bf (A)}.
There exists a Hunt process 
$$
\M= (\Omega, \F, (\F_t)_{t \ge 0}, ({X}_t)_{t \ge 0}, ({\P}_x)_{x \in \R^d \cup \{ \Delta \} })
$$ 
with life time $\zeta:=\inf\{t\ge 0\,: \, {X}_t=\Delta\}$, shift operator $({\vartheta}_t)_{t\ge 0}$, and cemetery $\Delta$, such that $x\mapsto R_{\alpha} f(x):={\E}_x[\int_0^{\infty}e^{-{\alpha t}}f({X}_t)dt]$ is a strictly $\mathcal{E}$-quasi-continuous (in short q.c.) $dx$-version (which is the same as a $\widehat{\mu}$-version by Remark \ref{inclupro}(i)) of $G_{\alpha}f$ for any $f\in L^{2}(\R^d,\widehat{\mu})$, $\alpha >0$, 
and moreover for strictly $\mathcal{E}$-quasi-every (in short q.e.) $x \in \R^d$,
$$
{\P}_{x} \big( \big\{ \omega \in {\Omega} \, : \,   {X}_{\cdot}(\omega) \in C\big([0, \infty), \R^d_{\Delta}\big),\, {X}_{t}(\omega) = \Delta, \, \forall t \geq \zeta(\omega)   \big\} \big) =1.
$$
where $\R^d_{\Delta}$ denotes the one-point-compactification of $\R^d$ with $\Delta$.
\end{proposition}
\begin{proof}
One first shows the existence of a Hunt process and then continuity of its sample paths. The strategy to obtain the Hunt process is to apply \cite[Theorem 3]{Tr5}. For this, we need to show the strict quasi-regularity of the generalized Dirichlet form $(\mathcal{E},D(L_2))$ associated with $(L_2,D(L_2))$ (cf. \cite[Definition 2]{Tr5}) and condition {\bf SD3}. By Lemma \ref{SD3sqr}(iii), {\bf SD3} is satisfied.   Defining $E_k\equiv \R^d$, $k\ge 1$, we obtain a strict $\mathcal{E}$-nest. Since $C_0^{\infty}(\R^d)\subset D(L_2)$, it is also clear how to find a countable set of strictly $\mathcal{E}$-q.c. functions in $D(L_2)$ that separate the points of $\R^d$ up to a strictly $\mathcal{E}$-exceptional set. However, showing that a dense subset of $D(L_2)$ admits strictly $\mathcal{E}$-quasi continuous $dx$-versions turns out to be more involved. We therefore follow the strategy of the proof of \cite[Theorem 6]{Tr5}, which is also applicable here. There the existence of strictly $\mathcal{E}$-q.c. $dx$-versions is shown from the existence of $\mathcal{E}$-q.c. $dx$-versions, more precisely from the quasi-regularity of $(\mathcal{E},D(L_2))$ and the existence of an $\widehat{\mu}$-tight special standard process associated with $(\mathcal{E},D(L_2))$. The quasi-regularity of $(\mathcal{E},D(L_2))$ and hence the existence of an $\widehat{\mu}$-tight special standard process associated with $(\mathcal{E},D(L_2))$ follows by \cite[IV. Theorem 2.2]{St99diss} and \cite[Proposition 2.1(i)]{PTr11}  from the quasi-regularity of $(\mathcal{E}^0, D(\mathcal{E}^0))$, $C_0^{\infty}(\R^d)\subset D(L_2)\subset D(\mathcal{E}^0)$ (see Lemma \ref{SD3sqr}(i)), Lemma \ref{e-nest/e0-nest}, and {\bf SD3}, i.e. Lemma \ref{SD3sqr}(iii).\\
The rest of the proof, i.e. the stated continuity of the sample paths,  follows from the strong local property of $(\mathcal{E},D(L_2))$, exactly as in the proof of \cite[Proposition 3.6]{LST22}.
\end{proof}
\begin{remark}\label{remstrictlyex}
(i) Since $T_t u\in D(L_2)$ for any $u\in D(L_2)$, it is easy to see that $x\mapsto P_{t} \tilde{u}(x):={\E}_x[\tilde{u}({X}_t)]$ is a strictly $\mathcal{E}$-q.c. $dx$-version of $T_t u$ for any strictly $\mathcal{E}$-q.c. $dx$-version $\tilde{u}$ of $u$, $t >0$. However, it is unclear whether  for general $u\in L^{2}(\R^d,\widehat{\mu})$, $T_t u$ has a  strictly $\mathcal{E}$-q.c. $dx$-version.\\
(ii) For every strictly $\mathcal{E}$-q.c. $dx$-version $\tilde{u}$ of $u\in D(L_2)$, the process
\[
A_t^u:=\tilde{u}({X}_t)-\tilde{u}({X}_0), \quad t\ge 0,
\]
is finite and uniquely defined (i.e. independent of the chosen $\tilde{u}$) $\P_x$-a.s. for strictly $\mathcal{E}$-q.e. $x\in \R^d$.\\
(iii) By the strict version of \cite[III. Corollary 3.4]{St99diss}, $R_{1}f$, $f\in L^{2}(\R^d,\widehat{\mu})$, is strictly  $\mathcal{E}$-q.e. independent of the Borel version chosen for $f$. Thus, using the inequality
\[
\E_x[\int_0^t  |f|(X_s)ds]\le e^t \E_x[\int_0^{t} e^{-s}|f|(X_s) ds] \le   e^t R_{1}|f|(x)
\]
we see that the integral process
\[
\Big (\int_0^t  f(X_s)ds\Big )_{t\ge 0}
\]
is uniquely defined (i.e. independent of the chosen Borel measurable version $f\in L^{2}(\R^d,\widehat{\mu})$) $\P_x$-a.s. for strictly $\mathcal{E}$-q.e. $x\in \R^d$. In particular, for $u\in D(L_2)$, the process
\[
N_t^u:=\int_0^t  L_2u(X_s)ds, \quad t\ge 0,
\]
is unique $\P_x$-a.s. for strictly $\mathcal{E}$-q.e. $x\in \R^d$, no matter which Borel measurable version $L_2 u$ is chosen in the integral. \\
(iv) By \cite[Remark 1(ii)]{Tr5} every strictly $\mathcal{E}$-exceptional set is $\mathcal{E}$-exceptional (the strict capacity is in general stronger) and by 
\cite[Lemma 1(i)]{Tr5}, we have that every $\mathcal{E}$-exceptional set is strictly $\mathcal{E}$-exceptional, if 
$\M$ is non-explosive for a.e. starting point $x\in \R^d$ (i.e. if ${\P}_{x} ({\zeta}=\infty)=1$  for a.e. $x\in \R^d$). Thus in the latter case the term \lq\lq strict\rq\rq\ can be dropped in Proposition \ref{Huntex}. 
\end{remark}

\begin{theorem}\label{prop:3.1.6}
(i) Let $\M$ be as in Proposition \ref{Huntex}. Let $u \in D(L_2)$. Then
$$
M_t ^u: = u(X_t) -u(X_0)- \int_0^t L_2 u(X_s) \, ds , \quad t \ge 0,
$$ 
is a continuous $(\mathcal{F}_t)_{t \ge 0}$-martingale under $\P_x$  for strictly $\mathcal{E}$-q.e. $x \in \R^d$. \\
(ii) Let $u\in C_0^{2}(\R^d)$. Then the quadratic variation process $\langle M^u \rangle$ of $M^u$ satisfies for strictly $\mathcal{E}$-q.e. $x \in \R^d$
$$
\langle M^u \rangle_t=\int_0^t \langle \widehat{A}\nabla u, \nabla u\rangle(X_s)ds, \quad t\ge 0, \quad \P_x\text{-a.s.}
$$
and therefore (cf. Remark \ref{remstrictlyex}(iii)) $(M_t ^u)_{t\ge 0}$ is $\P_x$-square integrable for strictly $\mathcal{E}$-q.e. $x \in \R^d$. \\
(iii) Let $\mathbf{G}=(g_1, \ldots, g_d)=\beta^{\rho, A,\psi} + \mathbf{B}$.
Consider the Hunt process $\M$ from Proposition \ref{Huntex} with coordinates $X_t=(X_t^1,\ldots,X_t^d)$.
Suppose that  $\M$ is non-explosive for strictly $\mathcal{E}$-q.e. starting point $x \in \R^d$, or equivalently that $ (T_t)_{t>0}$ is conservative (sufficient conditions for this can be derived as for instance in \cite[Section 3.2.1]{LST22}, or by using Theorem \ref{theoconserva}, or see also \eqref{growthconser 2} for a standard explicit condition).
Let $(\widehat\sigma_{ij})_{1 \le i,j \le d}$ be any matrix (possibly non-symmetric) consisting of locally bounded and measurable functions such that $\widehat\sigma \widehat\sigma^T(x) =\widehat A(x)$ for a.e. $x\in \R^d$. Then it holds $\P_x$-a.s. and for strictly $\mathcal{E}$-q.e. (hence a.e.) $x=(x_1,\ldots,x_d)\in \R^d$, and any $i=1,\ldots,d$, that
\begin{equation}\label{weaksolution}
X_t^i = X_0^i+ \sum_{j=1}^d \int_0^t \widehat\sigma_{ij} (X_s) \, dW_s^j +   \int^{t}_{0}   g_i(X_s) \, ds, \quad 0\le  t <\infty,
\end{equation}
where $W = (W^1,\dots,W^d)$ is a $d$-dimensional standard $(\mathcal{F}_t)_{t \geq 0}$-Brownian motion starting from zero and $\P_x(X_0^i=x_i)=1$.
\end{theorem}
\begin{remark}\label{class of drifts}
For any $A$ as in  {\bf (A)}, a symmetric positive definite matrix of functions $\sigma$ with $\sigma^2=A$ always exists by \cite[Lemma 2.1]{CH97}. On the other hand, any (possibly non-symmetric) matrix consisting of locally bounded and measurable functions $(\widehat\sigma_{ij})_{1 \le i,j \le d}$ such that $A:=\psi\widehat\sigma \widehat\sigma^T$ where $\psi, A$ are as in {\bf (A)} can be the starting point of our consideration.\\
\end{remark}
\begin{proof} {\bf of Theorem \ref{prop:3.1.6}}
(i) We will use the facts listed in Remark \ref{remstrictlyex} without explicitly referring to them. First $\int_0^t L_2 u(X_s) \, ds$,  $t \ge 0$, is $\P_x$-a.s. for strictly $\mathcal{E}$-q.e. $x \in \R^d$ independent of the chosen version $L_2u$. Let $f\in L^{2}(\R^d,\widehat{\mu})$ and $u=R_1 f$.  Then  $R_1f -f\in \mathcal{B}(\R^d)$ is a $dx$-version of $L_2 u$. Using in particular the Markov property for the second inequality right below, the following estimate for strictly $\mathcal{E}$-q.e. $x \in \R^d$
\begin{eqnarray}\label{estmart}
\E_x \big [ |M_t^{R_1f}|\big ]&\le& \E_x \big [R_1 |f | (X_t) \big ]+ R_1 |f | (x)+\E_x \big [\int_0^t\big  (R_1 |f |+| f|\big ) (X_s) ds\big ]\nonumber\\
&\le& e^t R_1 |f | (x) + R_1 |f | (x)+e^t R_1\big (R_1 |f |+| f|\big ) (x)\\ \nonumber
 \end{eqnarray}
is easily derived. From \eqref{estmart}, we deduce that $(M_t^{R_1f})_{t\ge 0}$ is $\P_x$-integrable for strictly $\mathcal{E}$-q.e. $x \in \R^d$ and since it is obviously adapted it remains to show the martingale property. For this consider $(f_n)_{n\ge 1}\subset L^{2}(\R^d,\widehat{\mu})_b$, such that $\lim_{n\to \infty }f_n=f$ in $L^{2}(\R^d,\widehat{\mu})$. Then \eqref{estmart}, the fact that $M_t^{R_1f}-M_t^{R_1f_n}=M_t^{R_1(f-f_n)}$, and \cite[III. Corollary 3.8]{St99diss}, imply that for each $t\ge 0$, $\lim_{k\to \infty}M_t^{R_1f_{n_k}}=M_t^{R_1f}$ in  $L^1(\Omega,\F_t, \P_x)$ for some subsequence $(n_k)_{k\ge 1}$ and strictly $\mathcal{E}$-q.e. $x \in \R^d$. Therefore, it is enough to show the martingale property of $(M_t^{R_1f})_{t\ge 0}$, when 
$f\in C^{\infty}_0(\R^d)$, which we assume from now on. We get for strictly $\mathcal{E}$-q.e. $x \in \R^d$
\begin{eqnarray*}
	\E_x\big [M_t^{R_1f}\big ]
	&=& \E_x \big [\E_{X_t} \big [\int_0^\infty e^{-s} f(X_s)ds\big ]- \int_0^\infty e^{-s} f(X_s)ds\big ]\\
	&&\qquad -\int_0^t \big (\E_x\Big [\E_{X_s} [\int_0^\infty e^{-u}f(X_u) du ]-f(X_s)\Big ]\big )ds\\
	&=&\E_x\big [e^t \int_t^\infty e^{-s} f(X_s)ds - \int_0^\infty e^{-s} f(X_s) ds\big ]\\
	&&\qquad -\int_0^t \E_x \Big [e^s \int_s^\infty e^{-u} f(X_u) du - e^s e^{-s} f(X_s)\Big ]ds,
	\end{eqnarray*}
	and since
	\begin{eqnarray*}
	-\int_0^t e^s e^{-s} f(X_s) ds
	&= &-\int_0^t e^s d\big (\int_0^s e^{-u} f(X_u)du\big )\\
	&= &\int_0^t \int_0^s e^{-u} f(X_u) du\, e^s\, ds - e^t \int_0^t e^{-u} f(X_u) du
	\end{eqnarray*}
	we obtain
	\[ 
	\E_x\big [M_t^{R_1f}\big ] = \E_x\Big [e^t \int_0^\infty e^{-s} f(X_s) ds - \int_0^\infty e^{-s}f(X_s)ds - \underbrace{\int_0^t e^s \big (\int_0^\infty e^{-u} f(X_u) du\big )ds}_{=(e^t - 1)\int_0^\infty e^{-u} f(X_u)du}\Big ]=0.
	\]
We have hence derived that there exists a strict $\mathcal{E}$-nest $(F_k)_{k\ge 1}$ and a set $N\subset \R^d$ with $N\subset \cap_{k\ge 1}F_k^c$, such that  $\E_x\big [M_t^{R_1f}\big ]=0$ for all $x\in \R^d\setminus N$ and $t\ge 0$. Since the intersection of two strict $\mathcal{E}$-nests is again a strict $\mathcal{E}$-nest, making $(F_k)_{k\ge 1}$ smaller if necessary, by \cite[Lemma 1(ii)]{Tr5}, we may assume that $\P_x(\lim_{k\to\infty}\sigma_{F_k^c}=\infty)=1$ for all $x \in \R^d\setminus N$. We thus obtain that 
\begin{eqnarray}\label{averageproperty}
\P_x(\E_{X_t}[M_s^{R_1f}]=0) = 1\qquad \text{for all } t,s\ge   0 \text{ and } x \in \R^d\setminus N.
\end{eqnarray}
It holds $M_{t+s}^{R_1f} =M_t^{R_1f} + M_s^{R_1f} \circ \vartheta_t$ for any $t,s\ge 0$, $\P_x$-a.s. for all $x\in \R^d$, 
since the same is true for $A^{R_1f}$ and $N^{R_1f}$. 
Then by the Markov property  and \eqref{averageproperty}, it holds $\P_x$-a.s. for all $x \in \R^d\setminus N$
\[
\E_x\left[M_{t+s}^{R_1f} \, |\,  \mathcal{F}_t\right]\ =\  \E_x \big[M_t^{R_1f} + M_s^{R_1f} \circ \vartheta_t \, | \, \mathcal{F}_t\big]\  =\ M_t^{R_1f} + \E_{X_t}[M_s^{R_1f}] \ =\ M_t^{R_1f}
\]
and (i) is shown.\\
(ii) In view of part (i) and Theorem \ref{pojjjde}, more precisely that for $u\in C_0^{2}(\R^d)$, we have $Lu^2 = 2u Lu +  \langle \widehat{A} \nabla u, \nabla u \rangle$, the statement can be proved as in \cite[Proposition 3.19]{LST22}.\\
(iii) This follows as in \cite[Theorem 3.22(i)]{LST22}.
\end{proof}

\section{Uniqueness of the $L^1(\mathbb{R}^d, \widehat{\mu})$-closed extension} \label{l1closeduni}
In order to investigate the uniqueness (in law) of weak solutions to stochastic differential equations, one possibility is to study the $L^1(\mathbb{R}^d, \widehat{\mu})$-uniqueness of $(L, C_0^{\infty}(\mathbb{R}^d))$.
We will here follow the standard line of arguments as outlined in \cite{St99} but we emphasize that it is necessary to present detailed proofs since in our case, we consider convergences with respect to two measures, $\mu$ and $\widehat{\mu}$ (see for instance \eqref{closablet}). Here, let us recall the definition of $L^1(\mathbb{R}^d, \widehat{\mu})$-uniqueness. A densely defined operator $(L, \mathcal{D})$ on $L^1(\mathbb{R}^d, \widehat{\mu})$ is said to be $L^1(\mathbb{R}^d, \widehat{\mu})$-unique if there exists only one closed extension on $L^1(\mathbb{R}^d, \widehat{\mu})$ that generates a $C_0$-semigroup on $L^1(\mathbb{R}^d, \widehat{\mu})$. For $\mathcal{D} \subset D(L^0)_{0,b}$ such that $\mathcal{D} \subset L^1(\mathbb{R}^d, \widehat{\mu})$ densely, it is well-known that the dense range condition is equivalent to the $L^1(\mathbb{R}^d, \widehat{\mu})$-uniqueness of $(L, \mathcal{D})$. By a simple application of the Hahn-Banach Theorem the dense range condition is in turn equivalent to the following: \\ \\
\centerline{\it
$\exists \ \alpha>0$ such that if $h\in L^\infty(\mathbb{R}^d, \widehat{\mu})$ and $\displaystyle \int_{\mathbb{R}^d} (\alpha-L) v \cdot h \, d \widehat{\mu}=0$ \;for all $v \in \mathcal{D}$, then $h=0$ a.e.
}
\text{}\\
For further details on the equivalence mentioned above, we refer to \cite[Remark 3.4]{LT22in}. 

\begin{theorem} \label{analycharin}
Assume {\bf (A)}. Let $(\overline{L}, D(\overline{L}))$, $(L, D(L^0)_{0,b})$ and $(\overline{T}_t)_{t>0}$ be defined as in Theorem \ref{mainijcie}. The following statements (i)--(iv) are equivalent:
\begin{itemize}
\item[(i)]
There exist $\alpha>0$ and $\chi_n \in \widehat{H}^{1,2}_{loc}(\mathbb{R}^d, \mu)$ (i.e. $ \chi_n\phi \in \widehat{H}^{1,2}_0(\mathbb{R}^d, \mu)$ for all $\phi \in C_0^{\infty}(\mathbb{R}^d)$), $n\ge 1$, such that $(\chi_n-1)^- \in \widehat{H}^{1,2}_0(\mathbb{R}^d, \mu)_{0,b}$, $\lim_{n \rightarrow \infty} \chi_n =0$ a.e. and
\begin{equation} \label{invinequ}
\mathcal{E}^0_{\alpha}(\chi_n, v) + \int_{\mathbb{R}^d} \langle \mathbf{B}, \nabla \chi_n \rangle v d \widehat{\mu} \geq 0 \quad \text{ for all } v \in \widehat{H}^{1,2}_0(\mathbb{R}^d, \mu)_{0,b}, \; v \geq 0.
\end{equation}

\item[(ii)]
$(L, D(L^{0})_{0,b})$ is $L^1(\mathbb{R}^d, \widehat{\mu})$-unique.

\item[(iii)]
$\widehat{\mu}$ is an invariant measure for $(\overline{L}, D(\overline{L}))$, i.e.
$$
\int_{\mathbb{R}^d} \overline{L}u \, d \widehat{\mu}=0, \quad \text{ for all $u \in D(\overline{L})$}.
$$

\item[(iv)]
$\widehat{\mu}$ is $(\overline{T}_t)_{t>0}$-invariant, i.e.
$$
\int_{\mathbb{R}^d} \overline{T}_t u d\widehat{\mu} = \int_{\mathbb{R}^d} u d\widehat{\mu} \quad \text{ for all $u \in L^1(\mathbb{R}^d, \widehat{\mu})$}.
$$
\end{itemize}
Moreover, if $\widehat{\mu}$ is $(\overline{T}_t)_{t>0}$-invariant, then {\bf for any} $\alpha>0$ there exists $\chi_n \in \widehat{H}^{1,2}_{loc}(\mathbb{R}^d, \mu)$ such that $(\chi_n-1)^- \in \widehat{H}^{1,2}_0(\mathbb{R}^d, \mu)_{0,b}$, $\lim_{n \rightarrow \infty} \chi_n =0$ a.e. and \eqref{invinequ} holds.
\end{theorem}
\begin{proof}
((i) $\Rightarrow$ (ii)):  Assume that $(i)$ holds. Let $\alpha>0$ and assume that $h \in L^{\infty}(\mathbb{R}^d)$ satisfies
$$
\int_{\mathbb{R}^d} (\alpha-L)v \cdot h\, d\widehat{\mu}=0 \quad \text{ for all $v \in D(L^0)_{0,b}$}.
$$
To show (ii), it suffices to show that $h=0$ a.e. First, Lemma \ref{appenlem} yields that $h\in \widehat{H}^{1,2}_{loc}(\mathbb{R}^d, \mu)$ and 
\begin{align} \label{invinequ2}
\mathcal{E}_{\alpha}^0(v, h) - \int_{\mathbb{R}^d} \langle \bold{B}, \nabla v \rangle h \, d \widehat{\mu}=0 \quad \text{ for all $v \in \widehat{H}^{1,2}_0(\mathbb{R}^d, \mu)_0$}.
\end{align}
Now let $v_n= \|h\|_{L^{\infty}(\mathbb{R}^d)} \chi_n -h$. Since $\|h\|_{L^{\infty}(\mathbb{R}^d)}(\chi_{n}-1) \leq v_n$, we get
$\|h\|_{L^{\infty}(\mathbb{R}^d)}(\chi_{n}-1)^- \geq v_n^- \geq 0$, so that $v_n^- \in \widehat{H}^{1,2}_0(\mathbb{R}^d, \mu)_{0,b}$. Substituting $v_n^-$ for $v$ in \eqref{invinequ} and \eqref{invinequ2},  we get
\begin{align} 
&\mathcal{E}_{\alpha}^0(v_n^-, \chi_n) - \int_{\mathbb{R}^d} \langle \mathbf{B}, \nabla v_n^- \rangle \chi_n d \widehat{\mu} \geq 0  \label{equa1} \\
&  \mathcal{E}_{\alpha}^0(v_n^-, h) - \int_{\mathbb{R}^d} \langle \bold{B}, \nabla v_n^- \rangle h \, d \widehat{\mu}=0 . \label{equa2}
\end{align} 
Subtracting \eqref{equa2} from \eqref{equa1} multiplied by $\|h\|_{L^{\infty}(\mathbb{R}^d)}$, we get
\begin{equation} \label{invid1}
\mathcal{E}_{\alpha}^0(v_n^-, v_n) - \int_{\mathbb{R}^d} \langle \mathbf{B}, \nabla v_n^- \rangle v_n d \widehat{\mu} \geq 0.
\end{equation}
By Lemma \ref{teclemma} and \eqref{divfreeext}
\begin{equation} \label{invid2}
- \int_{\mathbb{R}^d} \langle \mathbf{B}, \nabla v_n^- \rangle v_n d \widehat{\mu} = \int_{\mathbb{R}^d} \left \langle \mathbf{B}, \nabla (-v_n)^+ \right \rangle (-v_n) d \widehat{\mu} = \int_{\mathbb{R}^d} \left \langle \mathbf{B}, \nabla (-v_n)^+ \right \rangle (-v_n)^+ d \widehat{\mu} =0.
\end{equation}
Furthermore, since $(\mathcal{E}^0, D(\mathcal{E}^0))$ is a symmetric Dirichlet form, it holds $\mathcal{E}^0(v_n^+, v_n^-)\leq 0$ (see \cite[Chapter I. (4.1)]{MR}), and hence
\begin{align}
\mathcal{E}_{\alpha}^0(v_n^-, v_n)&=\mathcal{E}^0(v_n, v_n^-) + \alpha \int_{\mathbb{R}^d} v_n v_n^- d\widehat{\mu}= \mathcal{E}^0(v_n^+,v_n^-)-\mathcal{E}^0(v^{-}_n, v^{-}_n) - \alpha \|v_n^-\|^2_{L^2(\mathbb{R}^d, \widehat{\mu})} \nonumber  \\
&\leq - \alpha \|v_n^-\|^2_{L^2(\mathbb{R}^d, \widehat{\mu})} \label{invid3}.
\end{align}
Thus, \eqref{invid1}, \eqref{invid2} and \eqref{invid3} imply that $0 \leq  - \alpha \|v_n^-\|^2_{L^2(\mathbb{R}^d, \widehat{\mu})}$, so that $v_n^-=0$, equivalently $\|h\|_{L^{\infty}(\mathbb{R}^d)} \chi_n  \geq h$. Similarly, replacing $h$ with $-h$, we get
$-\|h\|_{L^{\infty}(\mathbb{R}^d)} \chi_n  \leq h$. Letting $n \rightarrow \infty$, we finally get $h=0$ a.e. \\ \\
((ii) $\Rightarrow$ (iii)): Let $u \in D(L^0)_{0,b}$ and $\chi \in C_0^{\infty}(\mathbb{R}^d)$. Then, it holds that
\begin{align*}
-\int_{\mathbb{R}^d} \overline {L} u \cdot \chi \, d \widehat{\mu} &=  -\int_{\mathbb{R}^d} L^0 u \cdot \chi \, d \widehat{\mu}-\int_{\mathbb{R}^d} \langle \mathbf{B}, \nabla u \rangle  \cdot \chi \, d \widehat{\mu} \\
&= \int_{\mathbb{R}^d} \langle \widehat{A} \nabla u, \nabla \chi \rangle d \widehat{\mu} - \int_{\mathbb{R}^d} \langle \mathbf{B}, \nabla u \rangle  \cdot \chi \, d \widehat{\mu}.
\end{align*}
Taking $\chi$ so that $\chi=1$ on $\text{supp}(u)$, we get
$$
\int_{\mathbb{R}^d} \overline{L} u\, d\widehat{\mu} = 0, \quad \text{ for all $u \in D(L^0)_{0,b}$}.
$$
Since $(\overline{L}, D(\overline{L}))$ is the closure of $(L, D(L^0)_{0,b})$ in $L^1(\mathbb{R}^d, \widehat{\mu})$, the above implies that
$$
\int_{\mathbb{R}^d} \overline{L} u\, d\widehat{\mu} = 0, \quad \text{ for all $u \in D(\overline{L})$}.
$$
((iii) $\Rightarrow$ (iv)): Using the Fundamental Theorem of Calculus for Bochner integrals,
$$
\overline{T}_t u - u = \int_0^t \overline{L} \,\overline{T}_s u \,ds \;\; \text{ on $L^1(\mathbb{R}^d, \widehat{\mu})$}
\quad \text{ for all $u \in D(\overline{L})$}.
$$
Taking integrals with respect to $\widehat{\mu}$ on both sides above and using Fubini's theorem,  we have 
$$
\int_{\mathbb{R}^d} \big (\overline{T}_t u - u \big )\, d\widehat{\mu} =  \int_{\mathbb{R}^d} \int_0^t \overline{L} \,\overline{T}_s u \,ds\, d \widehat{\mu}= \int_0^t   \int_{\mathbb{R}^d} \overline{L} \,\overline{T}_s u \, d \widehat{\mu}\, ds=0, \quad \text{for all $u\in D(\overline{L})$},
$$
where the last equality follows from (iii) and the fact that $\overline{T}_s u \in D(\overline{L})$ for all $s>0$. Since $D(\overline{L})$ is dense in $L^1(\mathbb{R}^d, \widehat{\mu})$, the assertion follows. \\ \\
((iv) $\Rightarrow$ (i)): We will show the final assertion. Let $\alpha>0$ be given. For each $n \geq 1$, let $B_n= \{ x \in \mathbb{R}^d:  \|x\|<n \}$ and
$$
\chi_n= 1-\alpha \overline{G}_{\alpha}'^{,B_n} (1_{B_n}),
$$
where $(\overline{G}_{\alpha}'^{, B_n})_{\alpha>0}$ is the resolvent defined in Remark \ref{remimpco}. \\
{\bf (Step 1):}  First, we will show that $\lim_{n \rightarrow \infty} \chi_n=0$ a.e. \\
Notice that $(\chi_n)_{n \geq 1}$ is decreasing  by Lemma \ref{increlem} and $\chi_n \geq 0$ for all $n \geq 1$. Thus, there exists $\chi_{\infty} \in \mathcal{B}_b(\mathbb{R}^d)$ such that $\lim_{n \rightarrow \infty} \chi_n = \chi_{\infty}$ a.e. Since $\widehat{\mu}$ is $(\overline{T}_t)_{t>0}$-invariant, Laplace transformation of $(\overline{T}_t)_{t>0}$
 implies that
$$
\int_{\mathbb{R}^d} \alpha \overline{G}_{\alpha}g \,d\widehat{\mu} = \int_{\mathbb{R}^d} g \, d \widehat{\mu} \quad \text{ for all $g \in L^1(\mathbb{R}^d, \widehat{\mu})_b$}.
$$
Thus, for each $g \in L^1(\mathbb{R}^d, \widehat{\mu})_b$,
\begin{align*}
\int_{\mathbb{R}^d} g \chi_{\infty} \, d \widehat{\mu}  &= \lim_{n \rightarrow \infty} \int_{\mathbb{R}^d}g \chi_n \, d \widehat{\mu} = \lim_{n \rightarrow \infty}\int_{\mathbb{R}^d} g d\widehat{\mu} - \int_{\mathbb{R}^d} g \cdot \alpha \overline{G}_{\alpha}'^{, B_n} (1_{B_n}) d\widehat{\mu} \\
&=\lim_{n \rightarrow \infty} \int_{\mathbb{R}^d} g d\widehat{\mu} - \int_{\mathbb{R}^d} \alpha \overline{G}_{\alpha}^{B_n} g \cdot 1_{B_n} d \widehat{\mu}= \int_{\mathbb{R}^d} g d\widehat{\mu} - \int_{\mathbb{R}^d} \alpha \overline{G}_{\alpha} g\, d\widehat{\mu}=0,
\end{align*}
where Theorem \ref{mainijcie}(ii) and Lebesgue's theorem are used for the last part. Thus, $\chi_{\infty}=0$ a.e. as desired.
\\
{\bf (Step 2):} Note that $\chi_n \in \widehat{H}^{1,2}_{loc}(\mathbb{R}^d, \mu)$ and $(\chi_n-1)^- \in \widehat{H}^{1,2}_0(\mathbb{R}^d, \mu)_0$ since $\overline{G}_{\alpha}'^{, B_n}(1_{B_n})\in \widehat{H}^{1.2}_0(\mathbb{R}^d, \mu)_{0}$ with $\text{supp}\big (\overline{G}_{\alpha}'^{, B_n}(1_{B_n})\big ) \subset B_n$. Now, we will show that \eqref{invinequ} is satisfied. Fix $n \geq 1$ and let
$$
w_{\beta} = \beta \overline{G}_{\beta+\alpha}' \alpha \overline{G}_{\alpha}'^{, B_n} (1_{B_n}), \quad \beta>0.
$$
Using \eqref{weakineqsom},
\begin{align}
\mathcal{E}^{0}_{\beta+\alpha} (w_{\beta}, w_{\beta}) &= \mathcal{E}^0_{\beta+\alpha} \big(  \beta \overline{G}'_{\beta+\alpha} \alpha \overline{G}_{\alpha}'^{, B_n} (1_{B_n}),  \beta \overline{G}_{\beta+\alpha} \alpha \overline{G}_{\alpha}'^{, B_n} (1_{B_n})  \big)\nonumber \\
& \leq  \beta \big( \alpha \overline{G}_{\alpha}'^{, B_n} (1_{B_n}), w_{\beta}   \big)_{L^2(\mathbb{R}^d, \widehat{\mu})}. \label{varineq1}
\end{align}
Moreover, we have
$$
0 \leq \big( \alpha \overline{G}_{\alpha}'^{, B_n} (1_{B_n}) -w_{\beta}, \, \alpha \overline{G}_{\alpha}'^{, B_n} (1_{B_n}) -w_{\beta}		\big)_{L^2(\mathbb{R}^d, \widehat{\mu})},
$$
which implies that
\begin{align} \label{inequ3}
\big(\alpha \overline{G}_{\alpha}'^{, B_n} (1_{B_n}) -w_{\beta}, \, \alpha \overline{G}_{\alpha}'^{, B_n} (1_{B_n})		 \big)_{L^2(\mathbb{R}^d, \widehat{\mu})}  \geq\,  \big(\alpha \overline{G}_{\alpha}'^{, B_n} (1_{B_n}) -w_{\beta}, \, w_{\beta}  \big)_{L^2(\mathbb{R}^d, \widehat{\mu})}. 
\end{align}
Therefore, \eqref{varineq1} and \eqref{inequ3} lead us to obtain
\begin{align}\label{inequ3bis}
\mathcal{E}^0_{\alpha} ( w_{\beta}, w_{\beta}  ) &\leq \beta \big( \alpha \overline{G}_{\alpha}'^{, B_n} (1_{B_n})-w_{\beta},   w_{\beta} \big)_{L^2(\mathbb{R}^d, \widehat{\mu})}\nonumber \\
&\leq \beta \big( \alpha \overline{G}_{\alpha}'^{, B_n} (1_{B_n})-w_{\beta},\,  \alpha \overline{G}_{\alpha}'^{, B_n} (1_{B_n})   	\big)_{L^2(\mathbb{R}^d, \widehat{\mu})}.
\end{align}
Observe that by using \eqref{stampale}, for each $v \in \widehat{H}^{1,2}_0(\mathbb{R}^d, \mu)_{0, b}$
\begin{align*}
&\mathcal{E}^0_{\beta+\alpha} ( w_{\beta},  v )  +  \int_{\mathbb{R}^d} \langle \mathbf{B}, \nabla w_{\beta} \rangle v\, d\widehat{\mu} \\
&= \mathcal{E}^0_{\beta+\alpha} \big(\beta \overline{G}_{\beta+\alpha}' \alpha \overline{G}_{\alpha}'^{, B_n} (1_{B_n}), \, v \big)  + \int_{\mathbb{R}^d} \big\langle \mathbf{B}, \nabla \beta \overline{G}_{\beta+\alpha}' \alpha \overline{G}_{\alpha}'^{, B_n} (1_{B_n})    \big\rangle v \, d \widehat{\mu} \\
&= \beta \big( \alpha \overline{G}_{\alpha}'^{, B_n} (1_{B_n}), v  \big)_{L^2(\mathbb{R}^d, \widehat{\mu})},
\end{align*}
which implies that
\begin{equation} \label{identisob}
\mathcal{E}^0_{\alpha} \big(  w_{\beta}, v  \big)  + \int_{\mathbb{R}^d} \langle \mathbf{B}, \nabla w_{\beta} \rangle vd \widehat{\mu}  = \beta \big( \alpha \overline{G}_{\alpha}'^{, B_n} (1_{B_n}) - w_{\beta}, v  \big)_{L^2(\mathbb{R}^d, \widehat{\mu})} \quad \text{ for all $v \in  \widehat{H}^{1,2}_0(\mathbb{R}^d, \mu)_{0,b}$}.
\end{equation}
Thus, replacing $v$ with $\alpha \overline{G}_{\alpha}'^{, B_n} (1_{B_n})$ in \eqref{identisob} and using  \eqref{inequ3bis} and the Cauchy-Schwarz inequality, it holds
\begin{align*}
\mathcal{E}^0_{\alpha} ( w_{\beta}, w_{\beta}  ) & \leq \beta \big( \alpha \overline{G}_{\alpha}'^{, B_n} (1_{B_n}) -w_{\beta},\, \alpha \overline{G}_{\alpha}'^{, B_n} (1_{B_n})  \big)_{L^2(\mathbb{R}^d, \widehat{\mu})}\\
&=\mathcal{E}^0_{\alpha} \big(  w_{\beta}, \alpha \overline{G}_{\alpha}'^{, B_n} (1_{B_n})   \big)  + \int_{\mathbb{R}^d} \langle \mathbf{B}, \nabla w_{\beta} \rangle \alpha \overline{G}_{\alpha}'^{, B_n} (1_{B_n})  \,d \widehat{\mu}  \\
& \leq \mathcal{E}^0_{\alpha} (w_{\beta}, w_{\beta})^{1/2} \cdot \mathcal{E}^0_{\alpha} \big( \alpha \overline{G}_{\alpha}'^{, B_n} (1_{B_n}) , \alpha \overline{G}_{\alpha}'^{, B_n} (1_{B_n})   \big)^{1/2} + \| \psi \mathbf{B} \|_{L^2(B_n, \mu)} \| \nabla w_{\beta} \|_{L^2(B_n, \mu)} \\
& \leq \mathcal{E}^0_{\alpha} (w_{\beta}, w_{\beta})^{1/2} \big(    \|\alpha 1_{B_n}\|_{L^2(\mathbb{R}^d, \widehat{\mu})} \|\alpha \overline{G}_{\alpha}'^{, B_n}(1_{B_n})\|_{L^2(\mathbb{R}^d, \widehat{\mu})}  +\sqrt{2\lambda_{B_n}^{-1}} \| \psi \bold{B} \|_{L^2(B_n, \mu)}   \big) \\
&\leq  \mathcal{E}^0_{\alpha} (w_{\beta}, w_{\beta})^{1/2} \big(  \alpha \widehat{\mu}(B_n)  +\sqrt{2\lambda_{B_n}^{-1}} \| \psi \bold{B} \|_{L^2(B_n, \mu)}   \big).
\end{align*}
Thus, using the Banach-Alaoglu Theorem and the fact that
\begin{align*}
&\lim_{\beta\rightarrow \infty} w_{\beta} = \lim_{\beta \rightarrow \infty} \beta \overline{G}'_{\beta+\alpha}\alpha \overline{G}_{\alpha}'^{, B_n}(1_{B_n}) \\
&=  \lim_{\beta \rightarrow \infty} \frac{\beta}{\beta+\alpha}  (\beta+\alpha) \overline{G}'_{\beta+\alpha} \alpha \overline{G}_{\alpha}'^{, B_n}(1_{B_n}) = \alpha \overline{G}_{\alpha}'^{, B_n} (1_{B_n}) \quad \text{ in $L^2(\mathbb{R}^d, \widehat{\mu})$},
\end{align*}
there exists a subsequence of $(w_{\beta})_{\beta>0}$, say again, $(w_{\beta})_{\beta>0}$ such that
\begin{equation}\label{eq52}
\lim_{\beta \rightarrow \infty} w_{\beta} = \alpha \overline{G}_{\alpha}'^{, B_n} (1_{B_n}) \quad \text{ \; weakly in $D(\mathcal{E}^0)$}.
\end{equation}
Meanwhile, it follows from Lemma \ref{increlem}, the resolvent equation and the sub-Markovian property of $(\overline{G}_{\alpha}'^{, B_n})_{\alpha>0}$ that 
\begin{align*}
w_{\beta} &\geq \beta \overline{G}_{\beta+\alpha}'^{, B_n} \alpha \overline{G}_{\alpha}'^{, B_n} (1_{B_n}) = \overline{G}_{\alpha}'^{, B_n} (\alpha 1_{B_n}) - \overline{G}_{\beta+\alpha}'^{, B_n} (\alpha 1_{B_n})   \geq \alpha \overline{G}_{\alpha}'^{, B_n} (1_{B_n})  - \frac{\alpha}{\beta+\alpha}, \;\qquad \beta>0,
\end{align*}
so that
\begin{equation} \label{imptineq}
 \alpha \overline{G}_{\alpha}'^{, B_n} (1_{B_n}) -w_{\beta} \leq \frac{\alpha}{\beta+\alpha}, \;\quad \beta>0.
\end{equation}
Hence, we finally obtain that for each $v \in \widehat{H}^{1,2}_0(\mathbb{R}^d, \mu)_{0,b}$, $v\ge 0$,
\begin{align*}
	&\mathcal{E}^0_{\alpha} (\chi_n, v) + \int_{\mathbb{R}^d} \langle \mathbf{B}, \nabla \chi_n \rangle v d\widehat{\mu} \\
	&= \alpha \int_{\mathbb{R}^d} v\, d \widehat{\mu} - \mathcal{E}^0_{\alpha} \big(  \alpha \overline{G}_{\alpha}'^{, B_n} (1_{B_n}), \, v \big ) - \int_{\mathbb{R}^d} \langle \mathbf{B}, \alpha \overline{G}_{\alpha}'^{, B_n} (1_{B_n}) \rangle v d \widehat{\mu} \\
	& \underset{\text{by } \eqref{eq52}}{=}\big (\lim_{\beta \rightarrow \infty} \alpha \int_{\mathbb{R}^d} v \, d \widehat{\mu} - \mathcal{E}^0_{\alpha} \left( w_{\beta}, \, v     \right) - \int_{\mathbb{R}^d} \langle \mathbf{B}, w_{\beta} \rangle v d \widehat{\mu}\big )  \\
	& \underset{\text{by } \eqref{identisob}}{=} \,\lim_{\beta \rightarrow \infty}\big ( \alpha \int_{\mathbb{R}^d} v \, d \widehat{\mu} - \beta \int_{\mathbb{R}^d} \big( \alpha \overline{G}_{\alpha}'^{, B_n} (1_{B_n}) - w_{\beta}\big) v\, d\widehat{\mu} \big )\\
	& \underset{\text{by \eqref{imptineq}}}{\geq}  \,\lim_{\beta \rightarrow \infty} \alpha \int_{\mathbb{R}^d} v \, d \widehat{\mu} -\frac{\beta \alpha}{\beta+\alpha} \int_{\mathbb{R}^d} v d\widehat{\mu}=0,
\end{align*}
as desired.
\end{proof}

\begin{theorem}\label{theoconserva}
Assume {\bf (A)}. Let $N_0 \in \mathbb{N}$. Let $u \in C^2(\mathbb{R}^d \setminus \overline{B}_{N_0}) \cap C(\mathbb{R}^d)$, $u  \geq 0$ with
$$
\lim_{\|x\| \rightarrow \infty} u(x) = \infty.
$$
By obvious interpretation, we can define $Lu$ and $L'u$ as functions on $\mathbb{R}^d \setminus \overline{B}_{N_0}$, where $L$ and $L'$ are given as in \eqref{definition of L} and \eqref{definition of L'}. Then, the following hold:
\begin{itemize}
\item[(i)]
If there exists a constant $M>0$ such that
$$
L' u \leq M u \quad \text{ a.e. on $\mathbb{R}^d \setminus \overline{B}_{N_0}$},
$$
then $\widehat{\mu}$ is $(\overline{T}_t)_{t>0}$-invariant, equivalently $(T'_t)_{t>0}$ is conservative, i.e. $T'_t1=1$ for some and hence all $t>0$ (cf. \cite[Lemma 3.3(i)]{LT22in}).

\item[(ii)]
If there exists a constant $M>0$ such that
$$
L u \leq M u  \quad \text{ a.e. on $\mathbb{R}^d \setminus \overline{B}_{N_0}$},
$$
then $\widehat{\mu}$ is $(\overline{T}'_t)_{t>0}$-invariant, equivalently $(T_t)_{t>0}$ is conservative, i.e. $T_t1=1$ for some and hence all $t>0$ (cf. \cite[Lemma 3.3(ii)]{LT22in}).
\end{itemize}
\end{theorem}
\begin{proof}
(i) 
Assume that there exists a constant $M>0$ such that
$$
L' u \leq M u \quad \text{ a.e. on $\mathbb{R}^d \setminus \overline{B}_{N_0}$}.
$$
According to the method in the proof of \cite[Lemma 3.26]{LST22}, there exists $u_1 \in C^2(\mathbb{R}^d)$ such that
$$
\lim_{\|x\| \rightarrow \infty} u_1(x) = \infty, \qquad \quad L'u_1 \leq M u_1 \quad \text{ a.e. on $\mathbb{R}^d$}.
$$
For each $n \geq 1$, let $\chi_n = \frac{u_1}{n}$. Then,  $\chi_n \in C^2(\mathbb{R}^d)\subset \widehat{H}_{loc}^{1,2}(\mathbb{R}^d, \mu)$ for each $n \geq 1$. Since $\lim_{\|x\| \rightarrow \infty} \chi_n(x) = \infty$, we get $(\chi_n -1)^- \in \widehat{H}^{1,2}_0(\mathbb{R}^d, \mu)_{0, b}$ for each $n \geq 1$. Moreover, $\lim_{n \rightarrow \infty} \chi_n =0$ a.e. Since
$$
\int_{\mathbb{R}^d} (M \chi_n - L' \chi_n )\,v \, d \widehat{\mu}  \geq 0,\quad \text{ for all $v \in \widehat{H}^{1,2}_0(\mathbb{R}^d, \mu)_{0,b}$, $v\ge 0$,  and $n \geq 1$},
$$
it follows from integration by parts that
\begin{equation*}
\mathcal{E}^0_{M} (\chi_n, v) + \int_{\mathbb{R}^d} \langle \mathbf{B}, \nabla \chi_n \rangle v \,d \widehat{\mu} \geq 0 \quad \text{ for all $v \in \widehat{H}^{1,2}_0(\mathbb{R}^d, \mu)_{0,b}$, $v\ge 0$,  \,and\, $n \geq 1$}.
\end{equation*}
Then, by Theorem \ref{analycharin}, $\widehat{\mu}$ is $(\overline{T}_t)_{t>0}$-invariant. \\
(ii) Analogously to (i), there exists $\widehat{\chi}_n \in C^2(\mathbb{R}^d)\subset \widehat{H}_{loc}^{1,2}(\mathbb{R}^d, \mu)$ for each $n \geq 1$ such that 
$$
\lim_{\|x\| \rightarrow \infty} \widehat{\chi}_n(x) = \infty, \; \; \; \; (\widehat{\chi}_n -1)^- \in \widehat{H}^{1,2}_0(\mathbb{R}^d, \mu)_{0, b}  \text{ for each $n \geq 1$}
$$
and that
\begin{equation*}
\mathcal{E}^0_{M} ( \widehat{\chi}_n, v) + \int_{\mathbb{R}^d} \langle -\mathbf{B}, \nabla \widehat{\chi}_n \rangle v\, d \widehat{\mu} \geq 0 \quad \text{ for all $v \in \widehat{H}^{1,2}_0(\mathbb{R}^d, \mu)_{0,b}$, $v\ge 0$,  \,and\, $n \geq 1$}.
\end{equation*}
Therefore, $\widehat{\mu}$ is $(\overline{T}'_t)_{t>0}$-invariant.
\end{proof}

\noindent
The following proposition illustrates that weak existence of a solution to \eqref{weaksolution} can be obtained for a large class of coefficients.
\begin{proposition}\label{weakexfairlygeneral}
Let $p>d$ and $p \geq 2$.  Let $A=(a_{ij})_{1 \leq i,j \leq d}$ be a symmetric matrix of measurable functions on $\mathbb{R}^d$ which is locally uniformly strictly elliptic on $\mathbb{R}^d$ and satisfies ${\rm div}A \in L^p_{loc}(\mathbb{R}^d, \mathbb{R}^d)$ (which holds for instance if $a_{ij} \in H_{loc}^{1,p}(\R^d)$, $1 \le i,j \le d$). Let $\psi \in L^1_{loc}(\mathbb{R}^d)$ with $\frac{1}{\psi} \in L^{\infty}_{loc}(\mathbb{R}^d)$, $\widehat{A} = \frac{1}{\psi} A$ and $\widehat{\bold{G}}$ be  a measurable vector field with $\psi \widehat{\bold{G}} \in L^p_{loc}(\mathbb{R}^d, \mathbb{R}^d)$. Then, the following hold:
\begin{itemize}
\item[(i)] There exists $\rho \in H^{1,2}_{loc}(\mathbb{R}^d)\cap C(\mathbb{R}^d)$ with $\rho(x)>0$ for all $x\in \mathbb{R}^d$ such that $\mathbf{B}:=\widehat{\bold{G}}-\beta^{\rho, A, \psi}$ satisfies $\bold{B} \in L^2_{loc}(\mathbb{R}^d, \mathbb{R}^d, \psi\rho dx)$, $\psi \bold{B} \in L^2_{loc}(\mathbb{R}^d, \mathbb{R}^d, \rho dx)$ and 
$$
\int_{\mathbb{R}^d} \langle \mathbf{B}, \nabla f \rangle \rho \psi dx = 0, \quad \text{ for all $f \in C_0^{\infty}(\mathbb{R}^d)$}.
$$
In particular, defining $\mu:=\rho dx$ and $\widehat{\mu}:=\psi d\mu$, {\bf (A)} is fulfilled.
\item[(ii)]
Let $\widehat{\mu}$ be defined as in  (i). If there exists a constant $M>0$ and $N_0 \in \mathbb{N}$, such that
\begin{equation}\label{growthconser}
-\frac{\langle \widehat{A}(x)x, x\rangle }{\|x\|^2} + \frac12 {\rm trace}(\widehat{A}(x))  + \langle \widehat{\bold{G}}(x), x \rangle \leq M \|x\|^2 \left( \ln \|x\| +1   \right), \quad \text{ for a.e. $x \in \mathbb{R}^d \setminus B_{N_0}$},
\end{equation}
then $(T_t)_{t>0}$ as defined in the paragraph right after Theorem \ref{pojjjde} is conservative.  Moreover, $\M$ as in Proposition \ref{Huntex} solves \eqref{weaksolution} with $\widehat{\sigma}\widehat{\sigma}^T=\widehat{A}$ and $\bold{G}=\widehat{\bold{G}}$ for strictly $\mathcal{E}$-q.e. (hence a.e.) $x \in \mathbb{R}^d$.
\end{itemize}
\end{proposition}
\begin{proof}
(i) First consider the case of $d \geq 2$.
By \cite[Theorem 2.27(i)]{LST22}, there exists $\rho \in H^{1,2}_{loc}(\mathbb{R}^d) \cap C(\mathbb{R}^d)$ with $\rho(x)>0$ for all $x \in \mathbb{R}^d$ such that
$$
\int_{\mathbb{R}^d} \Big \langle \frac{1}{2} A \nabla \rho-\rho\big( \psi \widehat{\bold{G}} - \frac12 {\rm div} A   \big), \nabla f   \Big \rangle dx =0, \quad \text{ for all $f \in C_0^{\infty}(\mathbb{R}^d)$}.
$$
Thus, for all $f \in C_0^{\infty}(\mathbb{R}^d)$ it holds that
$$
\int_{\mathbb{R}^d} \langle \widehat{\bold{G}} - \beta^{\rho, A, \psi}, \nabla f \rangle d\widehat{\mu} = \int_{\mathbb{R}^d} \Big \langle \rho \big(  \psi \widehat{\bold{G}}-\frac12 {\rm div} A \big) - \frac12 A \nabla \rho, \nabla f  \Big \rangle dx = 0.
$$
Moreover, $\psi \mathbf{B}=\psi \widehat{\bold{G}}-\frac12 {\rm div} A -\frac{1}{2\rho} A \nabla \rho \in L^2_{loc}(\mathbb{R}^d, \mathbb{R}^d)$. Since $\rho$ is locally bounded below and above by positive constants and $\frac{1}{\psi} \in L^{\infty}_{loc}(\mathbb{R}^d)$, we get $\psi \bold{B} \in L^2_{loc}(\mathbb{R}^d, \mathbb{R}^d, \rho dx)$ and $\bold{B}=(\frac{1}{\psi})\psi \bold{B} \in L^2_{loc}(\mathbb{R}^d, \mathbb{R}^d, \psi\rho dx)$. The assertion hence follows for $d\ge 2$. Next, consider the case of $d=1$. Then $A$ consists of exactly one function, which we denote by $a$ and moreover ${\rm div}A$ coincides with the weak derivative $a'$. We also write $\widehat{g}=\widehat{\bold{G}}$. Our goal is to find $\rho \in H^{1,2}_{loc}(\mathbb{R})\cap C(\mathbb{R})$ with $\rho(x)>0$ for all $x \in \mathbb{R}$, with $\bold{B}=0$, i.e.
$$
\psi \bold{B}=\psi \widehat{g}-\frac12 a' - \frac{1}{2\rho}  a \rho'=0.
$$
Now let 
$$
\rho(x):=\exp\left( \int_0^x \frac{2}{a(y)}\Big( \psi(y)\widehat{g}(y)-\frac12 a'(y)  \Big) dy \right), \quad x \in \mathbb{R}.
$$
Then, by the fundamental theorem of calculus, $\bold{B}=0$ and $\rho \in H^{1,p}_{loc}(\mathbb{R})\cap C(\mathbb{R})$ with $\rho(x)>0$ for all $x \in \mathbb{R}$, so that the assertion follows. \\
(ii) Define $u(x):=\ln (\|x\|^2 \vee N_0^2)$+2, $x \in \mathbb{R}^d$. Then, $u \in C^{\infty}(\mathbb{R}^d \setminus \overline{B}_{N_0}) \cap C(\mathbb{R}^d)$. Observe that $Lu \leq Mu$ a.e. on $\mathbb{R}^d \setminus \overline{B}_{N_0}$ is equivalent to \eqref{growthconser} since
\begin{equation*} \label{specifical}
Lu = -2 \frac{\langle \widehat{A}(x)x, x \rangle}{\|x\|^4} + \frac{{\rm trace} (  \widehat{A}(x) )}{\|x\|^2} + \frac{2 \langle \widehat{\bold{G}}(x),x  \rangle}{\|x\|^2} \quad \text{ a.e. on $\mathbb{R}^d \setminus \overline{B}_{N_0}$}.
\end{equation*}
Thus, the assertion follows from Theorem \ref{theoconserva}(ii). 
\end{proof}

\text{}\\
\noindent The proof of the following Theorem \ref{regulsthm} is based on the proof of \cite[Theorem 2.1]{St99}, but there are minor differences in details.
For instance, unlike the proof of \cite[Theorem 2.1]{St99}, we will not use the resolvent kernel via the Riesz representation theorem. Rather we will use $W^{2,p}$-regularity results developed by Krylov in \cite{Kry08}. We mention that it may not be possible to directly apply \cite[Theorem 2.1]{St99} to obtain $h \in \widehat{H}^{1,2}_{loc}(\mathbb{R}^d, \mu)$ in Theorem \ref{regulsthm}.
Therefore, we present a detailed proof by checking the calculation involving the measure $\widehat{\mu}$. Additionally, we consider in Theorem \ref{regulsthm} $h \in L_{loc}^{\infty}(\mathbb{R}^d)$ which is a slightly more general assumption than $h \in L^{\infty} (\mathbb{R}^d)$ in \cite[Theorem 2.1]{St99}.
\begin{theorem}\label{regulsthm}
Assume {\bf (A)} and consider the framework in Section \ref{framework}. Assume that for each $1 \leq i,j \leq d$, we have that $a_{ij}$ is locally H\"{o}lder continuous on $\mathbb{R}^d$, i.e.
for each open ball $B$ there exist $L_B>0$ and $\gamma_B \in (0,1]$ such that
$$
|a_{ij}(x)-a_{ij}(y)| \leq L_{B}\|x-y\|^{\gamma_B}, \quad \text{ for all $x,y \in \overline{B}$}.
$$
Let $c \in L_{loc}^1(\mathbb{R}^d, \widehat{\mu})$ and $\bold{G} \in L^2_{loc}(\mathbb{R}^d, \mathbb{R}^d, \widehat{\mu})$ with $\psi \mathbf{G} \in L^2_{loc}(\mathbb{R}^d, \mathbb{R}^d, \mu)$.
Assume that $h \in L_{loc}^{\infty}(\mathbb{R}^d)$
satisfies
\begin{equation}\label{equationlstarhmuhat}
\int_{\mathbb{R}^d}  \Big ( \frac12 \sum_{i,j=1}^d \frac{1}{\psi} a_{ij} \partial_{i} \partial_{j} u +
 \langle \mathbf{G}, \nabla u    \rangle  + cu \Big ) h\, d \widehat{\mu} = 0, \quad \text{ for all $u \in C_0^{\infty}(\mathbb{R}^d)$}.
\end{equation}
Then, $h \in \widehat{H}^{1,2}_{loc}(\mathbb{R}^d, \mu)$ and
$$
\mathcal{E}^0 (u, h) - \int_{\mathbb{R}^d} \langle \mathbf{G} - \beta^{\rho, A, \psi}, \nabla u \rangle h \, d \widehat{\mu} - \int_{\mathbb{R}^d}cuh \,d \widehat{\mu} = 0 \quad \text{ for all $u \in \widehat{H}^{1,2}_0(\mathbb{R}^d, \mu)_0$}.
$$
\end{theorem}
\begin{proof}
Consider $h \in L_{loc}^{\infty}(\mathbb{R}^d)$ as in the statement and let $r>0$ and $\chi \in C_0^{\infty}(\mathbb{R}^d)$ with $\text{supp}(\chi) \subset B_r$ be arbitrarily given. To show the assertion, it is enough to show that $\chi h \in \widehat{H}^{1,2}_0(\mathbb{R}^d, \mu)$ since the rest of the assertion follows from integration by parts and the approximation of $u\in \widehat{H}^{1,2}(\mathbb{R}^d, \mu)_0$ as in Lemma \ref{applempn} by  $C_0^{\infty}(\mathbb{R}^d)$-functions. 
Let $L^{A} = \frac12 \sum_{i,j=1}^d a_{ij} \partial_i \partial_j$ 
and $p \in (d, \infty)$ be fixed. By \cite[Theorems 8.5.3, 8.5.6]{Kry08} and Sobolev's embedding, there exists a constant $c_{d,p}>0$ which only depends on $d$, $p$ and there exist constants $\lambda_0>0$ and $N>0$ which depend only on $\lambda_{B_r}$, $\Lambda_{B_r}$, $L_{B_r}$, $p$ and $d$ such that if $\alpha \geq \lambda_0$ and $f \in L^p(B_r)$, then there exists a unique $V_{\alpha} f \in H^{2,p}(B_r) \cap H^{1,2}_0(B_r) \cap C(\overline{B}_r)$ such that
\begin{equation*} 
(\alpha - L^{A}) V_{\alpha} f= f \quad \text{ on } \; B_r
\end{equation*}
and 
\begin{equation} \label{lpcomplete}
\| V_{\alpha} f \|_{C(\overline{B}_r)} \leq  c_{d,p}N \Big (  1 + \frac{1}{\alpha}  \Big ) \| f\|_{L^p(B_r)}.
\end{equation}
and it follows from \cite[Theorem 11.2.1]{Kry08} that for any $f \in L^{\infty}(B_r)$ and $\alpha \geq \lambda_0$,
\begin{equation} \label{submarkres}
\| \alpha V_{\alpha} f \|_{C(\overline{B}_r)} \leq \|f\|_{L^{\infty}(B_r)}.
\end{equation}
Meanwhile, it follows from \cite[Theorem 6.8]{Gilbarg} and \cite[Theorem 8.5.3]{Kry08} that for any $\alpha \geq \lambda_0$ and $g \in C_0^{\infty}(\mathbb{R}^d)$ we have 
\begin{equation} \label{glbscha}
V_{\alpha}g \in C^2(\overline{B}_r) \cap H^{1,2}_0(B_r).
\end{equation}
Since $h \in L^{\infty}(B_r)$, using a mollification approximation of $h 1_{B_r}$, there exists a sequence of functions $(f_n)_{n \geq 1} \subset C_0^{\infty}(\mathbb{R}^d)$ with $\text{supp}(f_n) \subset B_{2r}$ such that $\lim_{n \rightarrow \infty} f_n = h$, a.e. on $B_r$ and that
\begin{equation} \label{submarkes2}
 \| f_n \|_{L^{\infty}(B_r)} \leq \|h \|_{L^{\infty}(B_r)}, \quad \text{ for all $n \geq 1$}.
\end{equation}
Using \eqref{glbscha}, \eqref{lpcomplete}, and Lebesgue's theorem, we obtain that
 \begin{equation} \label{ptwiselimit}
V_{\alpha} f_n \in C^2(\overline{B}_r)\; \text{ for all $n \geq 1$}, \quad \quad  \lim_{n \rightarrow \infty}V_{\alpha} f_n =V_{\alpha} h \;\; \text{ in $C_b(\overline{B}_r)$}.
\end{equation}
Moreover, \eqref{submarkres} and \eqref{submarkes2} imply that
\begin{equation*} 
\|\alpha V_{\alpha} f_n \|_{L^{\infty}(B_r)} \leq
\|f_n\|_{L^{\infty}(B_r)} \leq\|h\|_{L^{\infty}(B_r)},\;\;\quad 
\|\alpha V_{\alpha} h \|_{L^{\infty}(B_r)} \leq \|h\|_{L^{\infty}(B_r)}.
\end{equation*}
Observe that $\chi \alpha V_{\alpha} f_n \in C_0^2(\mathbb{R}^d) \subset D(L^0) \subset D(\mathcal{E}^0)$, and hence by \eqref{submarkres}
\begin{align}
&\mathcal{E}^0 (\chi \alpha V_{\alpha} f_n, \chi \alpha V_{\alpha} f_n ) = -\int_{\mathbb{R}^d} L^0 (\chi \alpha V_{\alpha} f_n) \cdot (\chi \alpha V_{\alpha} f_n) d  \widehat{\mu} \nonumber \\
&=-\int_{\mathbb{R}^d} \frac{1}{\psi} L^{A} \left( \chi \alpha V_{\alpha} f_n \right) \cdot (\chi \alpha V_{\alpha} f_n) d \widehat{\mu} - \int_{\mathbb{R}^d} \langle  \beta^{\rho, A, \psi}, \nabla (\chi \alpha V_{\alpha} f_n)  \rangle\, (\chi \alpha V_{\alpha} f_n) d\widehat{\mu}  \nonumber \\
&= - \int_{\mathbb{R}^d} \chi (L^{A} \chi) \cdot(\alpha V_{\alpha} f_n)^2 d\mu  - \int_{\mathbb{R}^d}  L^{A} (\alpha V_{\alpha} f_n) \cdot (\chi^2 \alpha V_{\alpha} f_n) d \mu  \nonumber \\
&\qquad  - \int_{\mathbb{R}^d} \langle A \nabla \chi, \nabla \alpha V_{\alpha} f_n\rangle \chi \alpha V_{\alpha} f_n d \mu  - \int_{\mathbb{R}^d} \langle  \beta^{\rho, A}, \nabla (\chi \alpha V_{\alpha} f_n)  \rangle\, (\chi \alpha V_{\alpha} f_n) d \mu \nonumber \\
&= - \int_{\mathbb{R}^d} \chi (L^{A} \chi) \cdot(\alpha V_{\alpha} f_n)^2 d\mu  -\alpha \int_{\mathbb{R}^d} (\alpha V_{\alpha} f_n - f_n) \chi^2 \alpha V_{\alpha} f_n   d \mu \nonumber  \\
&\qquad  - \int_{\mathbb{R}^d} \langle A \nabla \chi, \nabla (\chi \alpha V_{\alpha} f_n) \rangle\alpha V_{\alpha} f_n d \mu 
 + \int_{\mathbb{R}^d} \langle A \nabla \chi, \nabla \chi \rangle (\alpha V_{\alpha} f_n)^2 d \mu  \nonumber \\
 &\qquad \quad - \int_{\mathbb{R}^d} \langle  \beta^{\rho, A}, \nabla (\chi \alpha V_{\alpha} f_n)  \rangle\, (\chi \alpha V_{\alpha} f_n) d \mu  \label{intermedien}   \\
 & \leq  \| \chi \|_{L^{1}(B_r, \mu)} \| L^{A} \chi \|_{L^{\infty}(B_r)} \|h\|^2_{L^{\infty}(B_r)} + 2\alpha \|h\|^2_{L^{\infty}(B_r)} \|\chi^2 \|_{L^1(B_r, \mu)}\nonumber \\
& \quad + 2 \| \chi \|_{L^{\infty}(B_r)} \| h \|_{L^{\infty}(\mathbb{R}^d)}  \mathcal{E}^0 (\chi, \chi)^{1/2} \cdot \mathcal{E}^0 (\chi \alpha V_{\alpha} h, \chi \alpha V_{\alpha} h)^{1/2}+ 2\| h  \|^2_{L^{\infty}(B_r)} \mathcal{E}(\chi, \chi) \nonumber \\
 & \quad \quad \quad + \| \beta^{\rho, A} \|_{L^2(B_r, \mu)}  \| h \|_{L^{\infty}(B_r)} 
  \sqrt{2} (\lambda_{B_r})^{-1/2}  \mathcal{E}(\chi \alpha V_{\alpha} f_n, \chi \alpha V_{\alpha} f_n)^{1/2} \nonumber  \\
 &  \leq  C_1+ C_2 \mathcal{E}^0 (\chi \alpha V_{\alpha}f_n, \chi \alpha V_{\alpha}f_n    )^{1/2}, \nonumber
\end{align}
where $C_1=\| \chi \|_{L^{1}(B_r, \mu)} \| L^{A} \chi \|_{L^{\infty}(\mathbb{R}^d)} \|h \|^2_{L^{\infty}(\mathbb{R}^d)} + 2\alpha \|h\|^2_{L^{\infty}(B_r)} \|\chi^2 \|_{L^1(B_r)}+2\| h \|^2_{L^{\infty}(B_r)} \mathcal{E}(\chi, \chi)$ and \\
$C_2=2 \| \chi \|_{L^{\infty}(B_r)} \| h \|_{L^{\infty}(\mathbb{R}^d)} \mathcal{E}^0(\chi, \chi)^{1/2} +\| \beta^{\rho, A} \|_{L^2(B_r, \mu)}  \| h \|_{L^{\infty}(B_r)} \sqrt{2} (\lambda_{B_r})^{-1/2}$.
Therefore,
$$
 \mathcal{E}_1^0 (\chi \alpha V_{\alpha}f_n, \chi \alpha V_{\alpha}f_n    ) \leq 2C_1+C_2^2 + \|h\|^2_{L^{\infty}(\mathbb{R}^d)} \| \chi \|^2_{L^2(B_r, \widehat{\mu})}.
$$
Thus, using the Banach-Alaoglu theorem and \eqref{ptwiselimit}, we obtain that $\chi \alpha V_{\alpha} h \in D(\mathcal{E}^0)$ and
$$
\lim_{n \rightarrow \infty} \chi \alpha V_{\alpha} f_n = \chi \alpha V_{\alpha} h, \quad \; \text{ weakly in $D(\mathcal{E}^0)$}.
$$
Note that
$$
-\alpha \int_{\mathbb{R}^d} (\alpha V_{\alpha} h - h ) (\alpha V_{\alpha} h - h ) \chi^2 d\mu \leq 0,
$$
and hence
\begin{align}
&- \alpha \int_{\mathbb{R}^d} (\alpha V_{\alpha} h - h) \cdot \alpha V_{\alpha} h \cdot \chi^2 d\mu \leq - \alpha \int_{\mathbb{R}^d} (\alpha V_{\alpha} h - h) h \chi^2 d\mu \nonumber \\
&= - \lim_{n \rightarrow \infty} \alpha \int_{\mathbb{R}^d} (\alpha V_{\alpha} f_n - f_n) h \chi^2 d\mu = \lim_{n \rightarrow \infty} - \int_{\mathbb{R}^d} L^{A} (\alpha V_{\alpha} f_n) h \chi^2 d\mu \nonumber \\
& = \lim_{n \rightarrow \infty} \big (- \int_{\mathbb{R}^d} L^A (\chi^2 \alpha V_{\alpha} f_n)\, h  \, d\mu + \int_{\mathbb{R}^d} \langle A \nabla (\chi^2), \nabla \alpha V_{\alpha} f_n \rangle h d\mu + \int_{\mathbb{R}^d} L^{A} (\chi^2) \cdot \alpha V_{\alpha} f_n \cdot h\, d \mu  \big )\nonumber \\
& \underset{\text{by }\eqref{equationlstarhmuhat}}{=} \lim_{n \rightarrow \infty}  \big (\int_{\mathbb{R}^d} \psi c (\chi^2 \alpha V_{\alpha} f_n) h\, d\mu + \int_{\mathbb{R}^d} \langle \psi \mathbf{G}, \nabla (\chi^2 \alpha V_{\alpha} f_n) \rangle h\, d\mu+ 2 \int_{\mathbb{R}^d} \langle A \nabla \chi, \nabla \alpha V_{\alpha} f_n \rangle \chi h d\mu\nonumber  \\
&\qquad \quad + \int_{\mathbb{R}^d} L^{A} (\chi^2) \cdot \alpha V_{\alpha} f_n \cdot h\, d \mu \big ) \nonumber \\
&= \lim_{n \rightarrow \infty} \big (\int_{\mathbb{R}^d} c (\chi^2 \alpha V_{\alpha} f_n) h\,d\widehat{\mu} + \int_{\mathbb{R}^d} \langle\psi \mathbf{G}, \nabla (\chi \alpha V_{\alpha} f_n) \rangle \chi h\, d\mu+ 2 \int_{\mathbb{R}^d} \langle A \nabla \chi, \nabla (\chi \alpha V_{\alpha} f_n) \rangle  h d\mu  \nonumber \\
&\quad \quad + \int_{\mathbb{R}^d} L^{A} (\chi^2) \cdot \alpha V_{\alpha} f_n \cdot h\, d \mu + \int_{\mathbb{R}^d} \langle \psi \mathbf{G}, \nabla \chi \rangle \chi \alpha V_{\alpha} f_n  \cdot  hd\mu  - 2 \int_{\mathbb{R}^d} \langle A \nabla \chi, \nabla \chi \rangle \alpha V_{\alpha} f_n \cdot h \,d\mu \big ) \nonumber  \\
&= \int_{\mathbb{R}^d} c (\chi^2 \alpha V_{\alpha} h)  h\, d \widehat{\mu} + \int_{\mathbb{R}^d} \langle \psi \mathbf{G}, \nabla (\chi \alpha V_{\alpha} h) \rangle \chi h\, d\mu+ 2 \int_{\mathbb{R}^d} \langle A \nabla \chi, \nabla (\chi \alpha V_{\alpha} h)\rangle  h d\mu \nonumber \\
&\qquad \quad + \int_{\mathbb{R}^d} L^{A} (\chi^2) \cdot \alpha V_{\alpha} h \cdot h\, d \mu + \int_{\mathbb{R}^d} \langle  \psi \mathbf{G}, \nabla \chi \rangle \chi \alpha V_{\alpha} h d\mu   - 2 \int_{\mathbb{R}^d} \langle A \nabla \chi, \nabla \chi \rangle \alpha V_{\alpha} h \cdot  h   \,d\mu  \nonumber \\
&\leq \|\chi^2 c\|_{L^1(B_r, \widehat{\mu})} \|h\|^2_{L^{\infty}(\mathbb{R}^d)} + \|\psi \mathbf{G} \|_{L^2(B_r, \mu)} \|\chi h \|_{L^{\infty}(\mathbb{R}^d)} \sqrt{2}(\lambda_{B_r})^{-1/2} \mathcal{E}^0(\chi \alpha V_{\alpha} h, \chi \alpha V_{\alpha} h )^{1/2} \nonumber  \\
& \quad +4 \mathcal{E}^0 (\chi, \chi)^{1/2} \|h\|_{L^{\infty}(\mathbb{R}^d)} \mathcal{E}^0 (\chi \alpha V_{\alpha} h, \chi \alpha V_{\alpha} h  )^{1/2}  + \| L^{A} (\chi^2) \|_{L^1(B_r, \mu)} \|h\|^2_{L^{\infty}(\mathbb{R}^d)}\nonumber \\
& \qquad  + \| \psi \mathbf{G} \|_{L^1(B_r, \mu)} \|\chi \nabla \chi \|_{L^{\infty}(\mathbb{R}^d)} \| h \|_{L^{\infty}(\mathbb{R}^d)} + 4 \mathcal{E}^0 (\chi, \chi) \| h\|^2_{L^{\infty}(\mathbb{R}^d)} \nonumber \\
& =C_3 \mathcal{E}^0 (\chi \alpha V_{\alpha} h, \chi \alpha V_{\alpha} h)^{1/2} + C_4, \label{criticineq}
\end{align}
where 
$C_3:=\|\psi \mathbf{G} \|_{L^2(B_r, \mu)} \|\chi h \|_{L^{\infty}(\mathbb{R}^d)} \sqrt{2}(\lambda_{B_r})^{-1/2} +4 \mathcal{E}^0 (\chi, \chi)^{1/2} \|h\|_{L^{\infty}(\mathbb{R}^d)}$ \\
and $C_4:=\|\chi^2 c\|_{L^1(B_r, \widehat{\mu})} \|h\|^2_{L^{\infty}(\mathbb{R}^d)}+ \| L^{A} (\chi^2) \|_{L^1(B_r, \mu)} \|h\|^2_{L^{\infty}(\mathbb{R}^d)} + \| \psi \mathbf{G} \|_{L^1(B_r, \mu)} \|\chi \nabla \chi \|_{L^{\infty}(\mathbb{R}^d)} \| h \|_{L^{\infty}(\mathbb{R}^d)} + 4 \mathcal{E}^0 (\chi, \chi) \| h\|^2_{L^{\infty}(\mathbb{R}^d)}$. Using \eqref{intermedien}, the weak convergence, and \eqref{criticineq}, it follows that
\begin{align*}
&\mathcal{E}^0(\chi \alpha V_{\alpha} h, \chi \alpha V_{\alpha} h ) \leq \liminf_{n \rightarrow \infty} \mathcal{E}^0(\chi \alpha V_{\alpha} f_n, \chi \alpha V_{\alpha} f_n ) \\
 &=- \int_{\mathbb{R}^d} \chi (L^{A} \chi) \cdot(\alpha V_{\alpha} h)^2 d\mu  -\alpha \int_{\mathbb{R}^d} (\alpha V_{\alpha} h - h) \chi^2 \alpha V_{\alpha} h   d \mu \nonumber  \\
 &  \qquad - \int_{\mathbb{R}^d} \langle A \nabla \chi, \nabla (\chi \alpha V_{\alpha} h) \rangle \chi \alpha V_{\alpha} h d \mu + \int_{\mathbb{R}^d} \langle A \nabla \chi, \nabla \chi \rangle (\alpha V_{\alpha} h)^2 d \mu  \nonumber \\
 &\qquad \quad - \int_{\mathbb{R}^d} \langle  \beta^{\rho, A}, \nabla (\chi \alpha V_{\alpha} h)  \rangle\, (\chi \alpha V_{\alpha} h) d \mu  \\
 & \leq \|h\|^2_{L^{\infty}(\mathbb{R}^d)} \| \chi L^A \chi \|_{L^1(B_r, \mu)}+C_3 \mathcal{E}^0 (\chi \alpha V_{\alpha} h, \chi \alpha V_{\alpha} h)^{1/2} + C_4 \\
&+ 2 \| \chi \|_{L^{\infty}(B_r)} \| h \|_{L^{\infty}(\mathbb{R}^d)}  \mathcal{E}^0 (\chi, \chi)^{1/2} \cdot \mathcal{E}^0 (\chi \alpha V_{\alpha} h, \chi \alpha V_{\alpha} h)^{1/2}+ 2\| h  \|^2_{L^{\infty}(B_r)} \mathcal{E}(\chi, \chi) \nonumber \\
 & \quad \quad \quad + \| \beta^{\rho, A} \|_{L^2(B_r, \mu)}  \| h \|_{L^{\infty}(B_r)} 
  \sqrt{2} (\lambda_{B_r})^{-1/2}  \mathcal{E}^0(\chi \alpha V_{\alpha} h, \chi \alpha V_{\alpha} h)^{1/2}  \\
 & \le  C_5 \mathcal{E}^0 (\chi \alpha V_{\alpha}h, \chi \alpha V_{\alpha}h)^{1/2} +C_6,
\end{align*}
where $C_5= C_2+C_3$ and  $C_6=C_1+C_4$. Thus,
$$
\mathcal{E}^0_1 (\chi \alpha V_{\alpha} h, \chi \alpha V_{\alpha} h) \leq C_5^2 + 2C_6 + \|h\|_{L^{\infty}(\mathbb{R}^d)} \| \chi \|^2_{L^2(B_r, \widehat{\mu})}=:C_7,
$$
where $C_7$ is a constant independent of $\alpha$. Hence, by the Banach-Alaoglu theorem, there exist $\widehat{h} \in D(\mathcal{E}^0)$
and an increasing sequence $(\alpha_k)_{k \geq 1}$ in $[\lambda_0, \infty)$
with $\alpha_k \nearrow \infty$ as $k \rightarrow \infty$
such that 
\begin{equation*} 
\lim_{k \rightarrow \infty} \chi \alpha_k V_{\alpha_k} h = \widehat{h} \quad \text{ weakly in $D(\mathcal{E}^0)$}.
\end{equation*}
Now let $u \in C_0^{\infty}(\mathbb{R}^d)$ be arbitrarily given.
Then, 
\begin{align*}
&  \big | \int_{\mathbb{R}^d} (\widehat{h}- \chi h) u d\mu\big |  = \lim_{k \rightarrow \infty} \big |  \int_{\mathbb{R}^d} \chi (\alpha_k V_{\alpha_k} h - h) u d\mu \big |  = \lim_{k \rightarrow \infty} \lim_{n \rightarrow \infty} \big |  \int_{\mathbb{R}^d} \chi (\alpha_k V_{\alpha_k} f_n - f_n) u d\mu \big |  \\
& = \lim_{k \rightarrow \infty} \lim_{n \rightarrow \infty} \big |  \int_{\mathbb{R}^d} \chi (L^A V_{\alpha_k} f_n) u d\mu \big | 
=\lim_{k \rightarrow \infty} \lim_{n \rightarrow \infty} \big |  \int_{\mathbb{R}^d} \chi \Big(  \psi L^0 (V_{\alpha_k} f_n) - \langle \beta^{\rho, A}, \nabla V_{\alpha_k} f_n   \rangle     \Big) u d\mu \big |  \\
& =  \lim_{k \rightarrow \infty} \lim_{n \rightarrow \infty} \big |  \int_{\mathbb{R}^d} \psi L^0(\chi u) \cdot V_{\alpha_k} f_n d\mu - \int_{\mathbb{R}^d} \langle \beta^{\rho, A}, \nabla (\chi V_{\alpha_k} f_n) \rangle u d\mu  + \int_{\mathbb{R}^d} \langle \beta^{\rho, A}, \nabla \chi \rangle V_{\alpha_k} f_n \cdot u d\mu \big | \\
&= \lim_{k \rightarrow \infty}\big |  \int_{\mathbb{R}^d} \psi L^0(\chi u) \cdot V_{\alpha_k} h \, d\mu - \int_{\mathbb{R}^d} \langle \beta^{\rho, A}, \nabla (\chi V_{\alpha_k} h) \rangle u \,d\mu  + \int_{\mathbb{R}^d} \langle \beta^{\rho, A}, \nabla \chi \rangle V_{\alpha_k} h \cdot u \,d\mu \big | 
\\
&\underset{\text{by }\eqref{submarkres}}{\leq}   \lim_{k \rightarrow \infty} \frac{1}{\alpha_k}\| \psi L^0 (\chi u) \|_{L^1(B_r, \mu)}\| \alpha_k V_{\alpha_k} h\|_{L^{\infty}(B_r)} + \frac{1}{\alpha_k} \| \beta^{\rho, A}  \|_{L^2(B_r, \mu)} \sqrt{2} \lambda_{B_r}^{-1/2} \mathcal{E}^0 (\chi \alpha_k V_{\alpha_k} h,\chi \alpha_k V_{\alpha_k} h)^{1/2} \\
& \qquad + \frac{1}{\alpha_k}\| \beta^{\rho, A}\|_{L^1(B_r, \mu)} \|\chi \|_{L^{\infty}(B_r)} \|\alpha_kV_{\alpha_k} h \|_{L^{\infty}(B_r)} \\
&\leq \lim_{k \rightarrow \infty} \frac{1}{\alpha_k} \big (\|  \psi L^0 (\chi u) \|_{L^1(B_r, \mu)} \|h\|_{L^{\infty}(B_r)} + \| \beta^{\rho, A}  \|_{L^2(B_r, \mu)} \sqrt{2} \lambda_{B_r}^{-1/2}  C_7 + \| \beta^{\rho, A}\|_{L^1(B_r, \mu)} \|\chi \|_{L^{\infty}(B_r)} \|h\|_{L^{\infty}(B_r)}\big ) \\
&=0.
\end{align*} 
Therefore, we get $\chi h = \widehat{h} \in D(\mathcal{E}^0)$, so that $\chi  h\in \widehat{H}^{1,2}_0(\mathbb{R}^d, \mu)$, as desired.
\end{proof}

\begin{cor}\label{cor4.4} 
Assume {\bf (A)} and the H\"{o}lder continuity of Theorem \ref{regulsthm}. Let $(\overline{L}, D(\overline{L}))$, $(L, D(L^0)_{0,b})$ and $(\overline{T}_t)_{t>0}$ be defined as in Theorem \ref{mainijcie}. Then, $(L, C_0^{\infty}(\mathbb{R}^d))$ is $L^1(\mathbb{R}^d, \widehat{\mu})$-unique if and only if  $(L, D(L^0)_{0,b})$ is $L^1(\mathbb{R}^d, \widehat{\mu})$-unique.
\end{cor}
\begin{proof}
It directly follows from the definition that if $(L, C_0^{\infty}(\mathbb{R}^d))$ is $L^1(\mathbb{R}^d, \widehat{\mu})$-unique, then $(L, D(L^0)_{0,b})$ is $L^1(\mathbb{R}^d, \widehat{\mu})$-unique. \\
Assume that $(L, D(L^0)_{0,b})$ is $L^1(\mathbb{R}^d, \widehat{\mu})$-unique. Moreover, assume that there exists $h \in L^{\infty}(\mathbb{R}^d)$ such that
$$
\displaystyle \int_{\mathbb{R}^d} (1-L) u \cdot h\, d \widehat{\mu}=0 \quad \text{ for all $u \in C_0^{\infty}(\mathbb{R}^d)$}.
$$
To show that $(L, C_0^{\infty}(\mathbb{R}^d))$ is $L^1(\mathbb{R}^d, \widehat{\mu})$-unique, it suffices to show that $h=0$ a.e. (see beginning of Section \ref{l1closeduni}). Using Theorem \ref{regulsthm} where $\mathbf{G}$ and $c$ are replaced by $\beta^{\rho, A, \psi} + \mathbf{B}$ and $-1$, respectively, we obtain that $h \in \widehat{H}^{1,2}_{loc}(\mathbb{R}^d, \mu)$ and 
$$
\mathcal{E}_1^0 (u, h) - \int_{\mathbb{R}^d} \langle \mathbf{B}, \nabla u \rangle h \, d \widehat{\mu} = 0 \quad \text{ for all $u \in \widehat{H}^{1,2}_0(\mathbb{R}^d, \mu)_0$}.
$$
Using integration by parts on the left hand side above, we get
$$
\displaystyle \int_{\mathbb{R}^d} (1-L) u \cdot h\, d \widehat{\mu}=0 \quad \text{ for all $u \in D(L^0)_{0,b}$}.
$$
Since $(L, D(L^0)_{0,b})$ is $L^1(\mathbb{R}^d, \widehat{\mu})$-unique, we finally get $h=0$ a.e, as desired.
\end{proof}

\section{Uniqueness in law} \label{conditionaluniquness}
In this section we show how to use the results on $L^1(\mathbb{R}^d, \widehat{\mu})$-uniqueness of Section \ref{l1closeduni} to obtain a conditional uniqueness in law result. We will assume throughout Section \ref{conditionaluniquness} that {\bf (A)} holds. We fix  a right process (for the definition of right process used here, cf. e.g. \cite[Definition 3.5]{LST22})
$$
\widetilde{\mathbb{M}}= (\widetilde{\Omega}, \widetilde{\mathcal{F}},  (\widetilde{\mathcal{F}}_t)_{t \geq 0},  (\widetilde{X}_t)_{t \geq 0}, (\widetilde{\mathbb{P}}_x)_{x \in \mathbb{R}^d \cup \{\Delta \}}  )
$$
with state space $\mathbb{R}^d$. For a $\sigma$-finite measure $\nu$ on $(\mathbb{R}^d, \mathcal{B}(\mathbb{R}^d))$, we define a measure $\widetilde{\mathbb{P}}_{\nu}$ on  $(\widetilde{\Omega}, \widetilde{\mathcal{F}})$ by
$$
\widetilde{\mathbb{P}}_{\nu} (\cdot) := \int_{\mathbb{R}^d} \widetilde{\mathbb{P}}_x (\cdot) \nu(dx).
$$
Let $\widetilde{\mathbb{E}}_{\nu}$ denote the expectation with respect to  $\widetilde{\mathbb{P}}_{\nu}$. We assume that $\widehat{\mu}$ (as defined in{\bf (A)}) is  {\bf a sub-invariant measure} for $\widetilde{\mathbb{M}}$, i.e.
\begin{equation}\label{subinvprocdef}
\widetilde{\mathbb{P}}_{\widehat{\mu}} (\widetilde{X}_t\in A) \le \widehat{\mu}(A), \quad A \in \mathcal{B}(\mathbb{R}^d), \ t\ge 0.  
\end{equation}
If in the above inequality \lq\lq $\le$\rq\rq\ is replaced by  \lq\lq $=$\rq\rq, then $\widehat{\mu}$ is called {\bf invariant measure} for $\widetilde{\mathbb{M}}$. Let 
$$
p^{\widetilde{\mathbb{M}}}_t f(x) = \widetilde{\mathbb{E}}_x [ f(\widetilde{X}_t)], \qquad f \in \mathcal{B}_b(\mathbb{R}^d), \,x \in \mathbb{R}^d,\, t>0,
$$
where $\widetilde{\mathbb{E}}_x$ is the expectation with respect to $\widetilde{\mathbb{P}}_x$. A direct consequence of the sub-invariance of $\widehat{\mu}$ for $\widetilde{\mathbb{M}}$ and the right continuity of the sample paths of  $\widetilde{\mathbb{M}}$ is that $(p^{\widetilde{\M}}_t)_{t \geq 0}|_{L^1(\R^d, \widehat{\mu})_b}$ uniquely extends to a sub-Markovian $C_0$-semigroup of contractions $(S_t)_{t \geq 0}$ on $L^1(\R^d, \widehat{\mu})$ (see the proof of \cite[Proposition 3.57]{LST22}). Let $(\widetilde{R}_{\alpha})_{\alpha>0}$ be the resolvent of $(S_t)_{t>0}$. Then, for each $\alpha>0$ and $f \in L^1(\mathbb{R}^d, \widehat{\mu})$
$$
\widetilde{R}_{\alpha} f(x) = \int_{0}^{\infty} e^{-\alpha t} S_t f(x) ds, \quad  \text{ for $\widehat{\mu}$-a.e. $x \in \mathbb{R}^d$.  }
$$
Let $w \in \mathcal{B}_b(\mathbb{R}^d)$, $w\ge 0$ and $w\widehat{\mu}:=wd \widehat{\mu}$. Using the $L^1(\mathbb{R}^d, \widehat{\mu})$-contraction property of $(\widetilde{R}_{\alpha})_{\alpha>0}$, we get
\begin{align}
\widetilde{\mathbb{E}}_{w\widehat{\mu}}\Big [  \int_0^t |f|(\widetilde{X}_s)ds \Big ]&\le \|w\|_{L^{\infty}(\mathbb{R}^d)} \int_{\mathbb{R}^d} \widetilde{\mathbb{E}}_{x}\Big [  \int_0^t |f|(\widetilde{X}_s)ds \Big] \widehat{\mu}(dx) \nonumber  \\
&\le e^{t}\|w\|_{L^{\infty}(\mathbb{R}^d)} \int_{\mathbb{R}^d} \widetilde{\mathbb{E}}_{x}\Big [   \int_0^{\infty}  e^{-s} |f|(\widetilde{X}_s)ds \Big] \widehat{\mu}(dx) \nonumber \\
&= e^t \|w\|_{L^{\infty}(\mathbb{R}^d)} \int_{\mathbb{R}^d} \widetilde{R}_1 |f| \, d\widehat{\mu}  \nonumber \\
& \leq e^t \|w\|_{L^{\infty}(\mathbb{R}^d)} \|f\|_{L^1(\mathbb{R}^d, \widehat{\mu})}. \label{eq:subinv 1}
\end{align}
Thus, for each $f \in L^1(\mathbb{R}^d, \widehat{\mu})$, $(\int_0^t f(\widetilde{X}_s) ds)_{t \geq 0}$ is well-defined $\widetilde{\mathbb{P}}_{w \widehat{\mu}}$-a.e. and $\widetilde{\mathbb{P}}_{w \widehat{\mu}}$-a.e. independent of the chosen Borel measurable version $f \in L^1(\mathbb{R}^d, \widehat{\mu})$.

\begin{defn}\label{defnofmartin}
(i) We say that $\widetilde{\mathbb{M}}$ has {\bf continuous paths with respect to $\widehat{\mu}$}, if 
$$
\widetilde{\P}_x(t\mapsto \widetilde{X}_t \text{ is continuous from }[0,\infty)\text{ to }\R^d)=1\ \text{for a.e. }x\in \R^d,
$$ 
or equivalently $\widetilde{\P}_{v \widehat{\mu}}(t\mapsto \widetilde{X}_t \text{ is continuous from }[0,\infty)\text{ to }\R^d)=1$ for any $v \in L^1(\R^d,\widehat{\mu})_b$, $v\ge 0$, such that $\int_{\mathbb{R}^d} v d\widehat{\mu} = 1$.\\
(ii) We say that {\bf $\widetilde{\mathbb{M}}$ solves the martingale problem for $(L, C_0^{\infty}(\mathbb{R}^d))$ with respect to $\widehat{\mu}$}, if for all $u \in C_0^{\infty}(\mathbb{R}^d)$,
$\widetilde{M}_t^u:=u(\widetilde{X}_t)- u(\widetilde{X}_0) - \int_0^t Lu (\widetilde{X}_s) \, ds$, $t \geq 0$ is a continuous $(\widetilde{\mathcal{F}}_t)_{t \geq 0}$-martingale under $\widetilde{\mathbb{P}}_{v \widehat{\mu}}$ for any $v \in L^1(\R^d,\widehat{\mu})_b$, $v\ge 0$, such that $\int_{\mathbb{R}^d} v d\widehat{\mu} = 1$. \\
(Note here that by the discussion above, in particular \eqref{eq:subinv 1}, $(\int_0^t Lu (\widetilde{X}_s) \, ds)_{t\ge 0}$ is $\widetilde{\mathbb{P}}_{v\widehat{\mu}}$-a.s.  well-defined and continuous, independently of the Borel measurable version $Lu\in L^{2}(\R^d,\widehat{\mu})_0\subset L^{1}(\R^d,\widehat{\mu})$ chosen.)\\
(iii) We say that {\bf $\widetilde{\mathbb{M}}$ solves \eqref{weaksolution} with respect to $\widehat{\mu}$}, if $\widetilde{\mathbb{M}}$ has continuous paths with respect to $\widehat{\mu}$ and $\widetilde{\mathbb{M}}$ solves \eqref{weaksolution} with respect to $\widetilde {\P}_{v\widehat{\mu}}$ (instead of $\P_x$) for any $v \in L^1(\R^d,\widehat{\mu})_b$, $v\ge 0$, such that $\int_{\mathbb{R}^d} v d\widehat{\mu} = 1$.\\
(Note here that analogously to \cite[Lemma 3.17(iii)]{LST22}, if $\widetilde{\mathbb{M}}$ has continuous paths with respect to $\widehat{\mu}$, then $(\int_0^t f(\widetilde{X}_s) \, ds)_{t\ge 0}$ is $\widetilde{\mathbb{P}}_{v\widehat{\mu}}$-a.s.  well-defined and continuous, for any $v \in L^1(\R^d,\widehat{\mu})_b$, $v\ge 0$, such that $\int_{\mathbb{R}^d} v d\widehat{\mu} = 1$, independently of the Borel measurable version $f\in  L^{1}_{loc}(\R^d,\widehat{\mu})$ chosen. In particular, the stochastic integral and drift parts of \eqref{weaksolution} are well-defined $\widetilde{\mathbb{P}}_{v\widehat{\mu}}$-a.s. for any $v \in L^1(\R^d,\widehat{\mu})_b$, $v\ge 0$, such that $\int_{\mathbb{R}^d} v d\widehat{\mu} = 1$, if $\widetilde{\mathbb{M}}$ solves \eqref{weaksolution} with respect to $\widehat{\mu}$.)\\
(iv) We say that uniqueness in law holds for \eqref{weaksolution} with respect to $\widehat{\mu}$, if uniqueness in law holds for \eqref{weaksolution} with initial condition $v\widehat{\mu}$ for any $v \in L^1(\R^d,\widehat{\mu})_b$, $v\ge 0$, such that $\int_{\mathbb{R}^d} v d\widehat{\mu} = 1$. (As we shall see in Theorem \ref{uniqueness in  law}(ii), uniqueness in law of \eqref{weaksolution} with respect to $\widehat{\mu}$ is equivalent to uniqueness in law of \eqref{weaksolution} with respect to $\delta_x$ for $x\in \R^d\setminus N$, where $\widehat{\mu}(N)=0$.)
\end{defn}
If $\widetilde{\mathbb{M}}$ solves the martingale problem for $(L, C_0^{\infty}(\mathbb{R}^d))$ with respect to $\widehat{\mu}$ and
$(L, C_0^{\infty}(\mathbb{R}^d))$ is $L^1(\mathbb{R}^d, \widehat{\mu})$-unique, then exactly as in the proof of \cite[Proposition 3.57]{LST22} it follows that
\begin{equation} \label{eq:3.40*}
S_t f =\overline{T}_tf \; \text{ in $L^1(\R^d, \widehat{\mu})$,\quad  for all $f\in L^1(\R^d,\widehat{\mu})$, $t\ge 0$},
\end{equation}
and in particular, in this case, $\widehat{\mu}$ is an invariant measure for $\widetilde{\M}$ and $(\overline{T}_t)_{t>0}$.
\begin{proposition}\label{eq MP and weak sol}
Suppose that {\bf (A)} holds. Let $\widetilde{\mathbb{M}}$ be a right process with sub-invariant measure $\widehat{\mu}$ and continuous paths with respect to $\widehat{\mu}$. 
Then $\widetilde{\mathbb{M}}$ solves \eqref{weaksolution} with respect to $\widehat{\mu}$, if and only if 
$\widetilde{\mathbb{M}}$ solves the martingale problem for $(L, C_0^{\infty}(\mathbb{R}^d))$ with respect to $\widehat{\mu}$.
\end{proposition}
\begin{proof}Fix arbitrary $v \in L^1(\R^d,\widehat{\mu})_b$, $v\ge 0$, such that $\int_{\mathbb{R}^d} v d\widehat{\mu} = 1$. If $\widetilde{\mathbb{M}}$ solves the martingale problem for $(L, C_0^{\infty}(\mathbb{R}^d))$ with respect to $\widehat{\mu}$, it follows exactly as in Theorem \ref{prop:3.1.6}(ii) that for $u\in C_0^{\infty}(\R^d)$, the quadratic variation process $\langle \widetilde{M}^u \rangle$ of $\widetilde{M}^u$ satisfies 
\begin{equation}\label{quadvarvmuhat}
\langle \widetilde{M}^u \rangle_t=\int_0^t \langle \widehat{A}\nabla u, \nabla u\rangle(\widetilde{X}_s)ds, \quad t\ge 0, \quad\widetilde{\mathbb{P}}_{v \widehat{\mu}}\text{-a.s.}
\end{equation}
for any $v \in L^1(\R^d,\widehat{\mu})_b$, $v\ge 0$, such that $\int_{\mathbb{R}^d} v d\widehat{\mu} = 1$. Therefore, by \eqref{eq:subinv 1} $(\widetilde{M}_t ^u)_{t\ge 0}$ is $\widetilde{\mathbb{P}}_{v \widehat{\mu}}$-square integrable and we  can obtain that $\widetilde{\mathbb{M}}$ solves \eqref{weaksolution} with respect to $\widehat{\mu}$ exactly as in Theorem \ref{prop:3.1.6}(iii) (where $\P_x$  is replaced by $\widetilde{\mathbb{P}}_{v \widehat{\mu}}$) since $\widetilde{\mathbb{M}}$ has continuous paths with respect to $\widehat{\mu}$. \\
Conversely, let $\widetilde{\mathbb{M}}$ solve \eqref{weaksolution} with respect to $\widehat{\mu}$ and fix an arbitrary $v \in L^1(\R^d,\widehat{\mu})_b$, $v\ge 0$, such that $\int_{\mathbb{R}^d} v d\widehat{\mu} = 1$ and $u\in C_0^{\infty}(\R^d)$. Then by classical stochastic calculus with respect to $\widetilde{\mathbb{P}}_{v \widehat{\mu}}$, $(\widetilde{M}_t ^u)_{t\ge 0}$ is a continuous local $\widetilde{\mathbb{P}}_{v \widehat{\mu}}$-martingale up to infinity, starting from zero, with quadratic variation as in \eqref{quadvarvmuhat}. Since the quadratic variation in \eqref{quadvarvmuhat} is $\widetilde{\mathbb{P}}_{v \widehat{\mu}}$-integrable by \eqref{eq:subinv 1}, we obtain that $(\widetilde{M}_t ^u)_{t\ge 0}$ is a $\widetilde{\mathbb{P}}_{v \widehat{\mu}}$-square integrable (hence integrable) martingale and the assertion follows.
\end{proof}\\
The following is our main theorem about conditional uniqueness in law:
\begin{theorem}\label{uniqueness in  law}
(i) Suppose that {\bf (A)} holds and that
$(L, C_0^{\infty}(\mathbb{R}^d))$ is $L^1(\mathbb{R}^d, \widehat{\mu})$-unique (see the definition at the beginning of Section \ref{l1closeduni}).  Let $\widetilde{\mathbb{M}}$ be a right process with sub-invariant measure $\widehat{\mu}$ and continuous paths with respect to $\widehat{\mu}$ and suppose that $\widetilde{\mathbb{M}}$ solves \eqref{weaksolution} with respect to $\widehat{\mu}$. Then $T_t1=1$ for all $t> 0$, and $\widehat{\mu}$ is an invariant measure for $\widetilde{\M}$ and $\M$ as in Proposition \ref{prop:3.1.6}, and $\M$ solves \eqref{weaksolution} with respect to $\widehat{\mu}$ (cf. Theorem \ref{prop:3.1.6}(iii)). Moreover, for any $v \in L^1(\R^d,\widehat{\mu})_b$, $v\ge 0$, such that $\int_{\mathbb{R}^d} v d\widehat{\mu} = 1$, the laws of $\widetilde{\M}$ and $\M$ with initial distribution $v\widehat{\mu}$ coincide, i.e. 
\[
\widetilde {\P}_{v\widehat{\mu}} \circ \widetilde{X}^{-1}=\P_{v\widehat{\mu}} \circ X^{-1}  \quad \text{on } \  \mathcal{B}(C([0, \infty), \R^d)).
\] 
In particular, if $\widehat{\mu}$ is a probability measure, then (choosing $v\equiv 1$) $\widetilde {\P}_{\widehat{\mu}}$ and $\P_{\widehat{\mu}}$ are additionally to the above mentioned statements stationary measures for $\widetilde{\M}$ and $\M$ respectively.\\
(ii) Under the assumptions of (i) there exists a set $N\in \mathcal{B}(\R^d)$ with $\widehat{\mu}(N)=0$, such that
\[
\widetilde {\P}_{x} \circ \widetilde{X}^{-1}=\P_{x} \circ X^{-1}  \quad \text{on } \  \mathcal{B}(C([0, \infty), \R^d))
\] 
for all $x\in \R^d\setminus N$.
\end{theorem}
\begin{proof}(i) The given statements are all an easy consequence of our previous explanations. We only briefly explain the statement concerning the laws of $\widetilde{\M}$ and $\M$. For this fix $v \in L^1(\R^d,\widehat{\mu})_b$, $v\ge 0$, such that $\int_{\mathbb{R}^d} v d\widehat{\mu} = 1$.
The one dimensional marginals of $\widetilde{\M}$ and $\M$ with initial distribution $v\widehat{\mu}:=v d\widehat{\mu}$ coincide by  \eqref{eq:3.40*}. Indeed, using \eqref{eq:3.40*} we get for any $t\ge 0$, $f\in \mathcal{B}_b(\R^d)$
\begin{equation*} \label{eq:subinv 1bis}
\widetilde{\mathbb{E}}_{v\widehat{\mu}}\big [ f(\widetilde{X}_t)\big ]=\int_{\R^d}S_t f(x) v(x)\widehat{\mu}(dx)=\int_{\R^d}T_tf (x)v(x)\widehat{\mu}(dx)=\mathbb{E}_{v\widehat{\mu}}\big [ f(X_t)\big ].
\end{equation*}
Using 
the latter, the Markov property and induction, we then see that the finite dimensional marginals of $\widetilde{X}$ and $X$ with initial distribution $v  \widehat{\mu}$ coincide, i.e. for any $0 \le t_1 < \ldots < t_n, \ t_i \in [0,\infty)$, $f_i \in \mathcal{B}_b(\R^d)$, $1\le i\le n, \ n \ge 1$, 
\begin{eqnarray}\label{finite dim distr} 
\widetilde{\mathbb{E}}_{v\widehat{\mu}}\big [ f_1(\widetilde{X}_{t_1})\cdot \ldots \cdot f_n(\widetilde{X}_{t_n})\big ]  	
&=&\int_{\mathbb{R}^d} S_{t_1}	(f_1 S_{t_2-t_1}(f_2 \cdots S_{t_{n-1}-t_{n-2}}(f_{n-1} S_{t_n-t_{n-1}} f_n))) v d\widehat{\mu}	\nonumber \\	
&=&\int_{\mathbb{R}^d} T_{t_1}	(f_1 T_{t_2-t_1}(f_2 \cdots T_{t_{n-1}-t_{n-2}}(f_{n-1} T_{t_n-t_{n-1}} f_n))) v d\widehat{\mu}			\nonumber	\\
&=&\mathbb{E}_{v\widehat{\mu}}\big [ f_1(X_{t_1})\cdot \ldots \cdot f_n(X_{t_n})\big ].
\end{eqnarray}
(ii) Let $\mathcal{D}$ be any countable dense subset of $C_0(\R^d)$ with respect to the uniform norm. By the measurability of $x\mapsto \P_{x}(A)$ and  $x\mapsto \widetilde{\P}_{x}(\widetilde{A})$ for $A\in \sigma(X_t\,|\, t\ge 0)$ and $\widetilde{A}\in \sigma(\widetilde{X}_t\,|\, t\ge 0)$, and \eqref{finite dim distr}, there exists $N\in \mathcal{B}(\R^d)$, such that 
\begin{equation}\label{eq1aex}
\widetilde{\mathbb{E}}_{x}\big [ f_1(\widetilde{X}_{t_1})\cdot \ldots \cdot f_n(\widetilde{X}_{t_n})\big ]=
\mathbb{E}_{x}\big [ f_1(X_{t_1})\cdot \ldots \cdot f_n(X_{t_n})\big ]
\end{equation}
for any $0 \le t_1 < \ldots < t_n, \ t_i \in \mathbb{Q}\cap [0,\infty)$, $f_i \in \mathcal{D}$, $1\le i\le n, \ n \ge 1$, and
\begin{equation}\label{eq2aex}
\P_x(t\mapsto X_t \text{ is continuous from }[0,\infty)\text{ to }\R^d)=1=\widetilde{\P}_x(t\mapsto \widetilde{X}_t \text{ is continuous from }[0,\infty)\text{ to }\R^d)
\end{equation}
for any $x\in \R^d\setminus N$. By simple approximation, using \eqref{eq2aex}, \eqref{eq1aex} extends to arbitrary $t_i \in [0,\infty)$ and $f_i\in C_0(\R^d)$ (and then $f_i\in C_b(\R^d)$), $1\le i\le n$. Now the assertion follows applying a monotone class theorem to the multiplicative class $C_b(\R^d)$. 
\end{proof}\\
The following corollary is an immediate consequence of Theorems \ref{prop:3.1.6} and \ref{uniqueness in  law}.
\begin{cor}\label{wellposedness}
Suppose that {\bf (A)} holds and consider $\widehat{\mu}$ as defined there. Suppose further that
$(L, C_0^{\infty}(\mathbb{R}^d))$ is $L^1(\mathbb{R}^d, \widehat{\mu})$-unique (see the definition at the beginning of Section \ref{l1closeduni}), and that $T_t1=1$ for all $t> 0$. Then there exists a right process $\mathbb{M}$ that solves \eqref{weaksolution} with respect to $\widehat{\mu}$ and uniqueness in law holds with respect to $\widehat{\mu}$ among all right processes $\widetilde{\mathbb{M}}$ with sub-invariant measure $\widehat{\mu}$ and continuous paths with respect to $\widehat{\mu}$.
\end{cor}
\noindent
The following theorem provides our main explicit result on conditional uniqueness in law of \eqref{weaksolution}.
\begin{theorem}\label{weakexfairlygeneral 2}
Assume {\bf (A)} and that the functions $a_{ij}$, $1\le i,j\le d$, are locally H\"{o}lder continuous on $\mathbb{R}^d$ and suppose further that one of the following additional assumptions  (a) or (b) holds:
\begin{itemize}
\item[(a)] There exists a constant $M>0$ and $N_0 \in \mathbb{N}$, such that for a.e. $x \in \mathbb{R}^d \setminus \overline{B}_{N_0}$
\begin{equation}\label{growthconser 2}
-\frac{\langle \widehat{A}(x)x, x\rangle }{\|x\|^2} + \frac12 {\rm trace}(\widehat{A}(x))  + \langle\beta^{\rho, A,\psi}(x), x \rangle  +\langle\mathbf{B}(x), x \rangle  \leq M \|x\|^2 \left( \ln \|x\| +1   \right)
\end{equation}
and 
\begin{equation}\label{growthconser 3}
-\frac{\langle \widehat{A}(x)x, x\rangle }{\|x\|^2} + \frac12 {\rm trace}(\widehat{A}(x))  + \langle\beta^{\rho, A,\psi}(x), x \rangle  -\langle\mathbf{B}(x), x \rangle  \leq M \|x\|^2 \left( \ln \|x\| +1   \right).
\end{equation}
\item[(b)] $\widehat{\mu}$ is finite and either \eqref{growthconser 2} or \eqref{growthconser 3} holds.
\end{itemize}
Then $(L, C_0^{\infty}(\mathbb{R}^d))$ is $L^1(\mathbb{R}^d, \widehat{\mu})$-unique and $T_t1=1$, $t> 0$, where $(T_t)_{t>0}$ is defined in the paragraph following Theorem \ref{pojjjde}. In particular, uniqueness in law holds for \eqref{weaksolution} in the sense of Corollary \ref{wellposedness} (cf. also Definition  \ref{defnofmartin}(iv)).
\end{theorem}
\begin{proof}
Using Theorem \ref{theoconserva}(i) and (ii), the proof of the conservativeness of $(T_t)_{t>0}$ and $(T'_t)_{t>0}$ in case of (a) is analogous to the proof of Proposition \ref{weakexfairlygeneral}(ii). If $\widehat{\mu}$ is finite, it is enough to show that one semigroup is conservative in order to conclude that both, $(T_t)_{t>0}$ and $(T'_t)_{t>0}$, are conservative (cf. \cite[Remark 2.13]{LST22}). But the condition \eqref{growthconser 2} means that $(T_t)_{t>0}$ is conservative and the condition \eqref{growthconser 3} means that $(T'_t)_{t>0}$ is conservative. 
Thus the conservativeness of both, $ (T_t)_{t>0}$ and $(T'_t)_{t>0}$, also follows in the case of (b).
Now,  the $L^1(\mathbb{R}^d, \widehat{\mu})$-uniqueness of $(L, C_0^{\infty}(\mathbb{R}^d))$ follows from Theorem \ref{theoconserva}(i), Theorem \ref{analycharin} and Corollary \ref{cor4.4}.
\end{proof}
\centerline{}
\noindent
The following example illustrates among others Corollary \ref{wellposedness} and Theorem \ref{weakexfairlygeneral 2}, in particular, we consider a finite $\widehat{\mu}$.

\begin{example}\label{example section 5}
(i) Let $A=id$, $\bold{B}=0$, $\rho(x)=e^{-\|x\|^2}$, $\psi \in L^1_{loc}(\mathbb{R}^d)$ with $\psi >0$ a.e. and $\frac{1}{\psi} \in L^{\infty}_{loc}(\mathbb{R}^d)$, and $\rho \psi \in L^1(\mathbb{R}^d)$, i.e. $\widehat{\mu}$ is a finite measure. Then {\bf (A)} holds and $\beta^{\rho, A, \psi} =-x\frac{1}{\psi}(x)$. Thus \eqref{growthconser 2} writes as
\begin{equation} \label{growthcondi}
\big ( \frac{d}{2}-1 \big ) -  \|x\|^2 \leq M \psi(x) \|x\|^2 \left(  \ln \|x\| +1 \right) \quad \text{ for a.e.  $x\in \mathbb{R}^d \setminus \overline{B}_{N_0}$}.
\end{equation}
But for sufficiently large $N_0$, \eqref{growthcondi} always holds. Therefore by Theorem \ref{weakexfairlygeneral 2} $(T_t)_{t>0}$ is conservative and $(L, C_0^{\infty}(\mathbb{R}^d))$ is $L^1(\mathbb{R}^d, \widehat{\mu})$-unique.\\ 
(ii) Let $A,  \mathbf{B}$ and $\widehat{\mu}$ be as in {\bf (A)}. Assume that there exist constants $N_0 \in \mathbb{N}$ and $c, M>0$, such that
$$
\widehat{\mu}(B_{4n} \setminus B_{2n}) \leq (4n)^c \quad \text{ for all $n \geq 1$}
$$
and that
\begin{equation} \label{growtdivfr}
\frac{1}{\psi}(x) \frac{\langle A(x)x, x \rangle}{\|x\|^2} +  \left| \langle \mathbf{B}(x), x  \rangle \right | \leq M \|x\|^2 \ln (\|x\|  +1) \quad \text{ for a.e. $x \in \mathbb{R}^d \setminus \overline{B}_{N_0}$}.
\end{equation}
Then, by a direct consequence of \cite[Corollary 10(i)]{GT17} (see also \cite[Proposition 3.31]{LST22}), $(T_t)_{t>0}$ and $(T'_t)_{t>0}$  are conservative and 
therefore using Theorem \ref{analycharin} and Corollary \ref{cor4.4}, we obtain that $(L, C_0^{\infty}(\mathbb{R}^d))$ is $L^1(\mathbb{R}^d, \widehat{\mu})$-unique and $\widehat{\mu}$ is an invariant measure (but possibly infinite).  A typical example for a divergence-free vector field $\bold{B}$ (with respect to $\widehat{\mu}$) that satisfies {\bf (A)} is
$$
\beta^{\rho, C^T, \psi} =  \frac{1}{2\psi}{\rm div}\,C +\frac{1}{2\psi\rho} C^T \nabla \rho,
$$
where $C=(c_{ij})_{1 \leq i,j \leq d}$ is an anti-symmetric matrix, i.e. $c_{ij}=-c_{ji}$,  of locally bounded functions with $\beta^{\rho, C, \psi} \in L^2_{loc}(\mathbb{R}^d, \mathbb{R}^d, \widehat{\mu})$, $\psi\beta^{\rho, C, \psi} \in L^2_{loc}(\mathbb{R}^d, \mathbb{R}^d, \mu)$ (${\rm div}\,C\in L^1_{loc}(\mathbb{R}^d, \mathbb{R}^d)$). Indeed, for each $f \in C_0^{\infty}(\mathbb{R}^d)$
$$
\int_{\mathbb{R}^d} \langle \beta^{\rho, C, \psi}, \nabla f \rangle d\widehat{\mu} = -\frac12\sum_{i,j=1}^d \int_{\mathbb{R}^d}   \rho c_{ij} \partial_i \partial_j f dx -\frac12\int_{\mathbb{R}^d}\langle  C \nabla f,  \nabla \rho \rangle dx +\frac12\int_{\mathbb{R}^d}\langle  C^T \nabla \rho,  \nabla f \rangle dx= 0.
$$

\item[(iii)] Here we present an explicit  example satisfying the conditions in (ii) above such that additionally $\widehat{\mu}$ is finite: we let $A=id$, $\rho(x)=e^{-\|x\|^2}$, $\psi(x)=\frac{1}{\|x\|^{\alpha}} \phi(x)$, $x \in \mathbb{R}^d$ with $\alpha \in [0, d)$ and $\phi \in L^{\infty}(\mathbb{R}^d)$ with $\frac{1}{\phi} \in L^{\infty}_{loc}(\mathbb{R}^d)$ and $\phi(x)>0$ for a.e. $x \in \mathbb{R}^d$. In particular the growth condition on $\widehat{\mu}$ of (ii) is satisfied. Let $\beta \in \mathbb{R}$ with $2(\beta+1)<d$ and
$v(x):= 2+\cos\left( \frac{1}{\|x\|^{\beta}}  \right)$, $x \in \mathbb{R}^d \setminus \{0\}$, $v(0):=2$. Then, $v \in H^{1,2}_{loc}(\mathbb{R}^d)$ with 
$$
\nabla v (x) =  \beta \sin \Big ( \frac{1}{\|x\|^{\beta}} \Big ) \frac{1}{\|x\|^{\beta+1}} \frac{1}{\|x\|} x, \quad \text{for a.e. $x \in \mathbb{R}^d$}.
$$
Now define an anti-symmetric matrix of functions $C=(c_{ij})_{1 \leq i,j \leq d}$ where $c_{1d}=v$, $c_{d1}=-v$ and $c_{ij}=0$ for all $(i,j) \neq (1,d)$, $(d,1)$. Define
$$
\bold{B} := \beta^{\rho, C, \psi} = \frac{1}{2\psi} {\rm div} C +\frac{1}{2\psi\rho} C^T \nabla  \rho.
$$
Observe that
$$
\frac12 {\rm div} C  = \frac12 (-\partial_d v, 0, \ldots, 0, \partial_1 v)  = \frac{\beta}{2} \sin  \big(  \frac{1}{\|x\|^{\beta}}   \big) \frac{1}{\|x\|^{\beta+2}}  \big( -x_d, 0, \ldots, 0, x_1  \big)
$$
and
$$
\frac{1}{2\psi\rho } C^T \nabla  \rho =-\frac{\|x\|^{\alpha}}{\phi(x)} C^T x = \frac{\|x\|^{\alpha}}{\phi(x)} v \cdot (-x_d, 0, \ldots, 0, x_1).
$$
Then, we can check that $\langle \bold{B}(x), x \rangle =0$ for a.e. $x \in \mathbb{R}^d$. Finally, assume that there exist constants $N_0 \in \mathbb{N}$ and $M>0$, such that
\begin{equation*} \label{gtgrowthcon}
 \|x\|^{\alpha} \frac{1}{\phi(x)}  \leq M \|x\|^2 \ln (\|x\|  +1) \quad \text{ for a.e. $x \in \mathbb{R}^d \setminus \overline{B}_{N_0}$}.
\end{equation*}
Then, $\alpha \leq 2$ and the growth condition \eqref{growtdivfr} is satisfied, so that now all conditions of (ii) are met. Thus $(T_t)_{t>0}$ is conservative and $(L, C_0^{\infty}(\mathbb{R}^d))$ is $L^1(\mathbb{R}^d, \widehat{\mu})$-unique. 
The diffusion coefficient of $L$, $\widehat{A} = \frac{1}{\psi} A = \|x\|^{\alpha} \frac{1}{\phi}(x) id$, $\alpha \in [0, d)$ can be discontinuous and additionally degenerate at $0$. The drift coefficient of $L$ is expressed as
\begin{align*}
\mathbf{G} &=\beta^{\rho, A,\psi} + \mathbf{B} = \frac{1}{2\rho \psi} \nabla \rho+ \frac{1}{2\psi}{\rm div}\,C +\frac{1}{2\rho \psi} C^T \nabla \rho \\
&= \frac{\|x\|^{\alpha}}{\phi(x)} \Big (-x +\Big( \frac{\beta}{2} \sin \big(  \frac{1}{\|x\|^{\beta}}   \big) \frac{1}{\|x\|^{\beta+2}} +  \Big(  2+\cos \big( \frac{1}{\|x\|^{\beta}}   \big)  \Big) \Big) \cdot (-x_d, 0, \ldots, 0, x_1)\Big )
\end{align*}
where the $L^2$-singularities of $\mathbf{G}$ arise at $x=0$  in the in the term related to $\frac{\beta}{2} \sin  \big(  \frac{1}{\|x\|^{\beta}}   \big) \frac{1}{\|x\|^{\beta+2}}$ if $\beta \neq 0$ and $\alpha-\beta-1<0$.
\end{example}

\section{Auxiliary Results}\label{aux}
\begin{lemma}  \label{proptrunc} \rm
Let $B$ be an open ball in $\mathbb{R}^d$ and $f \in H^{1,1}(B)$. Let $f_n = (f \wedge n) \vee (-n)$. Then, 
$$
\nabla f_n= 
\left\{\begin{matrix}
0, \quad &\text{$dx$-a.e.} \; \{ f   \geq n \} \\
\nabla f, \quad &\qquad \text{$dx$-a.e.} \; \{ -n <f<n \} \\
0  \quad  &\;\text{$dx$-a.e.} \; \{ f   \leq  -n \}.
\end{matrix}\right.
$$
Thus, $\lim_{n \rightarrow \infty} f_n(x) = f(x)$ and
$\lim_{n \rightarrow \infty}\nabla f_n(x) = \nabla f(x)$ for a.e. $x \in B$. Moreover,
$$
\lim_{n \rightarrow \infty} f_n = f \quad \text{ in $H^{1,1}(B)$.}
$$
\end{lemma}
\noindent
\begin{proof}
Let $g_n = f \wedge n$. Then, $g_n =n-(f-n)^-$, hence it follows from \cite[Theorem 4.4(iii)]{EG15} that
$g_n \in H^{1,1}(B)$ and
$$
\nabla g_n= 
\left\{\begin{matrix}
0, \quad &\text{$dx$-a.e.} \; \{ f  \geq n \} \\
\nabla f  \quad  &\;\text{$dx$-a.e.} \; \{ f  < n \}.
\end{matrix}\right.
$$
Likewise, by  \cite[Theorem 4.4(iii)]{EG15} $f_n = g_n \vee (-n) = (g_n+n)^+ -n
\in H^{1,1}(B)$ and
$$
\nabla f_n= 
\left\{\begin{matrix}
\nabla g_n, \quad &\text{$dx$-a.e.} \; \{ g_n  >-n \} \\
0  \quad  &\;\text{$dx$-a.e.} \; \{ g_n   \leq  -n \}.
\end{matrix}\right.
$$
Note that $\{ g_n >-n \} = \{ f>-n \}$ and $\{ g_n \leq -n \} = \{ f \leq -n \}$, hence
$$
\nabla f_n= 
\left\{\begin{matrix}
0, \quad &\text{$dx$-a.e.} \; \{ f   \geq n \} \\
\nabla f, \quad &\qquad \text{$dx$-a.e.} \; \{ -n <f<n \} \\
0  \quad  &\;\text{$dx$-a.e.} \; \{ f   \leq  -n \}.
\end{matrix}\right.
$$
Thus, $\lim_{n \rightarrow \infty} f_n(x) = f(x)$ and
$\lim_{n \rightarrow \infty}\nabla f_n(x) = \nabla f(x)$ for a.e. $x \in B$, and hence by Lebesgue's theorem
$$
\lim_{n \rightarrow \infty} f_n = f \quad \text{ in $H^{1,1}(B)$.}
$$
\end{proof}

\begin{lemma} \label{basicprop}
Assume {\bf (A)}. Then,
\begin{eqnarray*}
\mathcal{E}^0(f,g) =   - \int_{\R^d} \big( \frac12 {\rm trace} (\widehat{A}\nabla^2f ) + \langle \beta^{\rho, A, \psi}, \nabla f \rangle \big) g \,d \widehat{\mu}, \quad \text{  for all }  f, g \in C_0^{\infty}(\mathbb{R}^d).
\end{eqnarray*}
\end{lemma}
\noindent
\begin{proof}
Let $f,g \in C_0^{\infty}(\mathbb{R}^d)$.
For each $k \geq1$, let $A_k=(a_{ij, k})_{1 \leq i,j \leq d}$ be the matrix of smooth functions defined by 
$a_{ij,k} = a_{ij} *\zeta_k$ where $\zeta_k(x) = k^d \zeta(kx)$ and  $\zeta$ is a standard mollifier on $\mathbb{R}^d$. For each $n, m \geq 1$, let
$\rho_n:= (\rho \wedge n) \vee (-n)$ and $\rho_{n,m}= \rho_n * \zeta_m$. Then, using integration by parts, Lemma \ref{proptrunc} and Lebesgue's theorem, we obtain that
\begin{eqnarray*}
\mathcal{E}^0(f,g) &=& \lim_{n \rightarrow \infty} \frac12 \int_{\R^d} \langle \rho_n  A \nabla f, \nabla g \rangle dx = \lim_{n \rightarrow \infty} \lim_{m \rightarrow \infty} \frac12 \int_{\R^d} \langle \rho_{n,m}  A \nabla f, \nabla g \rangle dx \\
&=&  \lim_{n \rightarrow \infty} \lim_{m \rightarrow \infty}\lim_{k \rightarrow \infty} \frac12 \int_{\R^d} \langle \rho_{n,m}  A_k \nabla f, \nabla g \rangle dx              \\
&=& \lim_{n \rightarrow \infty} \lim_{m \rightarrow \infty} \lim_{k \rightarrow \infty} -\frac12 \int_{\R^d} \left( \rho_{n,m} \, \text{trace} (A_k\nabla^2f ) + \langle \rho_{n,m} \nabla A^T_k +  A^T_k \nabla \rho_{n,m}, \nabla f \rangle \right) g \,dx \\
&=&  \lim_{n \rightarrow \infty} \lim_{m \rightarrow \infty}   -\frac12 \int_{\R^d} \left( \rho_{n,m} \, \text{trace} (A \nabla^2f ) + \langle \rho_{n,m} \cdot \text{div} A +  A^T\nabla \rho_{n,m}, \nabla f \rangle \right) g \,dx \\
&=& \lim_{n \rightarrow \infty}   -\frac12 \int_{\R^d} \left( \rho_{n} \, \text{trace} (A \nabla^2f ) + \langle \rho_{n} \cdot \text{div} A +  A^T\nabla \rho_n, \nabla f \rangle \right) g \,dx \\
&=&   - \int_{\R^d} \big( \frac12 \text{trace} (\widehat{A}\nabla^2f ) + \langle \beta^{\rho, A, \psi}, \nabla f \rangle \big) g \,d \widehat{\mu}.
\end{eqnarray*}

\end{proof}
\centerline{}
The standard arguments for the derivation of the following proposition are well-known (see e.g. \cite{Da85, Eberle}), but we provide the proof here for the reader's convenience.
\begin{proposition} \label{extlempt}
Assume {\bf (A)} and consider the framework in Section \ref{framework}.
 Then, the following hold:
\begin{itemize}
\item[(i)]
$(T^0_t)_{t>0}$ restricted to $L^1(\R^d, \widehat{\mu})_b$ can be extended to a sub-Markovian $C_0$-semigroup of contractions $(\overline{T}_t^0)_{t>0}$ with generator $(\overline{L}^0, D(\overline{L}^0) )$  on $L^1(\R^d, \widehat{\mu})$. If $f \in D(L^0)$ and $f, L^0 f \in L^1(\R^d, \widehat{\mu})$, then $f \in D(\overline{L}^0)$ and $\overline{L}^0 f = L^0 f$. Set $\mathcal{A}:= \{ u \in D(L^0) \cap L^1(\R^d, \widehat{\mu}) : L^0 u \in L^1(\R^d, \widehat{\mu}) \}$. Then $(\overline{L}^0, D(\overline{L}^0))$ is  the closure of $(L^0, \mathcal{A})$ on $L^1(\R^d, \widehat{\mu})$. 
\item[(ii)]
For a bounded open subset $V$ of $\R^d$, $(T^{0,V}_t)_{t>0}$ restricted to $L^1(V, \widehat{\mu})_b$ can be extended to a sub-Markovian $C_0$-semigroup of contractions $(\overline{T}^{0,V}_t)_{t>0}$  on $L^1(V, \widehat{\mu})$. Also if $f \in D(L^{0,V})$ and $f, \, L^{0,V} f \in L^1(V, \widehat{\mu})$, then $f \in D(\overline{L}^{0,V})$ and $\overline{L}^{0,V} f = L^{0,V} f$. Finally, the closure of $(L^{0,V}, D(L^{0,V}))$ on $L^1(V, \widehat{\mu})$ is $(\overline{L}^{0,V}, D(\overline{L}^{0,V}))$.
\end{itemize}
\end{proposition}
\begin{proof}
(i) Since $(\mathcal{E}^0, D(\mathcal{E}^0))$ is a symmetric Dirichlet from, $(T^{0}_t)_{t>0}$ is symmetric on $L^2(\mathbb{R}^d, \widehat{\mu})$ and sub-Markovian. Let $f \in L^1(\R^d, \widehat{\mu})_b$ with $f \geq 0$. Using the sub-Markovian property of $(T^0_t)_{t>0}$
\begin{eqnarray*}
\int_{\R^d}T^0_t f d  \widehat{\mu}  &=& \lim_{n \rightarrow \infty} \int_{\R^d} T^0_t f \cdot 1_{B_n}d  \widehat{\mu} =\lim_{n \rightarrow \infty} \int_{\R^d} f \cdot T_t^0 1_{B_n} d \widehat{\mu} \leq \int_{\R^d} f d \widehat{\mu}.
\end{eqnarray*}
Thus, for each $f \in L^1(\mathbb{R}^d, \widehat{\mu})$ 
\begin{align*}
\int_{\R^d} |T^0_t f| d \widehat{\mu}  &= \int_{\mathbb{R}^d}|T^0_t f^+ - T^0_t f^-| d \widehat{\mu} \leq  \int_{\mathbb{R}^d} T^0_t f^+ d \widehat{\mu} +\int_{\mathbb{R}^d} T_t^0 f^- d \widehat{\mu} \\
& \leq \int_{\mathbb{R}^d} f^+ d \widehat{\mu} + \int_{\mathbb{R}^d} f^- d \widehat{\mu} = \int_{\mathbb{R}^d}|f| d\widehat{\mu}.
\end{align*}
Since $L^1(\R^d, \widehat{\mu})_b$ is dense in $L^1(\R^d, \widehat{\mu})$, $(T_t^0)_{t>0}$ restricted to $L^1(\R^d, \widehat{\mu})_b$ uniquely extends to a sub-Markovian contraction semigroup $(\overline{T}^0_t)_{t>0}$
on $L^1(\R^d, \widehat{\mu})$.
Define 
\begin{eqnarray*}
\mathcal{D}:= \{g \in  L^{\infty}(\R^d): g \geq 0 \text{ and there exists } A \in \mathcal{B}(\R^d) \text{ with  } \widehat{\mu}(A) < \infty  \text{ and } g = 0  \text{ on } \R^d \setminus A \}.
\end{eqnarray*}
Since $ \mathcal{D}$ is dense in $L^1(\R^d, \widehat{\mu})^+$,\, $\mathcal{D}-\mathcal{D}$ is dense in $L^1(\R^d, \widehat{\mu})$. Let $f \in \mathcal{D}-\mathcal{D}$. Then, there exists $A \in \mathcal{B}(\R^d)$ with $\widehat{\mu}(A)< \infty$, such that $\text{supp}(f) \subset A$ and $f \in L^1(\R^d, \widehat{\mu})_b$. By strong continuity of $(T^0_t)_{t>0}$ on $L^2(\R^d, \widehat{\mu})$
$$
\lim_{t \rightarrow 0+}\int_{\R^d} 1_A |T^0_t f| d  \widehat{\mu} =  \int_{\R^d} 1_A  |f|  d \widehat{\mu} = \|f\|_{L^1(\R^d, \widehat{\mu})},
$$
hence using the contraction property on $L^1(\R^d, \widehat{\mu})$,
\begin{eqnarray*}
0 &\leq& \int_{\R^d} 1_{\R^d \setminus A} \,|T^0_t f| d \widehat{\mu}  =  \int_{\R^d} |T^0_t f | d \widehat{\mu} - \int_{\R^d} 1_{A} |T^0_t f| d \widehat{\mu}  \\
&\leq& \|f\|_{L^1(\R^d, \widehat{\mu})} - \int_{\R^d} 1_A |T^0_t f|d \widehat{\mu} \longrightarrow 0 \;\; \text{ as } t \rightarrow 0+.
\end{eqnarray*}
Therefore
\begin{eqnarray*}
\lim_{t \rightarrow 0+} \int_{\R^d} |T^0_t f -f | d \widehat{\mu} &=& \lim_{t \rightarrow 0+} \big(\int_{\R^d} 1_A |T^0_t f - f| d \widehat{\mu} + \int_{\R^d} 1_{\R^d \setminus A} |T^0_t f | d \widehat{\mu} \big) \\
&\leq& \widehat{\mu}(A)^{1/2} \lim_{t \rightarrow 0+} \|T_t f - f \|_{L^2(\R^d, \widehat{\mu})} =0. 
\end{eqnarray*}
By the denseness of $\mathcal{D}-\mathcal{D}$ in $L^1(\R^d, \widehat{\mu})$, we get the strong continuity of $(\overline{T}_t^0)_{t>0}$ on $L^1(\R^d, \widehat{\mu})$.
Now let $f \in D(L^0)$ and $f,\, L^0 f \in L^1(\R^d, \widehat{\mu})$. Then $f \in L^{1}(\R^d, \widehat{\mu}) \cap L^2(\R^d, \widehat{\mu})$, $L^0 f \in  L^{1}(\R^d, \widehat{\mu}) \cap L^2(\R^d, \widehat{\mu})$, hence we get \,$\overline{T}_t^0 f = T^0_t f$, \,  $\overline{T}_t^0 L^0f = T^0_t L^0f$\; for every $t>0$. Using the Fundamental Theorem of Calculus for Bochner integrals and the strong continuity of $(\overline{T}_t^0)_{t>0}$ on $L^1(\R^d, \widehat{\mu})$
\begin{eqnarray*}
\frac{\overline{T}_t^0 f- f}{t} &=& \frac{T_t^0 f- f}{t}  = \frac{1}{t} \int_{0}^t T^0_s L^0 f \,ds  \\
&=& \frac{1}{t} \int_{0}^t \overline{T}_s^0 L^0 f \,ds   \longrightarrow L^0 f \quad \text{ in } L^1(\R^d, \widehat{\mu}) \;\,\text{ as } t \rightarrow 0+.
\end{eqnarray*}
Consequently, $f \in D(\overline{L}^0)$ and $\overline{L}^0 f = L^0 f$.  Let $(\overline{G}_{\alpha}^0)_{\alpha>0}$ be the resolvent generated by $(\overline{L}^0, D(\overline{L}^0))$.\;
Let $\mathcal{C}:=\big \{ \overline{G}_{1}^0 g :  g \in C_0^{\infty}(\R^d) \big \}$. Then $\mathcal{C} \subset \mathcal{A}$ and one can directly check that $\mathcal{C}  \subset D(\overline{L}^0)$ is dense with respect to the graph norm $\| \cdot \|_{D(\overline{L}^0)}$, hence it completes our proof.\\
(ii) It is analogous to the case of (i).
\end{proof}

\begin{lemma}\label{bddapxlem}
Consider the framework in Section \ref{framework}.  Let $V$ be a bounded open subset of $\R^d$ and $f \in \widehat{H}^{1,2}_0(V, \mu)_b$. Then there exists a sequence $(f_n)_{n \geq 1} \subset C_0^{\infty}(V)$ and a constant $M>0$ such that $\|f_n\|_{L^{\infty}(V)} \leq M$ for all $n \geq 1$ and
\begin{eqnarray*}
&&\lim_{n \rightarrow \infty} f_n = f \quad \text{ in } \; \widehat{H}^{1,2}_0(V, \mu), \;\qquad \lim_{ n \rightarrow \infty} f_n =f \quad \text{a.e. on } V.
\end{eqnarray*}
\end{lemma}
\begin{proof}
Take $(g_n)_{n \geq 1} \subset C_0^{\infty}(V)$ such that 
\begin{equation} \label{appchar}
\lim_{n \rightarrow \infty} g_n = f \; \text{ in } \widehat{H}^{1,2}_0(V, \mu)  \;\; \text{ and  }\; \lim_{n \rightarrow \infty} g_n  = f \;\; \;\text{ a.e. on  } V.
\end{equation}
Define $\varphi \in C^{\infty}_0(\R)$ such that $\varphi(t) = t$\, if\, $|t| \leq \|f\|_{L^{\infty}(\R^d)}+1$ and $\varphi(t)= 0$ \, if \, $|t|\geq \|f\|_{L^{\infty}(\R^d)}+2$. \; Let $M:= \|\varphi \|_{L^{\infty}(\R)}$ and $\widetilde{f}_n:= \varphi(g_n)$. Then $\widetilde{f}_n \in C_0^{\infty}(V)$ and \;$\| \widetilde{f}_n\|_{L^{\infty}(V)} \leq M$ for all $n \geq 1$.
By Lebesgue's Theorem and \eqref{appchar}, 
$$
\; \quad  \lim_{n \rightarrow \infty} \widetilde{f}_n =\lim_{n \rightarrow \infty} \varphi (g_n) = \varphi(f)  = f \quad \text{ in } L^2(V, \widehat{\mu}) \;\; \text{ and  }\; \;\; \lim_{n \rightarrow \infty} \widetilde{f}_n  = f \;\; \;\text{ a.e. on  } V.
$$
Using the chain rule and \eqref{appchar}
\begin{eqnarray*}
\sup_{n \geq 1} \| \nabla \widetilde{f}_n \|_{L^{2}(V, \widehat{\mu})} 
&\leq& \|\varphi'\|_{L^{\infty}(\R)}  \sup_{n \geq 1} \|\nabla g_n \|_{L^2(V, \widehat{\mu})} \ <\infty.\\
\end{eqnarray*}
Thus by the Banach-Alaoglu Theorem and the Banach-Saks Theorem, there exists a subsequence of $(\widetilde{f}_n)_{n \geq 1}$, say again $(\widetilde{f}_n)_{n \geq 1}$, such that for the Ces\`aro mean
$$
f_N:= \frac{1}{N} \sum_{n=1}^{N} \widetilde{f}_n \longrightarrow f \;\; \;\text{ in } \;  \widehat{H}^{1,2}_0(V, \mu) \quad \text{ as } N \rightarrow \infty.
$$
Note that $f_N \in C_0^{\infty}(V)$, $\|f_N\|_{L^{\infty}(V)} \leq M$ for all $N \in \N$. Thus, $(f_n)_{n \geq 1}$ is the desired sequence.
\end{proof}

\begin{lemma} \label{teclemma}
Assume {\bf (A)} and consider the framework in Section \ref{framework}. Let $V$ be a bounded open subset of $\mathbb{R}^d$ and let $w \in \widehat{H}^{1,2}_0(V, \mu)_b$. Then, $w^+ \in \widehat{H}^{1,2}_0(V, \mu)_b$. Moreover,
\begin{equation*} 
\int_{V} \langle \mathbf{B}, \nabla w \rangle w^+ d \widehat{\mu} = \int_{V} \langle \mathbf{B}, \nabla w^+ \rangle w^+ d \widehat{\mu}.
\end{equation*}
\end{lemma}

\noindent
\begin{proof}
By Lemma \ref{bddapxlem}, we can take a sequence $(f_n)_{n \geq 1} \subset C_0^{\infty}(V)$ and a constant $M_1>0$ such that 
$$
\sup_{n \geq 1} \|f_n\|_{L^{\infty}(V)} \leq M_1,   \; \;\; \lim_{n \rightarrow \infty} f_n = w \; \; \text{a.e.}  \;\; \text{ and } \; \lim_{n \rightarrow \infty} f_n =w \; \text{ in $H^{1,2}_0(V, \widehat{\mu})$}.
$$
Moreover, there exists a constant $M_2>0$ such that 
$$
\|f^+_n \|_{\widehat{H}^{1,2}_0(V, \mu)} \leq \|f_n \|_{\widehat{H}^{1,2}_0(V, \mu)} \leq M_2.
$$
By the Banach-Alaoglu Theorem, there exists a subsequence of $(f^{+}_n)_{n \geq 1}$, say again $(f^+_n)_{n \geq 1}$ such that
$$
\lim_{n \rightarrow \infty} f^+_n = w^+ \; \text{ weakly in  $\widehat{H}^{1,2}_0(V, \mu)$} \;\; \text{ and }\;\; \lim_{n \rightarrow \infty} f^+_n = w^+ \; \text{ a.e.}
$$
By \cite[Theorem 4.4 (iii)]{EG15},
\begin{equation} \label{appdivme}
\int_{V} \langle \mathbf{B}, \nabla f_n \rangle f_n^+ d  \widehat{\mu} = \int_{V} \langle \mathbf{B}, \nabla f_{n}^+ \rangle f^{+}_{n}d  \widehat{\mu}.
\end{equation}
For the convergence of the right hand side, observe that
\begin{eqnarray*}
&&\big | \int_{V} \langle \mathbf{B}, \nabla w^+ \rangle w^{+}d \widehat{\mu}  - \int_{V} \langle \mathbf{B}, \nabla f_{n}^+ \rangle f^{+}_{n}d \widehat{\mu} \big |  \\
&\leq&  \underbrace{ \big | \int_{V }  \langle \mathbf{B}, \nabla (w^{+}-f^+_n )\rangle w^{+} d \widehat{\mu} \big |}_{=:I_n} + \underbrace{\big | \int_{V} \langle \mathbf{B}, \nabla f_{n}^+ \rangle (w^{+}- f^{+}_{n}) d \widehat{\mu}   \big |}_{=:J_n}.
\end{eqnarray*}
Since $\lim_{n \rightarrow \infty} f^+_n = w^+ \, \text{ weakly in } \widehat{H}^{1,2}_0(V, \mu)$, we have $\lim_{n \rightarrow \infty} I_n = 0$. Using the Cauchy-Schwarz inequality, it holds that
\begin{eqnarray*}
J_n &\leq& \int_{V} \|\mathbf{B}\| \|\nabla f_n^+\|\, |w^+ - f^+_n| d\widehat{\mu} \\
&\leq&  \big( \int_{V} \|\nabla f_n^+\|^2\, |w^+ -f^+_n|\, \rho dx \big)^{1/2} \big(\int_{V}  \|\psi \mathbf{B}\|^2 |w^+ - f^+_n| \,\rho dx \big)^{1/2} \\
&\leq& (M+\|w \|_{L^{\infty}(V)})^{1/2} \sup_{n \geq 1} \|f_n^+\|_{\widehat{H}^{1,2}_{0}(V, \mu)}
 \big(\int_{V}  \|\psi \mathbf{ B}\|^2 |w^+ - f^+_n| d\mu \big)^{1/2} \\
&& \quad \longrightarrow 0 \;\; \text{ as } n \rightarrow \infty
\end{eqnarray*}
by Lebesgue's Theorem. The left hand side of \eqref{appdivme} can be treated similarly. We thus obtain
\begin{equation*} 
\int_{V} \langle \mathbf{B}, \nabla w \rangle w^+ d \widehat{\mu} = \int_{V} \langle \mathbf{B}, \nabla w^+ \rangle w^+ d \widehat{\mu}.
\end{equation*}
\end{proof}

\begin{lemma}\label{applempn}
Consider the framework in Section \ref{framework}.  Let $f \in \widehat{H}^{1,2}_0(\R^d, \mu)_{0}$ and $V$ be a bounded open subset of $\R^d$ with $\text{supp}(f) \subset V$. Then $f \in \widehat{H}^{1,2}_0(V, \mu)$ and there exists $(g_n)_{n \geq 1} \subset C_0^{\infty}(V)$ such that 
\begin{eqnarray*}
&&\lim_{n \rightarrow \infty} g_n = f \quad \text{ in } \; \widehat{H}^{1,2}_0(\R^d, \mu), \;\qquad \lim_{ n \rightarrow \infty} g_n =f \quad \text{a.e. on } \R^d.
\end{eqnarray*}
Moreover, if $f \in \widehat{H}^{1,2}_0(\R^d, \mu)_{0,b}$ with $\text{supp}(f) \subset V$, then there exists $(f_n)_{n \geq 1} \subset C_0^{\infty}(V)$ and a constant $M>0$ such that $\|f_n\|_{L^{\infty}(V)} \leq M$ for all $n \geq 1$ and
\begin{eqnarray*}
&&\lim_{n \rightarrow \infty} f_n = f \quad \text{ in } \; \widehat{H}^{1,2}_0(\R^d, \mu), \;\qquad \lim_{ n \rightarrow \infty} f_n =f \quad \text{a.e. on } \R^d.
\end{eqnarray*}
\end{lemma}
\begin{proof}
Let $W$ be an open subset of $\R^d$ satisfying $\text{supp}(f) \subset W \subset \overline{W} \subset V$.
Take a cut-off function $\chi \in C_0^{\infty}(\R^d)$ satisfying $\text{supp}(\chi) \subset V$ and $\chi \equiv 1$ on $W$. Since $f \in \widehat{H}^{1,2}_0(\R^d, \mu)$, there exists $\widetilde{g}_n \in C_0^{\infty}(\R^d)$ such that
$$
\lim_{n \rightarrow \infty} \widetilde{g}_n =f \;\; \text{ in } \widehat{H}^{1,2}_0(\R^d, \mu).
$$
Thus $\chi \widetilde{g}_n \in C_0^{\infty}(V)$ with  $\text{supp}(\chi g_n) \subset V$ and
\begin{align*}
\|\chi \widetilde{g}_n -  f  \|_{L^{2}(V, \widehat{\mu})} &=\|\chi \widetilde{g}_n -  f  \|_{L^{2}(\R^d, \widehat{\mu})} = \|\chi \widetilde{g}_n -  \chi f  \|_{L^{2}(\R^d, \widehat{\mu})} \\
&\leq \| \chi \|_{L^{\infty}(\R^d)} \|\widetilde{g}_n - f \|_{L^2(\R^d, \widehat{\mu})} \longrightarrow 0 \;\; \text{ as } n \rightarrow \infty.
\end{align*}
Note that $\chi \widetilde{g}_n \in C_0^{\infty}(V) \subset \widehat{H}^{1,2}_0(V,\mu)$ and
\begin{eqnarray*}
\sup_{n \geq 1} \|\nabla (\chi \widetilde{g}_n) \|_{L^2(V, \mu)} &=& \sup_{n \geq 1} \left(\| \widetilde{g}_n \nabla \chi  \|_{L^2(V, \mu)} + \|\chi \nabla \widetilde{g}_n  \|_{L^2(V, \mu)}\right) \\
&\leq& \sup_{n \geq 1} \Big(\frac{\| \nabla \chi \|_{L^{\infty}(V)}  }{\sqrt{\inf_V \psi}} \|\widetilde{g}_n\|_{L^2(\R^d, \widehat{\mu})} + \| \chi \|_{L^{\infty}(V)} \|\nabla \widetilde{g}_n\|_{L^2(\R^d, \mu)}\Big) \\
&<& \infty.
\end{eqnarray*}
By the Banach-Alaoglu Theorem and  the Banach-Saks Theorem, $f \in \widehat{H}^{1,2}_0(V, \mu)$ and there exists a subsequence of $(\chi \widetilde{g}_n)_{n \geq 1}$, say again $(\chi \widetilde{g}_n)_{n \geq 1}$, such that for the Ces\`aro mean
$$
g_N:= \frac{1}{N} \sum_{n=1}^{N} \chi \widetilde{g}_n \longrightarrow f \;\; \;\text{ in } \;  \widehat{H}^{1,2}_0(V, \mu) \quad \text{ as } N \rightarrow \infty.
$$
Take a subsequence of $(g_n)_{n \geq 1}$, say again $(g_n)_{n \geq 1}$, such that $\lim_{n \rightarrow \infty} g_n =f$\, a.e.
The last assertion holds by the same line of arguments  as in the proof of Lemma \ref{bddapxlem}.
\end{proof}

\begin{lemma} \label{tecpjp}
Consider the framework in Section \ref{framework}. Let $V_1$, $V_2$ be bounded open subsets of $\R^d$ satisfying $\overline{V}_1 \subset V_2$. Assume $f \in \widehat{H}^{1,2}_0(V_2, \mu)$, $g \in \widehat{H}_0^{1,2}(V_1, \mu)$ with $g =0$ on $V_2 \setminus V_1$. If $0 \leq f \leq g$, then $f \in \widehat{H}^{1,2}_0(V_1, \mu)$.
\end{lemma}
\begin{proof}
Take $(g_n)_{n \geq 1} \subset C_0^{\infty}(V_2)$ satisfying $\text{supp}(g_n) \subset V_1$ for all $n \in \N$ and 
$$
\lim_{n \rightarrow \infty} g_n = g \quad \text{ in } \widehat{H}^{1,2}_0(V_2, \mu).
$$
\;Observe that for all $n \in \N$
$$
\text{supp}(f \wedge g_n) \subset V_1 \;\text{ and }\; f \wedge g_n = \frac{f+g_n}{2}-\frac{|f-g_n|}{2}\in \widehat{H}^{1,2}_0(V_2, \mu).
$$
By Lemma \ref{applempn}, $f \wedge g_n \in \widehat{H}^{1,2}_0(V_1, \mu)$ for all $n \in \N$.
\;Moreover 
\begin{eqnarray*}
\lim_{n \rightarrow \infty} f \wedge g_n  = \lim_{n \rightarrow \infty} \big(\frac{f+g_n}{2}-\frac{|f-g_n|}{2} \big)=\frac{f+g}{2}-\frac{|f-g|}{2} = f \wedge g = f  \;\; \text{ in }  L^2(V_1, \widehat{\mu}).
\end{eqnarray*}
Since $\big(  \langle \cdot,  \cdot \rangle_{\widehat{H}^{1,2}_0(V_2, \mu)}, \widehat{H}^{1,2}_0(V_2, \mu)  \big)$ is a Dirichlet form,
\begin{eqnarray*}
&\, & \sup_{n \geq 1} \|f \wedge g_n \|_{\widehat{H}^{1,2}_0(V_1, \mu)} \\
&=&  \sup_{n \geq 1} \|f \wedge g_n \|_{\widehat{H}^{1,2}_0(V_2, \mu)}  \\
&=&  \sup_{n \geq 1} \big\| \frac{f+g_n}{2}-\frac{|f-g_n|}{2}   \big \|_{\widehat{H}^{1,2}_0(V_2, \mu)} \\
&\leq& \frac{1}{2} \sup_{n \geq 1} \big( \|f\|_{\widehat{H}^{1,2}_0(V_2, \mu)}  +\|g_n\|_{\widehat{H}^{1,2}_0(V_2, \mu)} +\big \| |f| \big\|_{\widehat{H}^{1,2}_0(V_2, \mu)} +\big \||g_n| \big \|_{\widehat{H}^{1,2}_0(V_2, \mu)}    \big)   \\
&\leq& \sup_{n \geq 1} \big( \|f\|_{\widehat{H}^{1,2}_0(V_2, \mu)}  +\|g_n\|_{\widehat{H}^{1,2}_0(V_2, \mu)}  \big) < \infty.
\end{eqnarray*}
Thus by the Banach-Alaoglu Theorem,  $f \in \widehat{H}^{1,2}_0(V_1, \mu)$. \vspace{-0.4em}
\end{proof}

\begin{lemma} \label{appenlem}
Assume {\bf (A)} and consider the framework in Section \ref{framework}. Let $\alpha>0$ and $h \in L_{loc}^{\infty}(\mathbb{R}^d)$ satisfy
$$
\int_{\mathbb{R}^d} (\alpha-L)u \cdot h\, d\widehat{\mu}=0 \quad \text{ for all $u \in D(L^0)_{0,b}$}.
$$
Then, $h \in \widehat{H}^{1,2}_{loc}(\mathbb{R}^d, \mu) \cap  L_{loc}^{\infty}(\mathbb{R}^d)$ and it holds 
\begin{align} \label{invinequ2-1}
\mathcal{E}_{\alpha}^0(u, h) - \int_{\mathbb{R}^d} \langle \bold{B}, \nabla u \rangle h \, d \widehat{\mu}=0 \quad \text{ for all $u \in \widehat{H}^{1,2}_0(\mathbb{R}^d, \mu)_0$}.
\end{align}
\end{lemma}
 \begin{proof}
 Let $\alpha>0$ and $h \in L_{loc}^{\infty}(\mathbb{R}^d)$ be as in the statement.
 Let $\chi \in C_0^{\infty}(\mathbb{R}^d)$  and $u \in D(L^0)_{b}$. Take an open ball $B$ so that $\text{supp}(\chi) \subset B$. Then, $\chi u \in D(L^0)_{0,b}$ with $\text{supp}(\chi u) \subset B$ and
$$
L^0 (\chi u)= \chi L^0 u + u L^0 \chi + \langle A \nabla \chi, \nabla u\rangle.
$$
Thus, for each $u \in D(L^0)_b$ it holds that
\begin{align}
&\int_{\mathbb{R}^d} (\alpha-L^0) u \cdot (\chi h) d \widehat{\mu} \nonumber \\
\quad &= \int_{\mathbb{R}^d} (\alpha-L^0) (u \chi ) \cdot h\, d\widehat{\mu} + \int_{\mathbb{R}^d} \langle \widehat{A} \nabla \chi, \nabla u \rangle h\,d\widehat{\mu} + \int_{\mathbb{R}^d}  u L^0 \chi \cdot h \, d\widehat{\mu} \nonumber  \\
&= \int_{\mathbb{R}^d} \langle \bold{B}, \nabla (u\chi) \rangle h \, d \widehat{\mu}+ \int_{\mathbb{R}^d} \langle \widehat{A} \nabla \chi, \nabla u \rangle h\,d\widehat{\mu} + \int_{\mathbb{R}^d}  u L^0 \chi \cdot h \, d\widehat{\mu}.  \label{equse}
\end{align}
Define a map $\mathcal{T}:D(\mathcal{E}^0) \rightarrow \mathbb{R}$ by 
$$
\mathcal{T} u = \int_{\mathbb{R}^d} \langle \bold{B}, \nabla (u\chi) \rangle h \, d \widehat{\mu}+ \int_{\mathbb{R}^d} \langle \widehat{A} \nabla \chi, \nabla u \rangle h\,d\widehat{\mu} + \int_{\mathbb{R}^d}  u L^0 \chi \cdot h \, d\widehat{\mu}, \quad u \in D(\mathcal{E}^0).
$$
Then, $\mathcal{T}$ is continuous with respect to the norm $\| \cdot \|_{D(\mathcal{E}^0)}$ since
\begin{align*}
| \mathcal{T} u | &\leq  \| \psi \bold{B} \|_{L^2(\mathbb{R}^d, \mu)}  \|\chi \|_{H^{1,\infty}(B)}\|h\|_{L^{\infty}(B)} \cdot \Big(\lambda_{B}^{-\frac12} +\alpha^{-1} \left\| \frac{1}{\psi} \right\|_{L^{\infty}(B)} \Big)  \mathcal{E}^0_{\alpha}(u,u)^{\frac12}  \\
&\quad +  \|h\|_{L^{\infty}(\mathbb{R}^d)} \mathcal{E}^0(\chi, \chi)^{\frac12} \mathcal{E}^0(u,u)^{\frac12} + \|h\|_{L^{\infty}(\mathbb{R}^d)} \|L^0 \chi \|_{L^2(\mathbb{R}^d, \widehat{\mu})} \|u\|_{L^2(\mathbb{R}^d, \widehat{\mu})}.
\end{align*}
By the Riesz representation theorem, there exists $v \in D(\mathcal{E}^0)$ such that
\begin{equation}\label{rieszre}
\mathcal{E}_{\alpha}^{0}(u,v) = \mathcal{T}u \quad \text{ for all } u \in D(\mathcal{E}^0),
\end{equation}
which together with \eqref{equse} imply
\begin{equation} \label{den13}
\int_{\mathbb{R}^d} ( \alpha -L^{0})u \cdot  (\chi h-v) d \widehat{\mu} = 0  \quad \text{ for all } u \in D(L^{0})_b.
\end{equation}
Since $(L^{0}, D(L^{0}))$ is the generator of the symmetric Dirichlet form $(\mathcal{E}^0,D(\mathcal{E}^0))$, we have $ (\alpha-L^{0})(D(L^{0})_b) \subset L^2(\mathbb{R}^d, \widehat{\mu})$ densely. 
Therefore, \eqref{den13} implies $\chi h-v = 0$. In particular, $\chi h =v \in D(\mathcal{E}^0)$, so that $h\in \widehat{H}^{1,2}_{loc}(\mathbb{R}^d, \mu)$. Moreover, it follows from \eqref{rieszre} that
\begin{align} \label{intiden}
\mathcal{E}_{\alpha}^0(u, \chi h) = \int_{\mathbb{R}^d} \langle \bold{B}, \nabla (u\chi) \rangle h \, d \widehat{\mu}+ \int_{\mathbb{R}^d} \langle \widehat{A} \nabla \chi, \nabla u \rangle h\,d\widehat{\mu} + \int_{\mathbb{R}^d}  u L^0 \chi \cdot h \, d\widehat{\mu} \quad \text{ for all $u \in D(\mathcal{E}^0)$}.
\end{align}
Therefore, for each $u \in \widehat{H}^{1,2}_0(\mathbb{R}^d, \widehat{\mu})_0 \subset D(\mathcal{E}^0)$, choosing a function $\chi \in C_0^{\infty}(\mathbb{R}^d)$ with $\chi = 1$ on $\text{supp} (|u| dx)$,
we obtain \eqref{invinequ2-1}
from \eqref{intiden}, as desired.
\end{proof}

\text{}\\
\text{}\\
\centerline{}
Haesung Lee\\
Department of Mathematics and Big Data Science  \\
Kumoh National Institute of Technology \\
Gumi, Gyeongsangbuk-do 39177, South Korea \\
E-mail: fthslt@kumoh.ac.kr \\ \\ 
Gerald Trutnau\\
Department of Mathematical Sciences \\
and Research Institute of Mathematics\\
Seoul National University\\
1 Gwanak-ro, Gwanak-gu,
Seoul 08826, South Korea  \\
E-mail: trutnau@snu.ac.kr
\end{document}